\documentclass[11pt, oneside, article]{memoir}

\settrims{0pt}{0pt} 
\settypeblocksize{*}{36pc}{*} 
\setlrmargins{*}{*}{1} 
\setulmarginsandblock{1.1in}{1.4in}{*} 
\setheadfoot{\onelineskip}{2\onelineskip} 
\setheaderspaces{*}{1.5\onelineskip}{*} 
\checkandfixthelayout

\usepackage[centertags,sumlimits,intlimits,namelimits,reqno]{amsmath}
\usepackage{latexsym}
\usepackage{amsfonts,amsthm,amssymb}
\usepackage{mathtools}
\usepackage[cal=euler,scr=rsfso]{mathalfa}
\usepackage[boxslash]{stmaryrd}
\usepackage{newpxtext}
\usepackage{enumitem}
\usepackage{ifthen}
\usepackage[T1]{fontenc}
\usepackage[colorlinks,linkcolor=darkblue,citecolor=darkblue,urlcolor=darkblue,breaklinks=true]{hyperref}
\usepackage{tikz}
\usepackage[capitalize]{cleveref}
\usepackage[backend=biber, maxbibnames = 10, style = alphabetic,bibencoding=utf8]{biblatex}
\DeclareMathAlphabet{\mathpzc}{OT1}{pzc}{m}{it}
\usepackage[color=white]{todonotes}
\usepackage{xstring}
\usepackage{ebproof}
\usepackage[export]{adjustbox} 
\usepackage{graphicx}

\newcommand{\bigexists}{%
\mathop{\lower.9ex\hbox{%
   \scalebox{1.9}{\ensuremath{\exists}}}}\limits}

\DeclareFontFamily{U}{mathx}{\hyphenchar\font45}
\DeclareFontShape{U}{mathx}{m}{n}{
      <5> <6> <7> <8> <9> <10>
      <10.95> <12> <14.4> <17.28> <20.74> <24.88>
      mathx10
      }{}
\DeclareSymbolFont{mathx}{U}{mathx}{m}{n}
\DeclareFontSubstitution{U}{mathx}{m}{n}
\DeclareMathAccent{\widecheck}{0}{mathx}{"71}

	\allowdisplaybreaks
	
  \usetikzlibrary{ 
  	cd,
  	math,
  	decorations.markings,
		decorations.pathreplacing,
  	positioning,
	 	circuits.logic.US,
 		arrows.meta,
  	shapes,
		shadows,
		shadings,
  	calc,
  	fit,
  	quotes,
  	intersections,
  }

\newif\ifpgfshaperectangleroundnortheast
\newif\ifpgfshaperectangleroundnorthwest
\newif\ifpgfshaperectangleroundsoutheast
\newif\ifpgfshaperectangleroundsouthwest
\pgfkeys{/pgf/.cd,
  rectangle round north east/.is if=pgfshaperectangleroundnortheast,
  rectangle round north west/.is if=pgfshaperectangleroundnorthwest,
  rectangle round south east/.is if=pgfshaperectangleroundsoutheast,
  rectangle round south west/.is if=pgfshaperectangleroundsouthwest,
  rectangle round north east, rectangle round north west,
  rectangle round south east, rectangle round south west,
}
\makeatletter
\def\pgf@sh@bg@rectangle{%
  \pgfkeysgetvalue{/pgf/outer xsep}{\outerxsep}%
  \pgfkeysgetvalue{/pgf/outer ysep}{\outerysep}%
  \pgfpathmoveto{\pgfpointadd{\southwest}{\pgfpoint{\outerxsep}{\outerysep}}}%
  {\ifpgfshaperectangleroundnorthwest\else\pgfsetcornersarced{\pgfpointorigin}\fi%
    \pgfpathlineto{\pgfpointadd{\southwest\pgf@xa=\pgf@x\northeast\pgf@x=\pgf@xa}{\pgfpoint{\outerxsep}{-\outerysep}}}}%
  {\ifpgfshaperectangleroundnortheast\else\pgfsetcornersarced{\pgfpointorigin}\fi%
    \pgfpathlineto{\pgfpointadd{\northeast}{\pgfpoint{-\outerxsep}{-\outerysep}}}}%
  {\ifpgfshaperectangleroundsoutheast\else\pgfsetcornersarced{\pgfpointorigin}\fi%
    \pgfpathlineto{\pgfpointadd{\southwest\pgf@ya=\pgf@y\northeast\pgf@y=\pgf@ya}{\pgfpoint{-\outerxsep}{\outerysep}}}}%
  {\ifpgfshaperectangleroundsouthwest\else\pgfsetcornersarced{\pgfpointorigin}\fi%
    \pgfpathclose}}

\tikzset{
	oriented WD/.style={
		every to/.style={out=0,in=180,draw},
    label/.style={
    	font=\everymath\expandafter{\the\everymath\scriptstyle},
      inner sep=0pt,
      node distance=2pt and -2pt},
    semithick,
    node distance=1 and 1,
    decoration={markings, mark=at position \stringdecpos with \stringdec},
    ar/.style={postaction={decorate}},
    execute at begin picture={\tikzset{
    	x=\bbx, y=\bby,
      every fit/.style={inner xsep=\bbx, inner ysep=\bby}}}
    },
    string decoration/.store in=\stringdec,
    string decoration={\arrow{stealth};},
    string decoration pos/.store in=\stringdecpos,
    string decoration pos=.7,
    bbx/.store in=\bbx,
    bbx = 1.5cm,
    bby/.store in=\bby,
    bby = 1.5ex,
    bb port sep/.store in=\bbportsep,
    bb port sep=1.5,
    bb port length/.store in=\bbportlen,
    bb port length=4pt,
    bb penetrate/.store in=\bbpenetrate,
    bb penetrate=0,
    bb min width/.store in=\bbminwidth,
    bb min width=1cm,
    bb rounded corners/.store in=\bbcorners,
    bb rounded corners=2pt,
    bb spider/.style={
    	bb port sep=1, bb port length=10pt, bbx=.4cm, bb min width=.4cm, bby=.8ex},
    bb small/.style={
    	bb port sep=1, bb port length=2.5pt, bbx=.4cm, bb min width=.4cm, bby=.7ex},
		bb medium/.style={
			bb port sep=1, bb port length=2.5pt, bbx=.4cm, bb min width=.4cm, bby=.9ex},
    bb/.code 2 args={
    	\pgfmathsetlengthmacro{\bbheight}{\bbportsep * (max(#1,#2)+1) * \bby}
      \pgfkeysalso{draw,minimum height=\bbheight,minimum
       width=\bbminwidth,outer sep=0pt,
         rounded corners=\bbcorners,thick,
         prefix after command={\pgfextra{\let\fixname\tikzlastnode}},
         append after command={\pgfextra{\draw
            \ifnum #1=0{} \else foreach \i in {1,...,#1} {
            	($(\fixname.north west)!{\i/(#1+1)}!(\fixname.south west)$) +(-\bbportlen,0) coordinate (\fixname_in\i) -- +(\bbpenetrate,0) coordinate (\fixname_in\i')}\fi 
            \ifnum #2=0{} \else foreach \i in {1,...,#2} {
            	($(\fixname.north east)!{\i/(#2+1)}!(\fixname.south east)$) +(-
\bbpenetrate,0) coordinate (\fixname_out\i') -- +(\bbportlen,0) coordinate (\fixname_out\i)}\fi;
           }}}
		},
			bb name/.style={
     	append after command={
				\pgfextra{\node[anchor=north] at (\fixname.north) {#1};}
			}
		}
  }

\tikzset{
	unoriented WD/.style={
  	every to/.style={draw},
  	shorten <=-\penetration, shorten >=-\penetration,
  	label distance=-2pt,
  	thick,
  	node distance=\spacing,
  	execute at begin picture={\tikzset{
  		x=\spacing, y=\spacing, circuit logic US, tiny circuit symbols}
		}
  },
  pack size/.store in=\psize,
  pack size = 12pt,
	penetration/.store in=\penetration,
	penetration = 0pt,
  spacing/.store in=\spacing,
  spacing = 8pt,
  link size/.store in=\lsize,
  link size = 2pt,
  pack color/.store in=\pcolor,
  pack color = blue,
 	pack inside color/.store in=\picolor,
  pack inside color=blue!20,
 	pack outside color/.store in=\pocolor,
  pack outside color=blue!50!black,
 	surround sep/.store in=\ssep,
  surround sep=8pt,
 	link/.style={
  	circle, 
  	draw=black, 
  	fill=black,
  	inner sep=0pt, 
 		minimum size=\lsize
 	},
  pack/.style={
 		circle, 
 		draw = \pocolor, 
  		fill = \picolor,
  		inner sep = .25*\psize,
 		minimum size = \psize
  },
  func/.style={
  	pack,
		rectangle,
		rounded corners=.5*\psize,
		inner ysep=.125*\psize,
		minimum width=1.125*\psize,
		inner xsep=.25*\psize,
  },
  funcr/.style={
    func,
    rectangle round north west=false, 
		rectangle round south west=false,
  },
  funcl/.style={
    func,
		rectangle round north east=false, 
		rectangle round south east=false,
  },
  funcu/.style={
    func,
		rectangle round south east=false, 
		rectangle round south west=false,
  },
  funcd/.style={
    func,
		rectangle round north east=false, 
		rectangle round north west=false,
  },
  outer pack/.style={
 		ellipse, 
 		draw,
  	inner sep=\ssep,
  	color=gray,
 	},
  intermediate pack/.style={
 		ellipse,
 		dashed, 
  	draw,
  	inner sep=\ssep,
 		color=\pocolor,
 	},
}
  
\tikzset{
	spider diagram/.style={
		every to/.style={out=0, in=180, draw, thick},
		thick,
		execute at begin picture={\tikzset{
    	x=\leglen, y=\leglen/3}}
	},
	dot size/.store in=\dotsize,
	dot size = 5pt,
	dot fill/.store in=\dotfill,
	dot fill = black,
	leg length/.store in=\leglen,
	leg length = 15pt,
	baby/.style={dot size = 2pt, leg length = 6pt},
	young/.style={dot size = 3pt, leg length = 10pt},
	special spider/.code n args={4}{
		\pgfkeysalso{circle, draw, thick, inner sep=0, fill=\dotfill, minimum width=\dotsize,
  		prefix after command={\pgfextra{\let\fixname\tikzlastnode}},
  		append after command={\pgfextra{
  			\ifnum #1=0{} \else {\foreach \i in {1,...,#1} {
					\tikzmath{\anglei={-90*(#1+1-2*\i)/#1};}
  				\draw [thick]
						(\fixname) .. controls 
						($(\fixname.center)-(\anglei:#3/3)$) and ($(\fixname.center)-(\anglei:#3*2/3)$) .. 
						({$(\fixname.center)-(\anglei:#3*2/3)$}-|{$(\fixname.center)-(#3,0)$}) coordinate (\fixname_in\i);
  			}}\fi
  			\ifnum #2=0{} \else {\foreach \i in {1,...,#2} {
					\tikzmath{\anglei={90*(#2+1-2*\i)/#2};}
  				\draw [thick]
						(\fixname.center) .. controls 
						($(\fixname.center)+(\anglei:#4/3)$) and ($(\fixname.center)+(\anglei:#4*2/3)$) .. 
						({$(\fixname.center)+(\anglei:#4*2/3)$}-|{$(\fixname.center)+(#4,0)$}) coordinate (\fixname_out\i);
  			}}\fi
  		}}
		}
	},
	spider/.code 2 args={
		\pgfkeysalso{special spider={#1}{#2}{\leglen}{\leglen}}
	}
}

\tikzset{
	inner WD/.style={
		unoriented WD, 
		surround sep=2pt, 
		font=\tiny, 
		anchor=center
	}
}

\tikzset{
  function/.style={->, thin, shorten <=4pt, shorten >=4pt}
}

\tikzset{
  tick/.style={
  	postaction={
    	decorate,
      decoration={
      	markings, mark=at position 0.5 with {
					\draw[-] (0,.4ex) -- (0,-.4ex);
				}
			}
		}
	}
} 

\newcommand{\tickar}{\begin{tikzcd}[baseline=-0.5ex,cramped,sep=small,ampersand 
replacement=\&]{}\ar[r,tick]\&{}\end{tikzcd}}

\tikzcdset{
	twocell/.style={Rightarrow, shorten <=5pt, shorten >=5pt}
}
 
\newcommand{\simpletheta}{
\begin{tikzpicture}[unoriented WD, surround sep=2pt, pack size=6pt, font=\tiny, baseline=(theta.base)]
	\node[pack] (theta) {$\theta$};
	\draw (theta.west) -- +(-2pt, 0);
	\draw (theta.east) -- +(2pt, 0);
\end{tikzpicture}
}

\newcommand{\dectheta}{\begin{tikzpicture}[unoriented WD, surround sep=2pt, font=\tiny, pack size=6pt, baseline=(theta.base)]
	\node[pack] (theta) {$\theta$};
	\draw (theta.west) to[pos=1] node[left=-3pt] {$\Gamma_1$} +(-2pt, 0);
	\draw (theta.east) to[pos=1] node[right=-3pt] {$\Gamma_2$} +(2pt, 0);
\end{tikzpicture}}

  \setlist{noitemsep, nolistsep}
	\setlist[description]{leftmargin=0em, itemindent=2em}
	
\theoremstyle{plain}
\newtheorem{theorem}{Theorem}[chapter] 
\newtheorem{proposition}[theorem]{Proposition}
\newtheorem{corollary}[theorem]{Corollary}
\newtheorem{lemma}[theorem]{Lemma}

\theoremstyle{definition}
\newtheorem{definition}[theorem]{Definition}

\newtheorem{notation}[theorem]{Notation}

\newtheorem*{axiom*}{Axiom}

\theoremstyle{remark}
\newtheorem{example}[theorem]{Example}
\newtheorem{remark}[theorem]{Remark}
\newtheorem{warning}[theorem]{Warning}

\crefalias{chapter}{section}

\newcommand{\Set}[1]{\mathrm{#1}}
\newcommand{\ord}[1]{\underline{#1}}
\newcommand{\const}[1]{\mathtt{#1}}
\newcommand{\cat}[1]{\mathcal{#1}}
\newcommand{\ccat}[1]{\mathscr{#1}}
\newcommand{\Cat}[1]{{\mathsf{#1}}}
\newcommand{\CCat}[1]{\mathbb{\StrLeft{#1}{1}}\Cat{\StrGobbleLeft{#1}{1}}}
\newcommand{\funn}[1]{\mathrm{#1}}
\newcommand{\funr}[1]{\mathcal{#1}}
\newcommand{\ffunr}[1]{\mathbf{#1}}

\DeclareMathOperator{\ob}{\Set{Ob}}

\DeclareMathOperator{\im}{im}
\DeclareMathOperator{\inc}{inc}
\DeclarePairedDelimiter{\pair}{\langle}{\rangle}
\DeclarePairedDelimiter{\unary}{{\langle\,}}{{\,\rangle}}
\DeclarePairedDelimiter{\copair}{[}{]}
\DeclarePairedDelimiter{\classify}{{\raisebox{1pt}{$\ulcorner$}}}{{\raisebox{1pt}{$\urcorner$}}}
\DeclarePairedDelimiter{\abs}{\lvert}{\rvert}
\DeclarePairedDelimiter{\church}{\llbracket}{\rrbracket}

  \definecolor{darkblue}{rgb}{0,0,0.7} 
  \newcommand{\define}[1]{\emph{#1}}

\newcommand{\rrel}[1]{\CCat{Rel}_{#1}}

\newcommand{\rgcat}{\Cat{RgCat}}
\newcommand{\rgcalc}{\Cat{RgCalc}}
\newcommand{\rrgcat}[1][]{\CCat{RgCat}}
\newcommand{\rgpocat}{\Cat{RgPocat}}

\newcommand{\sub}{\Cat{Sub}}

\newcommand{\syn}{\ffunr{syn}}
\newcommand{\prd}{\ffunr{prd}}

\newcommand{\lsh}[2][T]{{#2}_!}
\newcommand{\ust}[2][T]{#2^*}

\newcommand{\tn}[1]{\textnormal{#1}}

\newcommand{\id}{\funn{id}}
\newcommand{\finset}{\Cat{FinSet}}
\newcommand{\smset}{\Cat{Set}}
\newcommand{\cospan}{\Cat{Cospan}}

\newcommand{\pposet}{\CCat{Poset}}

\newcommand{\ladj}{\Cat{LAdj}}
\newcommand{\radj}{\Cat{RAdj}}
\newcommand{\rradj}{\CCat{RAdj}}

\newcommand{\smcat}{\Cat{Cat}}

\newcommand{\ppocat}{\CCat{Pocat}}
\newcommand{\rela}[1]{\CCat{IntRel}_{#1}}
\newcommand{\func}[1]{\cat{R}_{#1}}

\newcommand{\str}{\const{str}}
\newcommand{\pr}{P}

\newcommand{\true}{\const{true}}

\newcommand{\ol}[1]{\overline{#1}}
\newcommand{\comp}{\mathtt{comp}}

\newcommand{\typeset}{\mathrm{T}}

\newcommand{\frc}[1][\typeset]{
  \ifthenelse{\equal{#1}{blank}}{\Cat{FRg}}{\Cat{FRg}(#1)}
}
\newcommand{\frb}[1][\typeset]{\CCat{FRg}(#1)}

\newcommand{\out}{\mathrm{out}}
\newcommand{\cocolon}{:\!}

\newcommand{\too}{\longrightarrow}
\newcommand{\tto}{\rightrightarrows}
\newcommand{\To}[1]{\xrightarrow{#1}}

\newcommand{\from}{\leftarrow}
\newcommand{\From}[1]{\xleftarrow{#1}}
\newcommand{\tofrom}{\leftrightarrows}
\newcommand{\surj}{\twoheadrightarrow}
\newcommand{\inj}{\rightarrowtail}

\renewcommand{\ss}{\subseteq}
\newcommand{\imp}{\Rightarrow}
\renewcommand{\iff}{\Leftrightarrow}

\newcommand{\op}{^\mathrm{op}}

\newcommand{\inv}{^{-1}}
\newcommand{\tp}{^\dagger}

\newcommand{\slice}[1]{_{/#1}}

\newcommand{\terminal}{\star}

\newcommand{\ppf}{\mathbb{P}_f}

\newcommand{\cp}{\mathbin{\fatsemi}}
\newcommand{\tens}{\oplus}
\newcommand{\unit}{0}

\newcommand{\powset}{\cat{P}}
\newcommand{\powfin}{\powset_{f}}
\newcommand{\supp}{\Set{Supp}}
\newcommand{\vars}{\Set{Vars}}
\newcommand{\type}{\Set{Tp}}

\newcommand{\Prod}[1][-]{\classify{#1}}

\newcommand{\nn}{\mathbb{N}}

\newcommand{\ajax}{\Cat{Ajax}}
\newcommand{\aajax}{\CCat{Ajax}}
\newcommand{\lax}{\Cat{Lax}}
\newcommand{\aadjmon}{\CCat{AdjMon}}
\newcommand{\msl}{\wedge\tn{-}\Cat{SL}}

\newcommand{\qqand}{\qquad\text{and}\qquad}
\newcommand{\qand}{\quad\text{and}\quad}

\newcommand{\exclude}[1]{}

\newcommand{\step}[1]{\goodbreak\bigskip\begin{center}------\quad\textbf{#1}\quad------\end{center}\vspace{-.1in}\nopagebreak}
\newcommand{\commentout}[1]{}

\newcommand{\adj}[5][30pt]{
\begin{tikzcd}[ampersand replacement=\&, column sep=#1]
  #2\ar[r, shift left=5pt, "{#3}"]\ar[r, phantom, "\Rightarrow" yshift=-.6pt]\&
  #5\ar[l, shift left=5pt, "{#4}"]
\end{tikzcd}
}

\newcommand{\adjr}[5][30pt]{
\begin{tikzcd}[ampersand replacement=\&, column sep=#1]
  #2\ar[r, shift left=5pt, "{#3}"]\ar[r, phantom, "\Leftarrow" yshift=-.6pt]\&
  #5\ar[l, shift left=5pt, "{#4}"]
\end{tikzcd}
}

\newcommand{\pb}[1][very near start]{\ar[dr, phantom, #1, "\lrcorner"]}

\begin{document}   

  \title{Graphical Regular Logic}
  \author{Brendan Fong and David I.\ Spivak\thanks{Spivak and Fong acknowledge support from AFOSR grants FA9550-14-1-0031 and FA9550-17-1-0058.}}
  \date{\vspace{-.3in}}

\maketitle

\begin{abstract}
  Regular logic can be regarded as the \emph{internal language} of regular
  categories, but the logic itself is generally not given a categorical
  treatment. In this paper, we understand the syntax and proof rules of regular
  logic in terms of the free regular category $\frc$ on a set $\typeset$. From this point of view,
  regular theories are certain monoidal 2-functors from a suitable 2-category of
  contexts---the 2-category of relations in $\frc$---to that of posets. Such functors assign to each context the
  set of formulas in that context, ordered by entailment. We refer to such a
  2-functor as a \emph{regular calculus} because it naturally gives rise to a
  graphical string diagram calculus in the spirit of Joyal and Street. Our key
  aim to prove that the category of regular categories is essentially reflective
  in that of regular calculi. Along the way, we demonstrate how to use this
  graphical calculus.
  
  Keywords: regular logic, category theory, primitive positive formula
\end{abstract}

\chapter{Introduction}\label{chap.intro}

Regular logic is the fragment of first order logic generated by equality ($=$),
true ($\true$), conjunction ($\wedge$), and existential quantification ($\exists$). A
defining feature of this fragment is that it is expressive enough to define
\emph{functions} and \emph{composition} of functions, or more generally of
relations: given relations $R \subseteq X \times Y$ and $S \subseteq Y \times
Z$, their composite is given by the formula
\[
  R \cp S = \{(x,z) \mid \exists y.R(x,y)\wedge S(y,z)\}.
\]
Indeed, regular logic is the internal language of regular categories, which may
in turn be understood as a categorical characterization of the minimal structure
needed to have a well-behaved notion of relation.

While regular categories put emphasis on the notion of \emph{binary} relation, the
existence of finite products allows them to handle $n$-ary relations---that is,
subobjects of $n$-fold products---and their composition. To organize more complicated multi-way composites of relations, many fields have
developed some notion of wiring diagram. A good amount of recent work, including but not
limited to control theory
\cite{bonchi2014categorical,baez2015categories,fong2016categorical}, database
theory and knowledge representation
\cite{bonchi2018graphical,patterson2017knowledge}, electrical engineering
\cite{baez2018compositional}, and chemistry \cite{baez2017compositional},
all serve to demonstrate the link between these languages and categories for which the morphisms
are relations. 

A first goal of this paper is to clarify the link between regular logic and
these various graphical languages. In doing so, we provide a new diagrammatic syntax for regular
logic, the titular \emph{graphical regular logic}. Rather than pursue a direct
translation with the classical syntax for first order logic, we demonstrate a tight connection between graphical regular logic and
the notion of \emph{regular category}. A second goal, then, is to repackage the structure of a
regular category into terms that cleanly reflect its underlying logical theory. We call
the resulting categorical structure a \emph{regular calculus}. Regular calculi are based on free regular categories, so let's begin there. 

We will show that the \emph{free regular category} $\frc[blank]$ on a singleton set can
be obtained by freely adding a fresh terminal object to $\finset\op$. Here is a
depiction of a few objects in $\frc[blank]$:
\begin{equation}\label{eqn.free_reg}
\begin{tikzcd}
	0&
	s\ar[l, >->]&
	1\ar[l, ->>]\ar[r]&
	2\ar[l, shift left=5pt]\ar[l, shift right=5pt]&
	\cdots
\end{tikzcd}
\end{equation}
The object $s$ is the coequalizer of the two distinct maps $2\tto 1$, so in a
sense it prevents the unique map $1\to 0$ from being a regular epimorphism. Thus
one may think of $s$ as representing the \emph{support} of an abstract object in
a regular category. In $\smset$, the support of any object is either empty or singleton, but in general the concept is more refined. For example, the topos of sheaves on a space $X$ is regular, and the support of a sheaf $r$ is the union $U\ss X$ of all open sets on which $r(U)$ is nonempty.

For any small set $\typeset$ of \emph{types} (also known as \emph{sorts}), the free regular category on $\typeset$ is then the $\typeset$-fold coproduct of regular categories $\frc\coloneqq\bigsqcup_\typeset\frc[blank]$. That is, we have an adjunction
\begin{equation}\label{eqn.fr_ob_adj}
\adj[40pt]{\smset}{\frc[blank]}{\ob}{\rgcat}
\end{equation}
which we will construct explicitly in \cref{thm.fr_is_free}. For any
regular category $\cat{R}$, the counit provides a canonical regular functor, which
we denote $\Prod\colon\frc[\ob\cat{R}]\to\cat{R}$. Note also that this extends to a
2-functor $\Prod\colon \rrel{\frc[\ob\cat{R}]}\to\rrel{\cat{R}}$ between the
associated relation bicategories.

Write $\frb\coloneqq \rrel{\frc}$ for this bicategory of relations. Just as $\frc[blank]$ is closely related to the opposite
of the category of finite sets (see \eqref{eqn.free_reg}), the objects in $\frb$ are, at a first
approximation, much like finite sets $\ord{n}$ equipped with a function $\ord{n}
\to \typeset$, and morphisms are much like corelations: equivalence relations
on some coproduct $\ord{n+n'}$. We draw objects and morphisms as on the left and
right below:
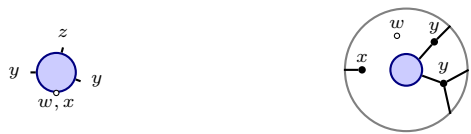
\[
  \begin{aligned}
  \begin{tikzpicture}[inner WD]
    \node[pack, minimum size = 3ex] (rho) {};
    \draw (rho.180) to[pos=1] node[left] (w) {$y$} +(180:2pt);
    \draw (rho.75) to[pos=1] node[above] (n) {$z$} +(75:2pt);
    \draw (rho.-20) to[pos=1] node[right] (e) {$y$} +(-20:2pt); 
    \node[link,fill=white,thin] at (rho.270) {};
    \node[below=-.2 of rho.270] (s) {$w,x$};
  \end{tikzpicture}
  \end{aligned}
  \hspace{3cm}
  \begin{aligned}
  \begin{tikzpicture}[inner WD]
    \node[pack] (a) {};
    \node[outer pack, inner sep=10pt, fit=(a)] (outer) {};
    \node[link] (link1) at ($(a.west)!.6!(outer.west)$) {};
    \node[link] (link2) at ($(a.45)!.5!(outer.45)$) {};
    \node[link] (link3) at ($(a.-20)!.5!(outer.-20)$) {};
    \node[link,fill=white,thin] at ($(a.100) + (110:7pt)$) (dot) {};
    \node[above=-.3 of link1] {$x$};
    \node[above=-.3 of link2] {$y$};
    \node[above=-.2 of link3] {$y$};
    \node[above=-.3 of dot] {$w$};
    \draw (outer.west) -- (link1);
    \draw (a.40) -- (link2);
    \draw (link2) -- (outer.45);
    \draw (a.-20) -- (link3);
    \draw (link3) -- (outer.0);
    \draw (link3) -- (outer.-45);
  \end{tikzpicture}	
\end{aligned}
\]
The left-hand circle, equipped with its labeled ports and white dot, represents an object in $\frc$; we call this picture a \emph{shell}. Here each port represents
an element of the associated finite set $\ord{3}$, the white dot captures aspects related to the
support object $s$ of $\frc[blank]$, and the labels $x,y$ etc.\ are
elements of $\typeset$. In the right-hand morphism, the inner shell
represents the domain, outer shell represents the codomain, and the things between them---the connected components of
the wires and the white dots---represent the equivalence classes of the aforementioned equivalence relation. 

A regular calculus lets us think of each object $\Gamma\in\frb$---each
shell---as a context for formulas in some regular theory, and of each morphism,
i.e.\ each wiring diagram $\Gamma\tickar\Gamma'$, as a method for converting
$\Gamma$-formulas to $\Gamma'$-formulas, using $=$, $\true$,$\wedge$, and $\exists$. We next
want to think about how regular categories fit into this picture.

If $\cat{R}$ is a regular category, formulas in the associated regular theory are given by relations $x\ss r_1\times\cdots\times r_n$, where $x$ and the $r_i$ are objects in $\cat{R}$, i.e.\ $r_\bullet\colon\ord{n}\to\cat{R}$. Thus we could consider $\Gamma\coloneqq r_\bullet$ as a context, and this brings us back to the free regular category $\frb[\ob\cat{R}]$. The counit functor $\Prod\colon\frc[\ob\cat{R}]\to\cat{R}$ sends $\Gamma$ to $\Prod[\Gamma]\coloneqq r_1\times\cdots\times r_n$. A key feature of regular categories is that the subobjects $\sub_{\cat{R}}(r_1\times\cdots\times r_n)$ form a meet-semilattice, elements of which we call \emph{predicates} in context $\Gamma$. As we shall see, the collection of all
these semilattices, when related by the structure of $\frb[\ob\cat{R}]$,
includes enough data to recover the regular category $\cat{R}$ itself.

Indeed, consider the commutative diagram
\[
\begin{tikzcd}
	\frc[\ob\cat{R}]\ar[r, "\Prod"]\ar[d]&
	\cat{R}\ar[d]\\
	\frb[\ob\ccat{R}]\ar[r, "\Prod"']&
	\ccat{R}\ar[r, "{\ccat{R}(I,-)}"']&[20pt]
	\pposet
\end{tikzcd}
\]
where the vertical maps represent inclusions of a regular 1-category into its bicategory of relations, and the hom-2-functor $\ccat{R}(I,-)$ sends each object $r\in\ob\cat{R}=\ob\ccat{R}$ to the subobject lattice $\sub_{\cat{R}}(r)=\ccat{R}(I,r)$. We can denote the composite of the bottom maps as
\begin{equation}\label{eqn.construct_rels_on_obs}
	\sub_{\cat{R}}\Prod\colon\frb[\ob\cat{R}]\too\pposet.
\end{equation}
The domain $\frb[\ob\cat{R}]$ is a category of contexts and the functor $\sub_{\cat{R}}\Prod[\Gamma]$ assigns the poset of predicates to each context $\Gamma$.

As mentioned, we will show how to reconstruct $\cat{R}$---up to
equivalence---from the contexts $\Gamma\in\frc[\ob\cat{R}]$ and their predicate posets $\sub_\cat{R}\Prod[\Gamma]$ as in
\cref{eqn.construct_rels_on_obs}, once we give the abstract structure by which they hang
together. The question is, given any set $\typeset$, what extra structure do we need on a functor 
\[\pr\colon\frb\too\pposet\] 
in order to construct a regular category from it?

Whatever the required structure on $\pr$ is, of course $\sub_{\cat{R}}\Prod$ needs to have that structure. First of all, $\sub_{\cat{R}}\Prod$ is a 2-functor, and it happens to be the composite of $\rrel{\ulcorner-\urcorner}$ and $\sub_\cat{R}$. It is not hard to check that the 2-functor $\Prod$ is strong monoidal, whereas the 2-functor $\ccat{R}(I,-)$ is only lax monoidal: given objects $r_1,r_2\in\cat{R}$ the induced monotone map $\times\colon\sub_{\cat{R}}(r_1)\times\sub_{\cat{R}}(r_2)\to\sub_{\cat{R}}(r_1\times r_2)$ is not an isomorphism. However, $\sub_{\cat{R}}\Prod$ has a bit more structure than merely being a lax functor: each laxator has a left adjoint
\[
\adjr{1}{\true}{!}{\sub_{\cat{R}}(1)}
\qquad
\adjr[50pt]{\sub_{\cat{R}}(r_1)\times\sub_{\cat{R}}(r_2)}{\times}{\pair{\im_1,\  \im_2}}{\sub_{\cat{R}}(r_1\times r_2).}
\]

Abstractly, if $\ccat{R}$ and $\ccat{P}$ are monoidal 2-categories, we say that
a lax monoidal functor $\ccat{R} \to \ccat{P}$ is \emph{ajax} (``adjoint-lax'')
if its laxators $\rho$ and $\rho_{v,v'}$ are right adjoints in $\ccat{P}$.
Thus we have seen that $\sub_{\cat{R}}\Prod\colon\frb[\ob\cat{R}]\too\pposet$ is
ajax. This is precisely the structure required to reconstruct a regular
category. 

Ajax functors have the important property that they preserve adjoint monoids, a notion we introduce.
An \emph{adjoint monoid} is an object with both monoid and
comonoid structures, such that the monoid maps are right adjoint to their
corresponding comonoid maps. In particular, we will see that each object in
$\frb$ has a canonical adjoint monoid structure, and that adjoint monoids in
$\pposet$ are exactly meet-semilattices. This guarantees that ajax functors $\frb
\to \pposet$ send objects in $\frb$---contexts---to meet-semilattices.

We now come to our main definition.

\begin{definition}\label{def.reg_sketch}
A \emph{regular calculus} is a pair $(\typeset, \pr)$ where $\typeset$ is a set and  $\pr\colon\frb\to\pposet$ is an ajax 2-functor. 

A \emph{morphism} $(\typeset, \pr)\to(\typeset',\pr')$ of regular calculi is a pair $(F,F^\sharp)$ where $F\colon\typeset\to\typeset'$ is a function and $F^\sharp$ is a monoidal natural transformation
\[
\begin{tikzcd}[row sep=0pt]
	\typeset\ar[dd, "F"']&
	\frb
		\ar[dd, "{\frb[F]}"']\ar[dr, bend left=15pt, "\pr", ""' name=T]\\
	&&\pposet\\
	\typeset'&
	\frb[\typeset']\ar[ur, bend right=15pt, "\pr'"', "" name=T']
	\ar[from=T, to=T'-|T, twocell, "F^\sharp"']
\end{tikzcd}
\]
that is strict in every respect: all the required coherence diagrams of posets commute on the nose. We denote the category of regular calculi by $\rgcalc$.
\end{definition}
%
%
%
%

The goal of this paper is to prove that $\rgcat$ is \emph{essentially reflective
in $\rgcalc$} (see \cref{thm.main}). More precisely, this means:

\begin{theorem}
The ``predicates'' mapping in \cref{eqn.construct_rels_on_obs} extends to a fully faithful functor
\begin{equation}\label{eqn.rels}
\begin{aligned}
	\prd\colon\rgcat\to&\rgcalc\\
	\cat{R}\mapsto&\left(\ob\cat{R},\sub_{\cat{R}}\Prod\right),
\end{aligned}
\end{equation}
and this functor has a left adjoint, the ``syntactic category,''
\[\adj{\rgcalc}{\syn}{\prd}{\rgcat.}\]
Moreover, for any regular category $\cat{R}$, the counit functor $\syn(\prd(\cat{R}))\to\cat{R}$ is an equivalence.
\end{theorem}

In order to prove this result, we will also show that each object $(\typeset, \pr)\in\rgcalc$ can be understood as a graphical language for a theory in regular logic. Indeed, the usual syntactic category for that theory will be the regular category $\syn(\typeset, \pr)$. 

\subsection{Related work}

Regular categories were first defined by Barr \cite{barr:1971a}, as a way to
elucidate the structure present in abelian categories. Shortly thereafter, Freyd
and Scedrov were the first to make the connection to regular logic. Similarly to
the present work, they focused on the structure of the bicategory
of relations, seeking an axiomatization through the notion of an allegory, a
poset-enriched category (a \emph{po-category}) with an identity-on-objects
involution, such that every hom-poset is a meet-semilattice, and such that the
modular law holds \cite{freyd1990categories}. 

Carboni and Walters also sought to axiomatize these objects, defining functionally complete cartesian bicategories of relations
\cite{Carboni:1987a}. A cartesian bicategory is a monoidal po-category in which
every object is equipped with an adjoint monoid in a coherent way. Functionally
complete bicategories of relations further require that these monoids and
comonoids obey the Frobenius law, and that a sensible notion of image
factorization exists.

Both allegories and bicategories of relations take the structure of a regular
category, and decompress it into a (locally posetal) 2-categorical expression.
While regular calculi have similar features to both allegories and cartesian
bicategories, such as emphasizing that the hom-posets are meet-semilattices or
that there are adjoint monoid structures on each object, they represent this
data in terms of a \emph{functor} rather than a category. 

In the world of databases, regular formulas correspond to conjunctive queries, and entailment corresponds to query containment. A well-known theorem of Chandra
and Merlin states that (conjunctive) query containment is decidable; their
proof translates logical expressions into graphical representations
\cite{chandra1977optimal}. In more recent work, Bonchi, Seeber, and Soboci\'nski
show that the Chandra--Merlin approach permits an elegant
formalization in terms of the Carboni--Walters axioms for bicategories of
relations \cite{bonchi2018graphical}. Patterson has also considered bicategories of
relations, and their Joyal-Street string calculus \cite{joyal1991geometry}, as a graphical way of
capturing the regular logical aspects knowledge representation
\cite{patterson2017knowledge}.

Presenting regular categories using monoidal maps $\frb\to\pposet$ fits into an emerging pattern. In \cite{Spivak.Schultz.Rupel:2016a} it was shown that lax monoidal functors $1\text{--}\Cat{Cob}_\typeset\to\smset$ present traced monoidal categories, and in \cite{fong2019hypergraph} it was shown that lax monoidal functors $\cospan_\typeset\to\smset$ present hypergraph categories. But now in all three cases, the domain of the functor represents a particular language of string diagrams, and the codomain represents a choice of enriching category. The present paper can be seen as an extension of that work, showing that regular categories are something like poset-enriched hypergraph categories.

\subsection{Outline}

We begin in \cref{chap.regular_categories} with a section reviewing the
definition and basic properties of regular categories $\cat{R}$, emphasizing in
particular the construction of the symmetric monoidal po-category
$\rrel{\cat{R}}$ of relations in $\cat{R}$. In fact, we will say that a
po-category $\ccat{R}$ is a \emph{regular po-category} if it is isomorphic to
the relations po-category of some regular category
$\ccat{R}\cong\rrel{\cat{R}}$.

In \cref{chap.adjoint_monoids} we introduce the notion
of adjoint monoid. We show the category of adjoint monoids in a po-category
$\ccat{C}$ is given by the category of ajax monoidal functors $1 \to \ccat{C}$,
that adjoint monoids in $\pposet$ are meet-semilattices, that every object in a
relations po-category has a canonical adjoint monoid structure, and that the
subobject functor of a regular po-category is ajax.

In \cref{chap.free_reg} we turn our attention to free regular categories and
free regular po-categories on a set. In particular, we give an explicit
construction of the free regular category on a set $\typeset$ as the opposite of
the comma category $\finset\downarrow\powfin(\typeset)$; the free regular
po-category on $\typeset$ is its relations po-category. At this point we can
give our main definition: a regular calculus is an ajax functor from a free
regular po-category to that of posets. We then give a fully faithful functor
$\prd\colon\rgcat\to\rgcalc$, from regular categories to regular calculi.

In \cref{chap.graphical_reglog}, we introduce graphical regular
logic. First, we give an explicit, graphical description of the objects,
morphisms, and order in a free regular po-category. We then define the
\emph{graphical terms} of a regular calculus. Given a regular calculus $P\colon
\frb \to \pposet$, a graphical term is a morphism $\omega\colon \Gamma_1 \times
\dots \times \Gamma_k \tickar \Gamma_\out$ in $\frb$ together with elements
$\theta_i \in P(\Gamma_i)$ for each $i = 1,\dots, k$. We give rules for
composing and reasoning with these. Having set up our language, we now proceed towards the construction of a regular
category from a regular calculus.

In \cref{chap.relations}, we define the
po-category of internal relations of an regular calculus. This construction is
a relational version of the standard \emph{syntactic category} constructions: an object
is a context--predicate pair $(\Gamma,\varphi)$, where $\Gamma$ is an object of
$\frb$ and $\varphi \in P(\Gamma)$, and a morphism $(\Gamma,\varphi) \to
(\Gamma',\varphi')$ is a predicate $\theta$ in the joint context $\Gamma \times
\Gamma'$ that entails $\varphi$ and $\varphi'$.

In \cref{chap.functions}, we show that the category of left adjoints in the
po-category of internal relations, which we call the category of internal functions,
is a regular category. We explicitly construct limits and image factorizations
using graphical regular logic.

Finally, in \cref{chap.ess_refl}, we construct the functor $\syn\colon \rgcalc
\to \rgcat$ adjoint to $\prd$, and show that the two form an
essential reflection.

\subsection{Notation and 2-categorical background}\label{page.notation}

Let us fix some notation. Most is standard, but we highlight in particular
our use of $\cp$ for composition, of the term \emph{po-category} for locally
posetal 2-category, and of an arrow $\Rightarrow$ pointing in the direction of
the left adjoint to signify an adjunction.
\smallskip

\begin{itemize}

 \item We typically denote composition in diagrammatic order, so the composite of $f\colon A\to B$ and $g\colon B\to C$ is $f\cp g\colon A\to C$. We often denote the identity morphism $\id_c\colon c\to c$ on an object $c\in\cat{C}$ simply by the name of the object, $c$. Thus if $f\colon c\to d$, we have $(c\cp f)=f=(f\cp d)$.

  \item We may denote the terminal object of any category by $\terminal$, and
    the associated map from an object $c$ as $!\colon c \to \terminal$, but we denote the top element of any poset $P$ by $\true\in P$.

  \item We denote the universal map into a product by $\pair{f,g}$ and the universal map out of a coproduct by $\copair{f,g}$.

  \item Given a natural number $n\in\nn$, define $\ord{n}\coloneqq\{1,2,\ldots,n\}\in\finset$; in particular $\ord{0}=\varnothing$.

	\item Given a lax monoidal functor $F\colon\cat{C}\to\cat{D}$, we denote the laxators by $\rho\colon I\to F(I)$ and $\rho_{c,c'}\colon F(c)\otimes F(c')\to F(c\otimes c')$ for any $c,c'\in\cat{C}$. We use the same notation for longer lists, e.g.\ we write $\rho_{c,c',c''}$ for the canonical map $F(c)\otimes F(c')\otimes F(c'')\to F(c\otimes c'\otimes c'')$.

\end{itemize}

\paragraph{Symmetric monoidal po-categories.}
We use the term \define{po-category} to mean locally posetal 2-category, i.e.\ a
category enriched in partially ordered sets (posets). \define{Po-functors} are,
of course, poset-enriched functors (functors that preserve the local
order). The set of po-functors $\ccat{C}\to\ccat{D}$ itself has a natural order, where
$F\leq G$ iff $F(c)\leq G(c)$ for all $c\in\ccat{C}$. We define $\ppocat$ to be
the po-category of po-categories and po-functors. 

We use $\CCat{Xyz}$---with first character made blackboard bold---to denote
named po-categories and $\Cat{Xyz}$ for named 1-categories. We rely fairly
heavily on this; for example our notations for the free regular category and the
free regular po-category on a set $\typeset$ differ only in this way: $\frc$ vs.\
$\frb$.

A po-category is, in particular, a (strict) 2-category, and po-functors are
(strict) 2-functors. As such there is a forgetful functor $\ppocat\to\smcat$
sending each po-category and po-functor to its underlying 1-category and
1-functor. A \define{symmetric monoidal po-category} is a po-category $\ccat{C}$ 
together with po-functors $\otimes\colon \ccat{C} \times \ccat{C} \to \ccat{C}$ 
and $I\colon \terminal \to \ccat{C}$ whose underlying 1-structures form a
symmetric monoidal category.

The symmetric monoidal po-category $\pposet$ has posets $P$ as objects, monotone
maps $f\colon P \to Q$ as morphisms, and order given by $f \leq g$ iff $f(p)
\leq g(p)$ for all $p$.  Its monoidal structure is given by cartesian product
$P\times Q$, with the terminal poset $1$ the monoidal unit.

\paragraph{Adjunctions in a 2-category.}
Recall that, given a 2-category $\ccat{C}$, an \emph{adjunction in $\ccat{C}$}
consists of a pair of objects $c,d\in\ob\ccat{C}$, a pair of morphisms
$L\colon c\to d$ and $R\colon d\to c$, and a pair of 2-morphisms $\eta\colon
d\imp (L\cp R)$ and $\epsilon\colon (R\cp L)\imp c$ such that a pair of equations hold:
\[
  \id_L =
  \begin{aligned}
    \begin{tikzcd}[row sep=7pt, column sep=large]
      c\ar[dr, "L"]\ar[dd, equal]\\
      \ar[r, phantom, pos=.3, "\overset{\eta}{\Longrightarrow}" above=-6pt]&
      d\ar[dl, "R" description]\ar[dd, equal]\\
      c\ar[dr, "L"']\ar[r, phantom, pos=.7, "\underset{\epsilon}{\Longrightarrow}" below=-5pt]&~\\&
      d
    \end{tikzcd}
  \end{aligned}
  \qquad\qqand\qquad
  \id_R =
  \begin{aligned}
    \begin{tikzcd}[row sep=7pt, column sep=large]
      &
      d\ar[dl, "R"']\ar[dd, equal]\\
      c\ar[dd, equal]\ar[dr, "L" description]\ar[r, phantom, pos=.7, "\overset{\epsilon}{\Longrightarrow}" above=-6pt]&
      ~\\
      \ar[r, phantom, pos=.3, "\underset{\eta}{\Longrightarrow}" below=-5pt]&
      d\ar[dl, "R"]\\
      c
    \end{tikzcd}
  \end{aligned}
\]
Noting that both $\eta$ and $\epsilon$ always point in the direction of the left
adjoint $L$, we write
\[\adj{c}{L}{R}{d}\] 
to denote an adjunction, where the 2-arrow points in the direction of the left
adjoint. We sometimes write $L\dashv R$ inline, but are careful to avoid the $\vdash$ symbol in this context; \emph{the symbol $\vdash$ always means entailment}. We denote the category with the same objects and with left adjoints as
morphisms as $\ladj(\ccat{C})$.

\goodbreak
\chapter{Background on regular categories} \label{chap.regular_categories}

Regular categories are, roughly speaking, categories that have a good notion of relations. Relations, which we sometimes call \emph{predicates}, are subobjects of products, and composites of relations are formed using pullbacks and image factorizations; regular categories are categories that have suitably interoperable finite limits and image factorizations. We now proceed to make this precise.

\section{Definition of regular categories and functors}

Regular categories were first defined by Barr \cite{barr:1971a} to isolate important aspects of abelian categories. The reader who is unacquainted with regular categories and/or regular logic may see \cite{butz1998regular}.

\begin{definition}[Barr]
  A \define{regular category} is a category $\cat{R}$ with the following properties:
  \begin{enumerate}
	\item it has all finite limits;
	\item the kernel pair of any morphism $f\colon r\to s$ admits a coequalizer $r\times_s r\tto r\to\Set{coeq}(f)$, which we denote $\im(f)\coloneqq\Set{coeq}(f)$ and call the \emph{image} of $f$; and 
	\item the pullback---along any map---of a regular epimorphism (a coequalizer of any parallel pair) is again a regular epimorphism.
\end{enumerate}

  A \define{regular functor} is a functor
  between regular categories that preserves finite limits and regular epis. We write $\rgcat$ for the category of regular categories.
\end{definition}

\begin{lemma}\label{lemma.OFS}
For any $f\colon r\to r'$, the universal map $\im(f)\to r'$ is monic. Thus every map $f$ can be factored into a regular epimorphism followed by a monomorphism: $r\surj \im(f)\inj r'$, and this constitutes an orthogonal factorization system. In particular, image  factorization is unique up to isomorphism.
\end{lemma}
\begin{proof}
This is \cite[Proposition 2.4]{butz1998regular}.
\end{proof}

\begin{definition} \label{def.support}
  The \define{support} of an object $r$ in a regular category is the image $r
  \surj \supp(r) \inj\terminal$ of its unique map to the terminal object.
\end{definition}

\begin{definition}
  A \define{subobject} of an object $r$ in a category is an isomorphism class of
  monomorphisms $r' \inj r$, where morphisms between monomorphisms are as in the
  slice category over $r$. This defines a partially ordered set $\sub(r)$. We write $r'\ss r$ to denote the equivalence class represented by $r'\inj r$.
\end{definition}

\begin{proposition}\label{prop.sub_ladj}
Any morphism $f\colon r\to s$ in a regular category $\cat{R}$ induces an adjunction
\begin{equation}\label{eqn.subobject_adj}
\adj{\sub(r)}{\lsh{f}}{\ust{f}}{\sub(s).}
\end{equation}
This extends to a functor $\sub\colon\cat{R}\to\ladj(\pposet)$.
\end{proposition}
\begin{proof}
Given a subobject $r'\ss r$ or $s'\ss s$, define $\lsh{f}(r')\ss s$ and $\ust{f}(s')\ss r$ as follows:
\[
\begin{tikzcd}
	r'\ar[r, >->]\ar[d, dashed, ->>]&
	r\ar[d, "f"]\\
	\lsh{f}(r')\ar[r, dashed, >->]&
	s
\end{tikzcd}
\hspace{1in}
\begin{tikzcd}
	\ust{f}(s')\ar[r, dashed, >->]\ar[d, dashed]\pb&
	r\ar[d, "f"]\\
	s'\ar[r, >->]&
	s
\end{tikzcd}
\]
The fact that these are adjoint follows from the orthogonality of the factorization system in \cref{lemma.OFS}, and the constructions are functorial.
\end{proof}

The following proposition discusses some well-known properties of subobjects in a regular category. In \cref{rem.tangible} we explain how these properties are 1-categorical reflections of a more elementary 2-categorical story.

\begin{proposition}\label{prop.tangible}
Let $\cat{R}$ be a regular category. The functor $\sub\colon\cat{R}\to\ladj(\pposet)$ satisfies the following:
\begin{enumerate}
	\item $\sub(r)$ is a meet-semilattice for each $r\in\cat{R}$,
	\item for each cospan $f\colon r'\to r\from s\cocolon g$, the Beck-Chevalley condition (right)
	holds for the pullback square (left):
	\[
	\begin{tikzcd}
		r'\times_r s\ar[r, "\pi_2"]\ar[d, "\pi_1"']\pb&
		s\ar[d, "g"]\\
		r'\ar[r, "f"']&
		r
	\end{tikzcd}
	\qquad
	\begin{tikzcd}
		\sub(r'\times_r s)\ar[d, "\lsh{\pi_1}"']&
		\sub(s)\ar[d, "\lsh{g}"]\ar[l, "\ust{\pi_2}"']\\
		\sub(r')&
		\sub(r)\ar[l, "\ust{f}"]
	\end{tikzcd}	
	\]
	\item for each regular epimorphism $f\colon r'\surj r$ and $\varphi\in \sub(r)$, the following holds:
	\[\lsh{f}\big(\ust{f}(\varphi)\big)=\varphi.\]
	\item for each $f\colon r'\to r$, and $\varphi\in \sub(r)$ and $\varphi'\in \sub(r')$, Frobenius reciprocity holds:
	\[
	\lsh{f}(\varphi\wedge\ust{f}(\varphi'))=\lsh{f}(\varphi)\wedge\varphi'
	\]
\end{enumerate}
A \emph{regular functor} $\funr{F}\colon\cat{R}\to\cat{R}'$ induces a natural transformation $\alpha\colon\sub_{\cat{R}}\to\sub_{\cat{R}'}(\funr{F}-)$ such that
\begin{enumerate}
	\item $\alpha$ is natural with respect to both adjoints, $\lsh{f}$ and $\ust{f}$, for each $f\colon r'\to r$, and
	\item $\alpha_r$ is a meet-semilattice map for each $r\in\cat{R}$.
\end{enumerate}
\end{proposition}
\begin{proof}
  For the properties of the functor $\sub$, (1) can be easily verified by
  checking that that binary meets are given by pullback and the top element is
  given by the identity map, (2) is \cite[Lemma 2.9]{butz1998regular}, (3)
  follows from pullback stability of regular epis and uniqueness of
  factorizations (\cref{lemma.OFS}), and (4) is \cite[Lemma
  2.6]{butz1998regular}.

  The properties of $\alpha\colon\sub\to\sub(\funr{F}-)$ are found in/above \cite[Lemma
  2.10]{butz1998regular}.
\end{proof}

\section{The relations po-category construction}

A regular category $\cat{R}$ has exactly the structure and properties necessary
to construct a po-category of relations, or \define{relations po-category}.
 

\begin{definition}
  Let $\cat{R}$ be a regular category; its \emph{relations po-category}
  $\rrel{\cat{R}}$ is the po-category with the same objects as
  $\cat{R}$ but whose morphisms, written $x\colon r\tickar s$, are relations
  $x\ss r\times s$ in $\cat{R}$ equipped with the subobject
  ordering $x\leq x'$ iff $x\ss x'$. The composite $x\cp y$ with a relation
$y\colon s\tickar t$ is obtained by pulling back over $s$ and image factorizing the
map to $r\times t$:
\begin{equation}\label{eqn.rel_composition}
\begin{tikzcd}[column sep=small, row sep=15pt]
  &[10pt]&
  x\times_{s}y\ar[dl]\ar[dr]\ar[d, ->>]&&[10pt]~\\&
  x\ar[d, >->]&
  x\cp y\ar[d, >->]&
  y\ar[d, >->]\\&
	r\times s\ar[dl, bend right=8pt]\ar[dr, bend left=8pt]&
	r\times t\ar[dll, bend left=12pt, crossing over]&
	s\times t\ar[dl, bend right=8pt]\ar[dr, bend left=8pt]\\[-5pt]
	r&&
	s&&
	t\ar[from=ull, bend right=12pt, crossing over]
\end{tikzcd}
\end{equation}
$\rrel{R}$ also inherits a symmetric monoidal structure $I\coloneqq 1$ and $r_1\otimes
  r_2\coloneqq r_1\times r_2$ from the cartesian monoidal structure on $\cat{R}$.
  
  Given a regular functor $\funr{F}\colon\cat{R}\to\cat{R'}$, mapping a relation $x \ss
  r \times s$ to its factorization $\funr{F}(x) \surj \rrel{\funr{F}}(x) \inj \funr{F}(r\times s)
  \cong \funr{F}(r) \times \funr{F}(s)$ induces a (strong) symmetric monoidal po-functor
  $\rrel{\funr{F}}\colon\rrel{\cat{R}} \to \rrel{\cat{R'}}$. We refer to this po-functor
  as the \emph{relations po-functor} of $\funr{F}$.
\end{definition}

It is straightforward to check that the composition rule
\cref{eqn.rel_composition} is unital and associative using the pullback
stability of factorizations, and to check that $\rrel{\funr{F}}$ is indeed a symmetric
monoidal po-functor using the fact that a regular functor $\funr{F}\colon \cat{R} \to
\cat{R'}$ preserves pullbacks and image factorizations. Direct proofs in the
literature of these two facts seem difficult to find, but see for
example \cite[Theorem~2.3]{Jayewardene2000} and \cite[Proposition
4.1]{fong2017decorated} respectively.

The relations po-category is just a repackaging of the data of the regular
category: any regular category can be recovered, at least up to isomorphism, by
looking at the adjunctions in its relations po-category. 

\begin{lemma}[Fundamental lemma of regular categories]\label{lemma.fundamental}
Let $\cat{R}$ be a regular category. Then there is an identity-on-objects isomorphism
\[\cat{R}\to\ladj(\rrel{\cat{R}}).\]
In particular, a relation $x\colon r\tickar s$ is a left adjoint iff it is the graph $x=\pair{\id_r,f}$ of a morphism $f\colon r\to s$ in $\cat{R}$.
\end{lemma}
\begin{proof}
This fact is well known, but since it is crucial to what follows, we provide a proof here. We shall show that there is an
identity-on-objects, full, and faithful functor from $\cat{R}$ to its relations
po-category $\rrel{\cat{R}}$, which maps a morphism $f\colon r \to s$ to its
graph $\pair{\id_r,f} \ss r \times s$. Indeed, it is straightforward to check
that any pair of the form $\pair{\id_r,f} \dashv \pair{f,\id_s}$ is an adjunction,
and subsequently that the proposed map is functorial.

To show that it is full and faithful, we characterize the adjunctions $x \dashv
x'$ in $\rrel{\cat{R}}$. Suppose we have $x \stackrel{\pair{g,f}}\inj r \times
s$ and $x'\stackrel{\pair{f',g'}}\inj s\times r$ with unit $i\colon r\inj (x\cp x')$
and counit $j\colon (x'\cp x)\to s$. This gives rise to the following diagram
(equations shown right):
\[
\begin{tikzcd}[row sep=small, column sep=large]
	x\times_s x'\ar[rr, "\pi_s'"]\ar[ddd, "\pi_s"']&&
	x'\ar[dl, "g'"']\ar[ddl, "f'"]\\&
	r\ar[ul, "i"]\\&
	s\\
	x\ar[uur, "g"]\ar[ur, "f"']&&
	x\times_r x'\ar[uuu, "\pi_r'"']\ar[ll, "\pi_r"]\ar[ul, "j"]
\end{tikzcd}
\hspace{.5in}
\parbox{1.5in}{$
i\cp\pi_s\cp g=\id_r=i\cp\pi_s'\cp g'\\
\pi_r\cp f=j=\pi_r'\cp f'
$}
\]
We shall show that $g$ and $g'$ are isomorphisms, and that $f'=g' \cp g\inv \cp
f$. 

We first show that $i\cp \pi_s$ is inverse to $g$. Since the unit already gives
that $i \cp \pi_s \cp g = \id_r$, it suffices to show that $g\cp i\cp\pi_s=\id_x$.
Moreover, since $\pair{g,f}\colon x\to r\times s$ is monic and $g=(g\cp
i\cp\pi_s)\cp g$, it suffices to show that $f=(g\cp i\cp\pi_s)\cp f$. This is a diagram chase: since $g=g\cp i\cp\pi_s'\cp g'$, we can define a morphism $q\coloneqq\pair{\id_x,g\cp i\cp\pi_s'}\colon x\to x\times_r x'$, and we conclude
\[
	f=q\cp \pi_r\cp f
	=q\cp \pi_r'\cp f'
	=g\cp i\cp\pi_s'\cp f'
	=g\cp i\cp\pi_s\cp f.
\]
Similarly, we see that $i\cp \pi_s'$ is inverse to $g'$, and hence obtain $f'=g' \cp g\inv \cp f$.

Note that this implies the adjunction $x \dashv x'$ is isomorphic to the
adjunction $\pair{1_r,(g\inv \cp f)} \dashv \pair{(g \inv \cp f),\id_s}$. Thus the
proposed functor is full. Faithfulness amounts to the fact that the existence of
a morphism $\pair{1_r,f} \to \pair{1_r,f'}$ implies $f =f'$. This proves the
lemma.
\end{proof}

%

\begin{remark}\label{rem.rels_adjoint_composites}
It follows from the proof of \cref{lemma.fundamental} that $x\colon r\tickar s$ is a right adjoint iff it is the co-graph $\pair{f,\id_s}$ of a morphism $f\colon s\to r$. Furthermore, since any morphism $x=\pair{g,f}\colon r \tickar s$ in $\ccat{R}$ can be written as $x=\pair{g,\id_x} \cp \pair{\id_x,f}$, it follows that every morphism in $\ccat{R}$ can be written as the composite of a right adjoint followed by a left adjoint.
\end{remark}
%

The fundamental lemma says that regular categories can be recovered from their
relations po-categories. Similarly, any regular functor can be recovered as the
action of its relations po-functor on left adjoints. Before expressing this as a
categorical equivalence in \cref{eqn.equiv_regpocat}, we first make the following observation.

\begin{proposition}\label{prop.relations_2functors}
For any regular functor $\funr{F}\colon\cat{R}\to\cat{R}'$, the relations po-functor $\rrel{\funr{F}}\colon\rrel{\cat{R}}\to\rrel{\cat{R}'}$ is strong symmetric monoidal.
\end{proposition}
\begin{proof}
The functor $\funr{F}$ and its relations po-functor $\rrel{\funr{F}}$ act the same on objects, so since $\funr{F}$ is product preserving, $\rrel{\funr{F}}$ is strong monoidal.
\end{proof}

Although we do not assume it below, it is a result of Carboni and
Walters that every strong symmetric monoidal functor
$\rrel{\cat{R}}\to\rrel{\cat{R}'}$ is the relations po-functor associated to a
regular functor $\funr{F}\colon\cat{R}\to\cat{R}'$ \cite{Carboni:1987a}. Indeed, this
foreshadows the rephrasing of regular structure in terms of monoidal structure,
which runs through this paper.

In any case, this motivates the following definition.

\begin{definition}\label{def.regular_pocat}
  A po-category is called a \define{regular po-category} if it is isomorphic to the relations po-category $\rrel{\cat{R}}$ of some regular category $\cat{R}$.
  
  A strong symmetric monoidal po-functor between regular po-categories is called
  a \emph{regular po-functor} if it is isomorphic to the relations po-functor
  $\rrel{\funr{F}}$ associated to a regular functor $\funr{F}$. We write $\rgpocat$ for the
  category of regular po-categories.
\end{definition}

By the fundamental lemma (\ref{lemma.fundamental}), we now have an equivalence of categories:
\begin{equation}\label{eqn.equiv_regpocat}
  \begin{tikzcd}
    \rgcat \ar[r, shift left=5pt,"\rrel{-}"] \ar[r, phantom, "\cong"] & 
    \rgpocat. \ar[l, shift left=5pt,"\ladj"]
  \end{tikzcd}
\end{equation}

\chapter{Adjoint monoids and adjoint-lax functors} \label{chap.adjoint_monoids}

The poset of subobjects of an object in a regular category is always a
meet-semilattice. We characterize these as precisely the adjoint monoids in $\pposet$. The seemingly new notion of adjoint monoid makes sense in any monoidal po-category (and more generally):
an \emph{adjoint monoid} is an object equipped with commutative monoid and comonoid
structures such that the multiplication and unit are right adjoint to the comultiplication and counit. Every regular po-category $\ccat{R}$ is isomorphic to its own po-category of adjoint monoids $\ccat{R}\cong\aadjmon(\ccat{R})$. Finally, the subobjects functor
preserves adjoint monoids.

All these ideas are founded on the notion of adjoint-lax monoidal (ajax) po-functor.

\section{Definition and motivation}

In this section we introduce the notions of ajax functor and adjoint monoid.

\begin{definition}
  Let $\ccat{C}$ and $\ccat{D}$ be monoidal po-categories. An \define{adjoint-lax} or \define{ajax} po-functor $F\colon \ccat{C} \to \ccat{D}$ is a lax symmetric monoidal po-functor for which the laxators are right adjoints.
  
  We denote the laxators by $\rho$ and their left adjoints by $\lambda$:
  \[
  \adjr{I}{\rho}{\lambda}{F(I)}
  \qqand
  \adjr{F(c)\otimes F(c')}{\rho_{c,c'}}{\lambda_{c,c'}}{F(c\otimes c')}.
  \]
 \end{definition}

\begin{warning}
The notion of ajax functor has a dual notion of op-ajax functor: an oplax functor $\ccat{C}\to\ccat{D}$ for which the op-laxators are left adjoints. These two notions \emph{do not coincide}! The laxator naturality squares are asked to strictly commute in an ajax functor, and this property only implies that their mate squares, the corresponding oplaxator naturality squares weakly commute.
\end{warning}

Here is a obvious, but useful, consequence of the definition.

\begin{lemma}\label{lemma.ajax}
Every strong monoidal functor between monoidal po-categories is ajax. The composite of ajax po-functors is ajax.
\end{lemma}

Recall that $1$ is the terminal monoidal po-category.

\begin{proposition}\label{prop.adjoint_monoids}
Let $(\ccat{C},I,\otimes)$ be a monoidal po-category. There is a bijection between:
\begin{enumerate}
	\item The set of ajax functors $1\to\ccat{C}$,
	\item The set of commutative monoid objects $(c,\mu,\eta)$ such that $\mu$ and $\eta$ are right adjoints, and
	\item The set of cocommutative comonoid objects $(c,\delta,\epsilon)$ such that $\delta$ and $\epsilon$ are left adjoints.
\end{enumerate}
\end{proposition}
\begin{proof}
\begin{description}
	\item[$(1)\Leftrightarrow(2)$:]The set $\lax(1,\ccat{C})$ of lax symmetric monoidal functors $1\to\ccat{C}$ is well-known to be in bijection with the set of commutative monoid objects $(c,\mu,\eta)$ in $\ccat{C}$. Indeed, $\eta$ and unit $\mu$ come from the 0-ary and 2-ary laxators respectively: $\eta=\rho$ and $\mu=\rho_{1,1}$. Hence the added condition that $\eta$ and $\mu$ have left adjoints is precisely the ajax condition.
	\item[$(2)\Leftrightarrow(3)$:] Suppose given an object $c\in\ccat{C}$ and two adjunctions
	\begin{equation}\label{eqn.adjmon}
		\adjr{I}{\eta}{\epsilon}{c}.
		\qqand
		\adjr{c\otimes c}{\mu}{\delta}{c}
	\end{equation}
	Then $\mu,\eta$ satisfy the commutative monoid laws iff
	$\delta,\epsilon$ satisfy the cocommutative comonoid laws.
	\qedhere
\end{description}
\end{proof}
To summarize, if $(c,\rho,\lambda)\colon 1\to\ccat{C}$ is an ajax functor then the corresponding monoid and comonoid structures on $c$ are given by
\begin{equation}\label{eqn.monoid_comonoid_ajax}
	\eta=\rho\qquad
	\mu=\rho_{1,1}
	\qqand
	\epsilon=\lambda\qquad
	\delta=\lambda_{1,1}
\end{equation}

\Cref{prop.adjoint_monoids} motivates the following definition.

\begin{definition}
Let $(\ccat{C},I,\otimes)$ be a monoidal po-category. An \emph{adjoint commutative monoid} (or simply \emph{adjoint monoid}) in $\ccat{C}$ is a commutative monoid object $(c,\mu,\eta)$ in $\ccat{C}$ such that $\mu$ and $\eta$ are right adjoints.
\end{definition}

Adjoint monoids are a slight weakening of the \emph{internal meet semi-lattice} notion from theoretical computer science; see \cite[Chapter 5]{schalk1994algebras} and references therein.

\begin{proposition}\label{prop.ajax_pres_adjmon}
Ajax functors send adjoint monoids to adjoint monoids.
\end{proposition}
\begin{proof}
The composite of ajax functors is ajax, so the result follows from \cref{prop.adjoint_monoids}.
\end{proof}

We give examples of adjoint monoids after recalling the proof of a well-known lemma.

\begin{lemma}\label{lemma.comonoids_unique}
Let $\ccat{C}$ be a monoidal po-category. If the monoidal structure is cartesian (given by finite products in the underlying 1-category) then every object has a unique comonoid structure, and it is commutative.
\end{lemma}
\begin{proof}
Since the unit object is terminal, the maps $c\times\epsilon$ and $\epsilon\times c$ are forced to be the projections $c\times c\to c$, so $\delta$ is forced to be the diagonal. 
\end{proof}

\begin{proposition}\label{prop.adjmon_msl}
A poset $P\in\pposet$ is an adjoint monoid iff it is a meet-semilattice, in which case $\eta=\true$ and $\mu=\wedge$.
\end{proposition}
\begin{proof}
By \cref{lemma.comonoids_unique}, $P$ has a unique comonoid structure given by the terminal and diagonal maps $\epsilon\colon P\to 1$ and $\delta\colon P\to P\times P$. Thus $P$ is an adjoint monoid iff these maps have adjoints as in \cref{eqn.adjmon}, which holds iff $\eta$ is a top element and $\mu$ is a meet.
\end{proof}

\begin{proposition}\label{prop.adjmon_reg}
Let $\cat{R}$ be a regular category. Every object $r\in\ccat{R}$ in its relations po-category has a unique adjoint monoid structure.
\end{proposition}
\begin{proof}
Since $\cat{R}$ is cartesian monoidal, there is a unique cocommutative comonoid
structure on every object by \cref{lemma.comonoids_unique}. By the fundamental
lemma (\ref{lemma.fundamental}), we have an isomorphism
$\cat{R}\cong\ladj(\ccat{R})$, and we are done by \cref{prop.adjoint_monoids}
$(2)\Leftrightarrow(3)$.
\end{proof}

\section{Notions of morphism between ajax functors}

Given two ajax functors $F,F'\colon\ccat{C}\tto\ccat{D}$, we will consider two
sorts of (strong) natural transformations $\alpha\colon F\to F'$ between them,
differing in terms of the strength of their \emph{laxator naturality}. The first
sort only demands that the laxator naturality squares for any $c\in\ccat{C}$ be
\emph{mate squares} in $\ccat{D}$:
\begin{equation}\label{eqn.lax_monoid_hom}
\begin{tikzcd}[column sep=50, row sep=35]
	F(I)\ar[r, "\alpha_I"]\ar[d, shift left=5pt, "\lambda"]&
	F'(I)\ar[d, shift left=5pt, "\lambda'"]\ar[dl, phantom, "\Downarrow"]\\
	I\ar[r, equal]\ar[u, shift left=5pt, "\rho"]\ar[u, phantom, "\Downarrow"]&
	I\ar[u, shift left=5pt, "\rho'"]\ar[u, phantom, "\Downarrow"]
\end{tikzcd}
\qqand
\begin{tikzcd}[column sep=50, row sep=35]
	F(c\otimes c')\ar[r, "\alpha_{c\otimes c'}"]\ar[d, shift left=5pt, "\lambda_{c,c'}"]&
	F'(c\otimes c')\ar[d, shift left=5pt, "\lambda'_{c,c'}"]\ar[dl, phantom, "\Downarrow"]\\
	F(c)\otimes F(c')\ar[r, "\alpha_c\otimes\alpha_{c'}"']\ar[u, shift left=5pt, "\rho_{c,c'}"]\ar[u, phantom, "\Downarrow"]&
	F'(c)\otimes F'(c')\ar[u, shift left=5pt, "\rho'_{c,c'}"]\ar[u, phantom, "\Downarrow"]
\end{tikzcd}
\end{equation}
The meaning of each diagram in \eqref{eqn.lax_monoid_hom} is that any (and all)
of the following four equivalent conditions hold (dropping subscripts and writing
$\alpha^{\otimes 2}\coloneqq(\alpha\otimes\alpha)$):
\begin{equation}\label{eqn.lax_mon_hom_2}
\begin{tikzcd}[column sep=30]
	\rho\cp\alpha\leq\rho'\ar[r, phantom, "\Leftrightarrow"]\ar[d, phantom, "\Updownarrow"]&
	\rho\cp\alpha\cp\lambda'\leq I\ar[d, phantom, "\Updownarrow"]\\
	\alpha\cp\lambda'\leq \lambda\ar[r, phantom, "\Leftrightarrow"]&
	\alpha\leq\lambda\cp\rho'
\end{tikzcd}
\qand
\begin{tikzcd}[column sep=30]
	\rho\cp\alpha \leq \alpha^{\otimes 2}\cp \rho'\ar[r, phantom, "\Leftrightarrow"]\ar[d, phantom, "\Updownarrow"]&
	\rho\cp\alpha\cp\lambda'\leq \alpha^{\otimes 2}\ar[d, phantom, "\Updownarrow"]\\
	\alpha\cp\lambda'\leq \lambda\cp\alpha^{\otimes 2}\ar[r, phantom, "\Leftrightarrow"]&
	\alpha\leq\lambda\cp\alpha^{\otimes 2}\cp\rho'
\end{tikzcd}
\end{equation}
The second sort demands further that some of these inequalities be equalities.

\begin{definition}
Let $F,F'\colon\ccat{C}\tto\ccat{D}$ be ajax functors. A \emph{mate morphism} between them is a natural transformation $\alpha\colon F\to F'$ with mate squares as in \cref{eqn.lax_monoid_hom}. We say that $\alpha$ is \emph{strong} if the monoid part of the diagram strictly commutes (for all $c,c'\in\ccat{C}$):
	\begin{equation}\label{eqn.strong_mate}
		\rho\cp\alpha_I=\rho'
		\qqand
		\rho_{c,c'}\cp\alpha_{c\otimes c'}=(\alpha_c\otimes\alpha_{c
		'})\cp\rho'_{c,c'}.
	\end{equation}
	We denote the corresponding categories by $\ajax(\ccat{C},\ccat{D})$ and $\ajax^\str(\ccat{C},\ccat{D})$ respectively.
	
Suppose $\alpha,\beta\colon F\to F'$ are mate morphisms (possibly strong). We write $\alpha\leq\beta$ if for all $c\in\ccat{C}$ we have $\alpha_c\leq\beta_c$ in the poset $\ccat{C}(F(c),F'(c))$. We denote the corresponding po-categories as
\[
	\aajax(\ccat{C},\ccat{D})
	\qqand
	\aajax^\str(\ccat{C},\ccat{D}).
\]
\end{definition}

Clearly, the po-category structure of $\aajax(\ccat{C},\ccat{D})$ is inherited from $\ppocat(\ccat{C},\ccat{D})$; we record this fact in the following obvious lemma.

\begin{lemma}
The map $\aajax(\ccat{C},\ccat{D})\to\ppocat(\ccat{C},\ccat{D})$ is locally fully faithful.
\end{lemma}

\begin{definition}
Let $\ccat{C}$ be a monoidal po-category. Define po-categories
\[
	\aadjmon(\ccat{C})\coloneqq\aajax(1,\ccat{C})
	\qqand
	\aadjmon^\str(\ccat{C})\coloneqq\aajax^\str(1,\ccat{C}),
\]
and refer to them as the po-category of \emph{adjoint monoids} and the po-category of \emph{adjoint monoids and strong maps} respectively.
\end{definition}

Let $\pposet\vert_{\msl}$ denote the full sub-po-category of $\pposet$ spanned
by the meet-semilattices, and let $\msl$ denote the po-category of
meet-semilattices and meet-preserving maps.

\begin{proposition}\label{prop.adjmon_poset}
There are isomorphisms of po-categories
\[
	\aadjmon(\pposet)\cong\pposet\vert_{\msl}
	\qqand
	\aadjmon^\str(\pposet)\cong\msl.
\]
\end{proposition}
\begin{proof}
By \cref{prop.adjmon_msl,eqn.monoid_comonoid_ajax} we have the desired isomorphisms on objects, and
$\rho=\eta=\true$ and $\rho_{1,1}=\mu=\wedge$. Since every poset map $\alpha\colon P\to P'$ is a
comonoid homomorphism, we have mate diagrams as in \cref{eqn.lax_monoid_hom},
giving $\aadjmon(\pposet)\cong\pposet\vert_{\msl}$.

To see the isomorphism
$\aadjmon^\str(\pposet)\cong\msl$, note that by \eqref{eqn.monoid_comonoid_ajax}, the equations in
\eqref{eqn.strong_mate} precisely say $\alpha(\true)=\true$ and
$\alpha(\wedge)=\wedge(\alpha,\alpha)$.
\end{proof}

\begin{proposition}\label{prop.adjmon_regular}
Let $\ccat{R}$ be a regular po-category. There are isomorphisms
\[
	\aadjmon(\ccat{R})\cong\ccat{R}
	\qqand
	\aadjmon^\str(\ccat{R})\cong\rradj(\ccat{R}).
\]
\end{proposition}
\begin{proof}
In \cref{prop.adjmon_reg} we gave an isomorphism $\ob\ccat{R}\cong\ob\aadjmon(\ccat{R})$ coming from the fact that every object $r\in\cat{R}=\ladj(\ccat{R})$ has a unique comonoid structure. Suppose $r\from \alpha\to r'$ is a morphism in $\ccat{R}$. Then there exist unique maps $e,d$ making following diagrams commute:
\[
\begin{tikzcd}[row sep=small]
	r&
	\alpha\ar[l, "f"']\ar[r, "f'"]\ar[dd, dotted, "e" description]&
	r'\\
	r\ar[u, equal]\ar[d]&&
	r'\ar[u, equal]\ar[d]\\
	1&
	1\ar[l, equal]\ar[r, equal]&
	1
\end{tikzcd}
\qqand
\begin{tikzcd}[row sep=small]
	r&
	\alpha\ar[l, "f"']\ar[r, "f'"]\ar[dd, dotted, "d" description]&
	r'\\
	r\ar[u, equal]\ar[d]&&
	r'\ar[u, equal]\ar[d]\\
	r\times r&
	\alpha\times \alpha\ar[l]\ar[r]&
	r'\times r'
\end{tikzcd}
\]
Thus there is an isomorphism between the posets $\ccat{R}(r,r')$ and $\aadjmon(r,r')$. 

Unwinding the definition of strong morphisms between the ajax maps $r,r'\colon 1\to\ccat{R}$ the equation $\eta\cp \alpha=\eta$ implies that the map $\im(f')\to r'$ is an iso, i.e.\ $f'$ is a regular epi; similarly the equation $\mu\cp\alpha=(\alpha\otimes\alpha)\cp\mu$ implies that the map $\alpha\to \alpha\times_{r'}\alpha$ is iso, i.e.\ $f'$ is a mono. In other words, $\alpha$ is strong iff $f'$ is iso, and this holds iff $\alpha$ is a right adjoint (see \cref{rem.rels_adjoint_composites}).
\end{proof}

In passing we note the following connection to hypergraph categories, which are
well known for their own graphical language, and may help some readers contextualize our
main result. This is a corollary of \cite[Theorem 3.1]{fong2017decorated}. 

\begin{proposition}
  Given a regular category $\cat{R}$, the monoidal category underlying $\rrel{\cat{R}}$
  has a hypergraph structure, where the symmetric monoidal structure is given by
  the product in $\cat{R}$, and where for any object $x$ in $\rrel{\cat{R}}$ we
  have $\mu_x$ and $\delta_x$ given by the diagonal subobject $x\ss x \times x
  \times x$, and $\eta_x$ and $\epsilon_x$ given by the maximal subobject $x\ss x$. 
  \end{proposition}

Loosely speaking, one might think of a regular po-category (\cref{def.regular_pocat}) as a poset-enriched hypergraph
category in which the hom-posets are meet-semilattices.

\section{The subobjects-functor is ajax}

Let $\cat{R}$ be a regular category and recall the subobjects functor $\sub\colon\cat{R}\to\ladj(\pposet)$ from \cref{prop.sub_ladj}. It extends to a po-functor $\sub\colon\ccat{R}\to\pposet$, where $\ccat{R}=\rrel{\cat{R}}$ is the relations po-category. To be explicit, write a relation $A\ss r\times r'$ as a span $r\From{f} A\To{f'} r'$. Then the map $\sub(A)\colon\sub(r)\to\sub(r')$ applied to a subobject $\varphi\ss r$ is given pulling back and then taking the image:
\[
\begin{tikzcd}
	\varphi\ar[d, >->]&
	\cdot\ar[d, >->]\ar[l]\ar[r, ->>]&
	\cdot\ar[d, >->]\\
	r\ar[ur, phantom, very near end, "\llcorner"]&
	A\ar[l]\ar[r]&
	r'
\end{tikzcd}
\]
That is, $\sub(A)=\lsh{f}\cp\ust{g}$. This po-functor is representable: $\sub(-)=\ccat{R}(I,-)$, where $I$ is the terminal object in $\cat{R}$. We now show this po-functor is ajax.

\begin{theorem}\label{thm.sub_is_ajax}
The po-functor 
$
  \sub_{\cat{R}}\colon \ccat{R} \longrightarrow \pposet
$
is ajax for any regular po-category $\ccat{R}$.
\end{theorem}
\begin{proof}
The functor $\sub_{\cat{R}}(-)=\ccat{R}(I,-)$ has a canonical lax monoidal
structure since $I\otimes I\cong I$.  We need to show the laxators $\otimes$ and
$\id_I$ have left adjoints in $\pposet$. The first is easy: $\id_I$ is the top element in
$\ccat{R}(I,I)$ and thus a right adjoint since there is a unique map
$\ccat{R}(I,I)\to 1$.

Now suppose given $r_1,r_2\in\ccat{R}$, and consider the morphisms $\pi_i\colon r_1\otimes r_2\to r_i$ and $\delta\colon r\to r\otimes r$ corresponding to the $i$th projection and the diagonal in $\cat{R}$. Composition with the $\pi_i$ induces a monotone map $\lambda_{r_1,r_2}\colon\ccat{R}(I, r_1\otimes r_2)\to\ccat{R}(I, r_1)\times\ccat{R}(I, r_2)$, natural in $r_1,r_2$. It remains to show that each $\lambda_{r_1,r_2}$ is indeed a left adjoint,
\[
	\adj{\ccat{R}(I, r_1\otimes r_2)}{\lambda_{r_1, r_2}}{\otimes}{\ccat{R}(I, r_1)\times\ccat{R}(I, r_2)}.
\]
For the unit, given $g\colon I\to r_1\otimes r_2$, we have
\begin{align*}
  g
  &=g\cp\delta_{r_1\otimes r_2}\cp((r_1\otimes r_2)\otimes\epsilon_{r_1\otimes r_2})\\
  &=g\cp\delta_{r_1\otimes r_2}\cp((r_1\otimes\epsilon_{r_2})\otimes(\epsilon_{r_1}\otimes r_2))\\
  &\leq \delta_I\cp (g\otimes g)\cp ((r_1\otimes\epsilon_{r_2}))\otimes(\epsilon_{r_1}\otimes r_2)\\
  &=(g\cp(r_1\otimes\epsilon_{r_2}))\otimes(g\cp(\epsilon_{r_1}\otimes r_2)).
\end{align*}
For the counit, given $f_1\colon I\to r_1$ and $f_2\colon I\to r_2$, it suffices to show that $(f_1\otimes f_2)\cp(r_1\otimes \epsilon_{r_2})\leq f_1$, since the other projection is similar. And this holds because
\[
	(f_1\otimes f_2)\cp(r_1\otimes \epsilon_{r_2})
	=f_1\otimes (f_2\cp\epsilon_{r_2})
	\leq f_1\otimes \epsilon_I=f_1.
\qedhere
\]
\end{proof}

\begin{corollary}\label{cor.meetsl}
The po-functor $\sub_{\cat{R}}\colon\ccat{R}\to\pposet$ sends each object $r\in\ccat{R}$ to a meet-semilattice.
\end{corollary}
\begin{proof}
This follows from the fact that ajax functors send adjoint monoids to adjoint
monoids; see \cref{thm.sub_is_ajax,prop.ajax_pres_adjmon,prop.adjmon_poset,prop.adjmon_regular}.
\end{proof}

\begin{remark}\label{rem.meet_pres}
In fact, we can get a bit more from \cref{prop.adjmon_poset,prop.adjmon_regular}. If $x\colon r\tickar r'$ is an arbitrary map in $\ccat{R}$ then the monotone map $\sub(x)\colon\sub_\cat{R}(r)\to\sub_{\cat{R}}(r')$ is \emph{not necessarily meet-preserving}. However, if $x$ is the image of a morphism in $\cat{R}\op=\radj(\ccat{R})$ then $\sub(x)$ is meet-preserving. That is, $\cat{R}\op$ is isomorphic to the underlying 1-category of $\aadjmon^\str(\ccat{R})$.
\end{remark}

\begin{remark}\label{rem.tangible}
In \cref{prop.tangible} we discussed four properties of the subobjects functor $\sub\colon\cat{R}\to\ladj(\pposet)$: a meet-semilattice on each object, Beck-Chevalley for pullbacks, Beck-Chevalley for pushouts of effective epimorphisms, and Frobenius reciprocity. These same four properties in fact hold for any ajax po-functor $F\colon \rrel{\cat{R}}\to\pposet$. Moreover, this construction is invertible: given any functor $\cat{R}\to\ladj(\pposet)$, there is an induced ajax po-functor $\rrel{\cat{R}}\to\pposet$, and the two constructions are mutually inverse
\[
\ajax(\rrel{\cat{R}},\pposet)\cong\{S\colon\cat{R}\to\ladj(\pposet)\mid \tn{ four properties in \ref{prop.tangible}}\}.
\]
We do not need this result for our work, so we omit the details. However, it is interesting to see these four well-known---though slightly mysterious---properties fall out of the more elementary definition of adjoint-lax functors to $\pposet$.
\end{remark}

\chapter{Free regular categories and regular calculi} \label{chap.free_reg}

We now construct the free regular category $\frc$---as well as the
free regular po-category $\frb$---on a set $\typeset$. This allows us to define,
in \cref{sec.reg_calc}, a \emph{regular calculus} to be an ajax po-functor
$\frb\to\pposet$. Eventually, in \cref{thm.main}, we will see that $\rgcat$ is
essentially a reflective subcategory of the category $\rgcalc$ of regular
calculi, in the sense that there is an adjunction $\rgcat\tofrom\rgcalc$ such that for any regular category the counit map is an
equivalence of categories. Towards that end, we conclude this section in
\cref{prop.rels} by defining the right adjoint part, $\prd\colon\rgcat\to\rgcalc$.

\section{The free regular category on a set}\label{sec.frc}
We will propose a graphical calculus based on regular categories $\frc$ free on a set $\typeset$. We define $\frc$ in \cref{def.free_reg} and show that it is free in \cref{thm.fr_is_free}.

Write $\ppf(\typeset)$ for the poset of finite subsets of $\typeset$; this, or
equally its opposite category $\ppf(\typeset)\op$, is a free
$\wedge$-semilattice on $\typeset$. Write also $\finset$ for the
category of finite sets and functions. Note that $\finset\op$ is the free
category with finite limits on one object. The free regular category on
$\typeset$ arises when these two structures interact.

Note that for any $\typeset$ there is an inclusion of categories
$\inc\colon\ppf(\typeset)\to\finset$.

\begin{definition}\label{def.free_reg}
Define $\frc\coloneqq(\ppf(\typeset)\op\downarrow\finset\op)$ to be the comma category
\[
\begin{tikzcd}[column sep=30pt, row sep=5pt]
	&\frc\ar[dr, "\vars"]\ar[dl, "\supp"']\\
	\ppf(\typeset)\op\ar[dr, "\inc"']\ar[rr, phantom, "\xRightarrow{\ \type\ }"]&&
	\finset\op\ar[dl, "\id"]\\
	&\finset\op
\end{tikzcd}
\]
for any set $\typeset$. We refer to objects $\Gamma\in\frc$ as \emph{contexts}.
\end{definition}
\goodbreak

We can unpack a context $\Gamma$ into a quasi-traditional form, e.g.\ as
\[\Gamma\qquad=\qquad x_1:\tau_1,\ldots,x_n:\tau_n\mid\tau_1',\ldots \tau_m'\]
which has a finite set of \emph{variables}, $\vars(\Gamma)=\{x_1,\ldots,x_n\}$, whose \emph{support} set is $\supp(\Gamma)=\{\tau_1,\ldots,\tau_n,\tau_1'\ldots,\tau_m'\}$, and which has the \emph{typing} function $\type(x_i)=\tau_i$. The notion of support does not typically have a place in traditional logical contexts, but we include it because $\supp(\Gamma)$ has a definite place in objects of the free regular category.

Working in the skeleton of $\frc$, we can assume that each cardinality has a unique set of variables, e.g.\ $\ord{n}=\{1,\ldots,n\}$. Here is an equivalent but more concrete description of the free regular category on $\typeset$:
\begin{equation}\label{eqn.fr_Lambda_explicit}
\begin{aligned}
	\ob\frc\coloneqq&
	\big\{(n, S, \tau)\mid n\in\nn, S\ss\typeset \tn{ finite}, \tau\colon\ord{n}\to S\big\}\\
	\frc\big(\Gamma,\Gamma'\big)\coloneqq&
	\left\{f\colon \ord{n}' \to \ord{n} \, \middle|
	\begin{tikzcd}[row sep=0]
		\ord{n}\ar[r, "\tau"]&
		S\ar[dr, phantom, sloped, "\ss"]\\
		&&[-13pt]
		\typeset\\
		\ord{n}'\ar[r, "\tau'"']\ar[uu, "f"]&
		S'\ar[uu, phantom, sloped, "\ss"]\ar[ur, phantom, sloped, "\ss"]
	\end{tikzcd}
	\right\}
\end{aligned}
\end{equation}
Given a map $f\colon \Gamma\to\Gamma'$, we denote the corresponding function as
$\ord{f}\colon\ord{n}'\to\ord{n}$. Say that a context $\Gamma=(n, S, \tau)$:
\begin{itemize}
	\item is a \emph{unary context} if it is of the form $(1,\{s\},!)$, i.e. if it has arity $n=1$ and full support $\abs{S}=1$; we denote it simply as $\unary{s}$.
	\item is a \emph{unary support context} if it is of the form
	  $(0,\{s\},!)$; i.e. if it has $n=0$ and $\abs{S}=1$; we abuse notation
	  to denote this $\supp(s)$.
\end{itemize}

\begin{example}
Suppose $\typeset=\{s\}$ is unary. When $n=0$, the map $\tau$ is unique, and we either have $S=\varnothing$ or $S=\{s\}$. Thus we recover the description from \cref{eqn.free_reg}, though in the present terms it looks like this:
\[
\begin{tikzcd}
	(0,\varnothing)&
	(0,\{s\})\ar[l, >->]&
	(1,\{s\})\ar[l, ->>]\ar[r]&
	(2,\{s\})\ar[l, shift left=5pt]\ar[l, shift right=5pt]&
	\cdots
\end{tikzcd}
\]
\end{example}

\begin{example}
For any set $\typeset$, the poset of subobjects of $0$ in $\frc$ is the free meet-semilattice on $\typeset$, i.e.\ the finite powerset $\ppf(\typeset)$. This will follow from \cref{cor.descriptions}.
\end{example}

In \cref{thm.fr_is_free} we will show that $\frc$ is indeed the free regular category on $\typeset$. The following is straightforward.

\begin{lemma}\label{lemma.comma_limits}
Suppose $\cat{C}$, $\cat{D}$, and $\cat{E}$ have $I$-shaped limits, for some small category $I$, and suppose that $f\colon\cat{C}\to\cat{E}$ and $g\colon\cat{D}\to\cat{E}$ preserve $I$-shaped limits. Then the comma category $\cat{B}\coloneqq(\cat{C}\downarrow\cat{D})$ has $I$-shaped limits, and they are preserved and reflected by the projection $(\pi_1,\pi_2)\colon\cat{B}\to\cat{C}\times\cat{D}$.
\end{lemma}
\exclude{\begin{proof}
Let $X\colon I\to\cat{B}$ be a diagram, and let $!_I\colon I\to\terminal$ denote the unique functor. Suppose we have limits, i.e.\ natural transformations $!_I\cp\lim(X\cp\pi_1)\to(X\cp\pi_1)$ and $!_I\cp\lim(X\cp\pi_2)\to(X\cp\pi_2)$ in $\cat{C}$ and $\cat{D}$ satisfying the correct universal property. Since $f$ and $g$ preserve $I$-shaped limits, we have $\lim(X\cp\pi_1\cp f)=\lim(X\cp\pi_1)\cp f$ and similarly for $\pi_2$ and $g$. The comma natural transformation $\pi_1\cp f\to \pi_2\cp g$ induces a map $\lim(X\cp\pi_1\cp f)\to\lim(X\cp\pi_2\cp g)$, which we regard as an object of $\cat{B}$, which we call $\lim X$. The commutative square
\[
\begin{tikzcd}
	!_I\cp\lim(X\cp\pi_1\cp f)\ar[r]\ar[d]&
	X\cp\pi_1\cp f\ar[d]\\
	!_I\cp\lim(X\cp\pi_2\cp g)\ar[r]&
	X\cp\pi_2\cp g
\end{tikzcd}
\]
provides a structure map $!_I\cp\lim X\to X$, and it is easy to check it has the correct universal property. Moreover, the maps $\pi_1$ and $\pi_2$ certainly preserve limits, and if $!_I\cp L\to X$ is a cone in $\cat{B}$ that projects to a limit in both $\cat{C}$ and $\cat{D}$ then we have just seen that it is a limit in $\cat{B}$.
\end{proof}}

\begin{proposition}\label{prop.comma_regular}
Let $\cat{R}\to\cat{T}\from\cat{S}$ be regular functors. Then the comma category $\cat{B}\coloneqq(\cat{R}\downarrow\cat{S})$ is regular, and the projection $\cat{B}\to\cat{R}\times\cat{S}$ preserves and reflects finite limits and regular epimorphisms. In particular, $\frc$ is regular for any $\typeset$.
\end{proposition}
\begin{proof}
It is well-known that $\finset\op$ is regular, and the finite powerset
$\ppf(\typeset)$ is regular because it has finite meets and, because it is a
poset, regular epis are equalities. Hence the second statement follows from the
first. Since the opposite of a comma category is the comma category of the
opposites of its defining data, \cref{lemma.comma_limits} shows that $\cat{B}$
has finite limits and coequalizers of kernel pairs, and that regular epis are
stable under pullback.
\end{proof}

\begin{corollary}\label{cor.descriptions}
It will be useful to be have the following explicit computations in $\frc$. 
\begin{description}
	\item[terminal:] $0\To{!}\varnothing\ss\typeset$ is terminal. We denote it $0$.
	\item[product:] The product of $\Gamma=(n, S, \tau)$ and $\Gamma'=(n',
	  S', \tau')$ is $(n+n', S\cup S',\copair{\tau,\tau'})$. We denote it $\Gamma\oplus\Gamma'$.
	\item[pullback:] The pullback of a diagram $(n_1, S_1, \tau_1)\to (n, S, \tau)\from (n_2, S_2, \tau_2)$ is obtained as a pushout (and union) in $\finset$:
\[
\begin{tikzcd}[sep=small]
	\ord{n}\ar[rr]\ar[rd]\ar[dd]&&
	\ord{n}_2\ar[rd]\ar[dd]\\&
	S\ar[rr, crossing over]&&
	S_2\ar[dd]\\
	\ord{n}_1\ar[rr]\ar[dr]&&
	\ord{n}_1\sqcup_{\ord{n}}\ord{n}_2\ar[dr]\\&
	S_1\ar[from=uu, crossing over]\ar[rr]&&
	S_1\cup S_2\ar[r, phantom, "\ss"]&[-10pt]\typeset
\end{tikzcd}
\]
	\item[monos:] A map $f\colon(n_1, S_1, \tau_1)\to (n_2, S_2, \tau_2)$ is
	  monic iff the function $\ord{f}\colon\ord{n}_2\to\ord{n}_1$ is
	  surjective.
	\item[regular epis:] A map $f\colon (n_1, S_1, \tau_1)\to (n_2, S_2,
	  \tau_2)$ is regular epic iff both: the corresponding function
	  $\ord{f}\colon \ord{n}_2\to \ord{n}_1$ is injective and $S_2=S_1$.
\end{description}
\end{corollary}

\begin{remark}
As mentioned in \cref{cor.descriptions}, we denote the product of $\Gamma_1$ and $\Gamma_2$ by $\Gamma_1\oplus\Gamma_2$. This is reminiscent of the notation for products in an abelian category. However, it is not quite analogous: in an abelian category the product $V\oplus W$ is a biproduct---i.e.\ also a coproduct---and this is not the case in $\frc$. We use the $\oplus$ notation to remind us that
 \[(n,S,\tau)\oplus(n',S',\tau')\cong(n+n',S\cup S', \copair{\tau,\tau'}).\]
\end{remark}

\begin{remark}
Note that one should think of the support $S=\supp(\Gamma)$ of a context $\Gamma$ as a kind of constraint, because the larger $S$ is, the smaller $\Gamma$ is. Indeed, for any $n\in\nn$ and context $\tau\colon\ord{n}\to S$, if one composes with an inclusion $S\ss S'\ss\typeset$ on the level of support, the result is a monic map in $\frc$ \emph{going the other way},
\[(\ord{n}\To{\tau} S\ss S')\inj(\ord{n}\To{\tau} S).\]
\end{remark}

Recall from \cref{def.support} that the support of an object in a regular
category is the image of its unique map to the terminal object. 


\begin{corollary}\label{cor.support_as_of_unary}
Every unary support context is the support of a unary context.
\end{corollary}
\begin{proof}
  Given any unary support context $\supp(s)$, the explicit descriptions in
  \cref{cor.descriptions} make it easy to check that $\unary{s} \surj
  \supp(s)\inj 0$ is the image factorization of the unique map $\unary{s} \to
  0$.
\end{proof}

\begin{corollary}\label{cor.factor_unary_support}
Every object $\Gamma=(n, S, \tau)\in\frc$ can be written as the product of $n$-many unary contexts and $\abs{S}$-many unary support contexts, and morphisms in $\frc$ correspond to projections and diagonals.
\end{corollary}
\begin{proof}
It follows directly from \cref{cor.descriptions} that
$\Gamma=\prod_{i\in\ord{n}}\unary{\tau(i)}\times\prod_{s\in
S}\supp(s).$ In particular, it will be useful to note the idempotence of support contexts:
\begin{equation}\label{eqn.idem_support}
  \supp(s)\times\supp(s)=\supp(s).
\end{equation}
If $f\colon\Gamma\to\Gamma'$ is a morphism as in \cref{eqn.fr_Lambda_explicit}, then the corresponding map
\[
\begin{tikzcd}
  \prod_{i\in\ord{n}}\unary{\tau(i)}\times\prod_{s\in S}\supp(s)\ar[d,"f"]\\
	\prod_{i'\in\ord{n}'}\unary{\tau'(i')}\times\prod_{s'\in S'}\supp(s')
\end{tikzcd}
\]
acts coordinatewise according to $\ord{f}\colon\ord{n}'\to\ord{n}$ and $S'\ss S$.
\end{proof}

The following theorem establishes the adjunction from \cref{eqn.fr_ob_adj}.

\begin{theorem}\label{thm.fr_is_free}
The category $\frc$ is the free regular category on $\typeset$, i.e.\ there is an adjunction
  \[
    \adj[40pt]{\smset}{\frc[-]}{\ob}{\rgcat.}
  \]
\end{theorem}
\begin{proof}
We denote the unit component for a set $\typeset$ by $\unary{-}\colon \typeset\to\ob\frc$; it is given by unary contexts, $\unary{t}=(1,\{t\},!)$. We denote the counit component $\Prod\colon\frc[\ob\cat{R}]\to\cat{R}$ for a regular category $\cat{R}$; it is roughly-speaking given by products and supports in $\cat{R}$ (see \cref{def.support}). More precisely, given a context $\Gamma=(n,S,\tau)\in\frc[\ob\cat{R}]$, we put
\[
	\Prod[\Gamma]\coloneqq
	\prod_{i\in\ord{n}}\tau(i)
	\times
	\prod_{s\in S}\supp(s),
\]
By the universal property of products, a morphism $f\colon\Gamma\to\Gamma'$, i.e.\ a function $\ord{f}\colon \ord{n}'\to \ord{n}$ as in \cref{eqn.fr_Lambda_explicit} naturally induces a map $\Prod[f]\colon\Prod[\Gamma]\to\Prod[\Gamma']$, so $\Prod$ is a functor. We need to check that it is regular and for this we use \cref{cor.descriptions}.

For preservation of finite limits, first
observe that $\Prod$ preserves the terminal object because the empty product in $\cat{R}$ is terminal. For pullbacks we need to check that for every pushout diagram
as to the left, the diagram to the right is a pullback:
\[
\begin{tikzcd}
  \ord{n}\ar[r]\ar[d]\ar[dr, phantom, very near end, "\ulcorner"]&
  \ord{n}_2\ar[d]
  	&[-15pt]&[20pt]
	\prod_{i\in\ord{n}}\unary{\tau'(i)} \ar[dr, phantom, very near end, "\ulcorner"]&
	\prod_{i_2\in\ord{n}_2}\unary{\tau'(i_2)} \ar[l]\\
  \ord{n}_1\ar[r]&
  \ord{n}'\ar[r, "\tau'"]
  &\typeset
	&
	\prod_{i_1\in\ord{n}_1}\unary{\tau'(i_1)} \ar[u]&
	\prod_{i'\in\ord{n}'}\unary{\tau'(i')}\ar[l]\ar[u]
\end{tikzcd}
\]
where $\ord{n'}\cong\ord{n}_1\sqcup_{\ord{n}}\ord{n}_2$ and
$\tau'\colon\ord{n}'\to\typeset$ is the induced map. This follows from the
well-known fact that $\finset$ is the free finite-colimit completion of a
singleton \cite{MacLane.Moerdijk:1992a}, and fact that the slice category
$\finset_{/\typeset}$ is the free
finite-colimit completion of the set $\typeset$.

Finally, suppose $f\colon (n_1, S_1, \tau_1)\to (n_2, S_2, \tau_2)$ is a regular epi in $\frc$; i.e.\ the corresponding function $\ord{f}\colon\ord{n}_2 \to \ord{n}_1$ is monic and $S_1=S_2$. Letting $n'\coloneqq n_1-n_2$, we use \cref{cor.factor_unary_support,eqn.idem_support} to write $\Prod[f]$ as follows:
\[
\begin{tikzcd}
  \displaystyle \prod_{i\in\ord{n}'}\unary{\tau_1(i)}\times
	 \prod_{i\in\ord{n}_2}\unary{\tau_2(i)}\times
	  \prod_{s\in S}\supp(s)\ar[d, "{\Prod[f]}"]
	\\
	\displaystyle \prod_{i\in\ord{n}'}\supp (\tau_1(i))\times
	\prod_{i\in\ord{n}_2}\unary{\tau_2(i)}\times
	\prod_{s\in S}\supp (s).{\color{white}Sup}
\end{tikzcd}
\]
Since for each $i\in\ord{n}'$ the map $\tau_1(i)\to\supp(\tau_1(i))$ is a
regular epi and regular epis are closed under finite products in a regular
category, this shows that $\Prod[f]$ is again a regular epi. Hence
$\Prod$ is a regular functor. The triangle identities are straightforward: the first is that for any $r\in\ob\cat{R}$, the product of a unary context $\unary{r}$ is just $r$. The second follows from \cref{cor.factor_unary_support}.
\end{proof}

Given a function $f\colon\typeset\to\typeset'$, we can use \cref{thm.fr_is_free} and the idempotence of support contexts to see that the induced regular functor $\frc[f]\colon\frc\to\frc[\typeset']$ sends $\Gamma=(\ord{n}\To{\tau}S\ss\typeset)$ to the composite
\[\frc[f](\Gamma)=(\ord{n}\To{\tau}S\stackrel{f\vert_S}\surj f(S)\ss\typeset'),\]
where $S\surj f(S)\inj \typeset'$ is the image factorization of $f$ restricted to
$S$.

\begin{remark}
The free finite limit category on a single generator is $\finset\op$, and there the unique map $\ord{n}\to \varnothing$ is a regular epimorphism for every object $\ord{n}$. Consequently, $\finset\op$ has another universal property: it is the free regular category \emph{in which every object is inhabited}. Of course the same holds for any set $\typeset$: the free finite limit category is also the free ``fully inhabited'' regular category. It is equivalent to the result of inverting the map $(\varnothing,S,!)\to(\varnothing,\varnothing,!)$ in $\frc$ for every $S\in\ppf(\typeset)$.

Because $(\finset\slice{T})\op$ is very similar to---but far more familiar than---$\frc$, it can be useful for intuition to replace $\frc$ with $\finset\op$ throughout this story; the only cost is the assumption of inhabitedness, which is a common assumption in classical logic.
\end{remark}

For any regular category $\cat{R}$, the counit map of the adjunction in \cref{thm.free_reg_bicat} gives a regular functor that we have been denoting
\begin{equation}\label{eqn.counit_reg_cat}
	\Prod\colon\frc[\ob\cat{R}]\too\cat{R}.
\end{equation}
It sends a context $\Gamma=(n,S,\tau)$ to the product
\begin{equation}\label{eqn.Prod}
  \Prod[\Gamma]\coloneqq\prod_{i\in\ord{n}}\unary{\tau(i)}\times\prod_{s\in
  S}\supp(s).
\end{equation}

\subsection{The free regular po-category on a set}

Since $\frc$ is a regular category, we may construct its po-category of
relations
\[\frb\coloneqq\rrel{\frc}.\]
It should be no surprise that these are the
free regular po-categories. Free regular po-categories will form the
foundation of our graphical calculus for regular logic; we give an explicit
description in \cref{chap.graphical_reglog}.

\begin{theorem}\label{thm.free_reg_bicat}
  The po-category $\frb \coloneqq \rrel{\frc}$ is the free regular po-category on the set
  $\typeset$. That is, there is an adjunction
  \[
    \adj[40pt]{\smset}{\frb[-]}{\ob}{\rgpocat.}
  \]
\end{theorem}
\begin{proof}
  This is immediate from \cref{thm.fr_is_free}, which says that $\frc$ is the
  free regular category on $\typeset$, and the fact that the category of regular
  po-categories is the essential image of $\rgcat$ under the po-category of
  relations construction, \cref{def.regular_pocat}.
\end{proof}

For any regular po-category $\ccat{R}$, the counit map of the adjunction in \cref{thm.free_reg_bicat} gives a morphism of regular po-categories that we again denote
\begin{equation}\label{eqn.counit_reg_bicat}
	\Prod\colon\frb[\ob\ccat{R}]\too\ccat{R}.
\end{equation}
It is a strong monoidal functor, basically because the functor in \cref{eqn.counit_reg_cat} in particular preserves finite products.

\section{Regular calculi}\label{sec.reg_calc}

In this section we introduce the notion of a regular calculus. This is a new category-theoretic way to look at the kinds of logical moves---and the relationships between them---found in regular logic.

\subsection{Definition of regular calculi}

The following was given as \cref{def.reg_sketch}, but we repeat it here for
convenience. Recall from \cref{thm.free_reg_bicat} that $\frb\coloneqq\rrel{\frc}$ is the free regular po-category on $\typeset$.

\begin{definition}
A \emph{regular calculus} is a pair $(\typeset, \pr)$ where $\typeset$ is a set
and  $\pr\colon\frb\to\pposet$ is an ajax po-functor. For any object
$\Gamma\in\frb$, we denote the order in the poset $\pr(\Gamma)$ using the
$\vdash_\Gamma$ or $\vdash$ symbol (rather than $\leq$). 

A \emph{morphism} $(\typeset, \pr)\to(\typeset',\pr')$ of regular calculi is a pair $(F,F^\sharp)$ where $F\colon\typeset\to\typeset'$ is a function and $F^\sharp$ is a monoidal natural transformation
\[
\begin{tikzcd}[row sep=0pt]
	\typeset\ar[dd, "F"']&
	\frb
		\ar[dd, "{\frb[F]}"']\ar[dr, bend left=15pt, "\pr", ""' name=T]\\
	&&\pposet\\
	\typeset'&
	\frb[\typeset']\ar[ur, bend right=15pt, "\pr'"', "" name=T']
	\ar[from = T, to = T'-|T, twocell, "F^\sharp"']
\end{tikzcd}
\]
that is strict in every respect: all the required coherence diagrams of posets commute on the nose. We denote the category of regular calculi by $\rgcalc$.
\end{definition}

\subsection{Adjoint notation ($\lsh{f}$ and $\ust{f}$) in regular calculi}

It will be convenient to define notation mimicking that in \cref{eqn.subobject_adj} for $\pr$'s action on adjoints in
$\rrel{\cat{R}}$. Given an ajax po-functor $\pr\colon\rrel{\cat{R}}\to\pposet$, we
can take adjoints and use the fundamental lemma (\cref{lemma.fundamental}) to
obtain the diagram below:
\[
\begin{tikzcd}
  \cat{R}\ar[r, "\cong"']\ar[d, equal]\ar[rr, bend left=15pt, "f\mapsto \lsh{f}"]&
  \ladj(\rrel{\cat{R}})\ar[r, "\ladj(\pr)"']\ar[d, "\cong"]&[15pt]
  \ladj(\pposet)\ar[d, "\cong"]\\
  \cat{R}\ar[r, "\cong"]\ar[rr, bend right=15pt, "f\mapsto\ust{f}"']&
  \radj(\rrel{\cat{R}})\op\ar[r, "\radj(\pr)\op"]&
  \radj(\pposet)\op
\end{tikzcd}
\]
That is, for any $f\colon r\to r'$ in $\cat{R}$ we have an adjunction between posets:
\[
  \adj{\pr(r)}{\lsh{f}}{\ust{f}}{\pr(r').}
\]

In particular, since $\frc$ has finite products (denoted using $0$ and $\oplus$), we will speak of projection maps $\pi_i\colon (\Gamma_1\oplus \Gamma_2)\to \Gamma_i$, for $i=1,2$, diagonal maps $\delta_r\colon r \to r \oplus r$, and the unique map $\epsilon_r\colon r \to 0$. Each determines an adjunction as above, e.g.
\[
	\adj{\pr(r_1\times r_2)}{\lsh{(\pi_i)}}{\ust{(\pi_i)}}{\pr(r_i),}
	\qquad
	\adj{\pr(r)}{\lsh{(\delta_r)}}{\ust{(\delta_r)}}{\pr(r\times r),}
	\qquad
	\adj{\pr(r)}{\lsh{(\epsilon_r)}}{\ust{(\epsilon_r)}}{\pr(0).}
\]

\subsection{Regular calculi send objects to meet-semilattices}

If $\pr\colon\frb\to\pposet$ is a regular calculus, i.e.\ an ajax po-functor, then by \cref{cor.meetsl} the poset $\pr(\Gamma)$ is a meet-semilattice for each object $\Gamma\in\ccat{R}$. Its top element and meet are given by the composites of right adjoints shown here:
\begin{equation}\label{eq.def_true_meet}
\begin{tikzcd}[column sep=25pt]
	1\ar[r, shift left=5pt, "\rho"]\ar[r, phantom, "\Rightarrow" yshift=-.6pt]&
	\pr(0)\ar[l, shift left=5pt, "!"]\ar[r, shift left=5pt, "\ust{\epsilon_\Gamma}"]\ar[r, phantom, "\Rightarrow" yshift=-.6pt]&
	\pr(\Gamma)\ar[l, shift left=5pt, "\lsh{\epsilon_\Gamma}"]
\end{tikzcd}
\:\quad \mbox{and}\quad\:
\begin{tikzcd}[column sep=25pt]
	\pr(\Gamma)\times \pr(\Gamma)\ar[r, shift left=5pt, "\rho_{\Gamma,\Gamma}"]\ar[r, phantom, "\Rightarrow" yshift=-.6pt]&
	\pr(\Gamma\oplus\Gamma)\ar[l, shift left=5pt, "\lambda_{\Gamma,\Gamma}"]\ar[r, shift left=5pt, "\ust{\delta_\Gamma}"]\ar[r, phantom, "\Rightarrow" yshift=-.6pt]&
	\pr(\Gamma).\ar[l, shift left=5pt, "\lsh{\delta_\Gamma}"]
\end{tikzcd}
\end{equation}
\section{The predicates functor $\prd\colon\rgcat \to \rgcalc$}

Let $\cat{R}$ be a regular category and let $\ccat{R}\coloneqq\rrel{\cat{R}}$ denote its relations po-category; note that $\ob\cat{R}=\ob\ccat{R}$. We have a counit map $\Prod\colon\frb[\ob\ccat{R}]\to\ccat{R}$ from \cref{eqn.counit_reg_bicat}, and it is a strong monoidal functor. We can compose it with the ``subobjects'' functor $\sub_\cat{R}\coloneqq\ccat{R}(I,-)\colon\ccat{R}\to\pposet$. The result is a po-functor
\begin{equation}\label{eqn.rels_on_objects}
  \sub_{\cat{R}}\Prod\colon\frb[\ob\ccat{R}]\to\pposet
\end{equation}
which assigns to each context $\Gamma$ the poset of \define{predicates} in $\Gamma$. By \cref{lemma.ajax,thm.sub_is_ajax}, the po-functor $\sub_{\cat{R}}\Prod$, is ajax, so $(\ob\cat{R},\sub_{\cat{R}}\Prod)$ is a regular calculus.

\begin{proposition}\label{prop.rels}
The mapping from \cref{eqn.rels_on_objects} extends to a faithful 
functor
\[\prd\colon\rgcat\to\rgcalc.\]
\end{proposition}
\begin{proof}
Given an object $\cat{R}$ of $\rgcat$---that is, given a regular category---we
define $\prd(\cat{R})\coloneqq(\ob\cat{R},\sub_{\cat{R}}\Prod)$. As mentioned
above, $\sub_{\cat{R}}\Prod$ is ajax, so $\prd(\cat{R})$ is a regular calculus.
We need to say how $\prd$ behaves on morphisms. 

A regular functor $\funr{F}\colon\cat{R}\to\cat{R}'$ induces a function $\ob \funr{F}\colon\ob\cat{R}\to\ob\cat{R}'$ and hence a morphism $\ol{\funr{F}}\coloneqq\frb[\ob \funr{F}]\colon\frb[\ob\cat{R}]\to\frb[\ob\cat{R}']$. We need to construct a (strict) monoidal natural transformation $\funr{F}^\sharp\colon\sub_{\cat{R}}\Prod\too(\ol{\funr{F}}\cp\sub_{\cat{R}'}\Prod)$.

Let $\Gamma\in\frb[\ob\cat{R}]$ be a context. The left-hand square in the
following diagram commutes by the naturality of the counit $\Prod[-]$, and we
have a map $\rrel{\funr{F}}(I,-)\colon\rrel{\cat{R}}(I,-)\too\rrel{\cat{R}'}(I,-)$
because $\funr{F}(I)=I$. We define $\funr{F}^\sharp$ to be the composite 2-cell, which we
denote $\sub_\funr{F}\Prod[-]$:
\[
\begin{tikzcd}[row sep=5pt]
	\ob\cat{R}\ar[dd, "\ob \funr{F}"']&
	\frb[\ob\cat{R}]
		\ar[dd, "{\ol{\funr{F}}}"']\ar[r, "\Prod"]&
	\rrel{\cat{R}}\ar[dd, "\rrel{\funr{F}}"']\ar[dr, bend left=15pt, pos=.25, "{\rrel{\cat{R}}(I,-)}", ""' name=T]\\
	&&&[35pt]\pposet\\
	\ob\cat{R}'&
	\frb[\ob\cat{R}']\ar[r, "\Prod"']&
	\rrel{\cat{R}'}\ar[ur, bend right=15pt, pos=.25, "{\rrel{\cat{R}'}(I,-)}"', "" name=T']
	\ar[from=T, to=T'-|T, twocell, "{\rrel{\funr{F}}(I,-)}"]
\end{tikzcd}
\]
Thus we define $\prd$ on morphisms by $\prd(\funr{F})=(\ob \funr{F},\sub_\funr{F}\Prod[-])$; it is
easy to check that $\prd$ preserves identities and compositions. It remains to check
that it is faithful, so let $\funr{F},\funr{G}\colon\cat{R}\to\cat{R}'$ be regular
functors and suppose $\prd(\funr{F})=\prd(\funr{G})$. There is agreement on objects $\ob\funr{F}=\ob
\funr{G}$, so let $f\colon r_1\to r_2$ be a morphism in $\cat{R}$ and consider the its
graph $\hat{f}\coloneqq \pair{\id_r,f}\ss r_1\times r_2$. Write $(r_1,r_2) \coloneqq
(2,\{r_1,r_2\},\cong) \in \frb$. From the fact that $\sub_\funr{F}\Prod[r_1, r_2](\hat{f})=\sub_{\funr{G}}\Prod[r_1,r_2](\hat{f})$ it follows that $\funr{F}(f)=\funr{G}(f)$, completing the proof.
\end{proof}
In \cref{cor.rels_full}, we will show that in fact $\prd$ is also full.

The goal for the rest of this paper is to construct a left adjoint to $\prd$ and prove the essential reflection. Our proof will rely on some properties of regular calculi, in particular that they can be incarnated as a sort of \emph{graphical calculus} for regular logic reasoning.

\goodbreak

\chapter{Graphical regular logic} \label{chap.graphical_reglog}

A key advantage of the regular calculus perspective on regular categories and
regular logic is that it suggests a graphical notation for relations in regular
categories, as well as how they behave under base-change and co-base-change. This is the promised graphical regular
logic.

In this section we develop this graphical formalism, first by giving a
graphical description of the free regular po-category on a set, and then by
defining the notion of graphical term, showing how these represent elements of posets,
and explaining how to reason with them. In subsequent sections, we'll use this
graphical regular logic to prove the main theorem.

\section{Depicting free regular po-categories $\frb$}
Since the po-categories $\frb$ form the foundation of our diagrammatic language for regular
logic, we begin our exploration of graphical regular logic by giving an explicit
description of the objects, morphisms, 2-cells, and composition in $\frb$ in
terms of wiring diagrams.

\begin{notation}
By definition, an object of $\frb$ is simply a context $\Gamma=(\ord{n}
\To{\tau} S \subseteq \typeset)$ of $\frc$. We represent a context graphically by a circle with $n$ ports around the exterior, with $i$th port
annotated by the value $\tau(i)$, and with a white dot at the base annotated by
the remaining elements of the support $S \setminus \im \tau$.%
\footnote{By the idempotence of support contexts \cref{eqn.idem_support}, one may equivalently include the whole support, $S$.}
\begin{equation}\label{eqn.shell_pic}
  \begin{tikzpicture}[inner WD]
    \node[pack,minimum size = 3ex] (rho) {};
    \draw (rho.180) to[pos=1] node[left] (w) {$\tau(1)$} +(180:2pt);
    \draw (rho.135) to[pos=1] node[above] (n) {$\tau(2)$} +(135:2pt);
    \node at ($(rho.45)+(45:6pt)$) {$\ddots$}; 
    \draw (rho.-30) to[pos=1] node[right] (e) {$\tau(n)$} +(-30:2pt); 
    \node[link,fill=white,thin] at (rho.270) {};
    \node[below=-.2 of rho.270] (s) {$S\setminus \im \tau$};
    \pgfresetboundingbox
		\useasboundingbox (s-|w) rectangle (n-|e);
  \end{tikzpicture}
\end{equation}
Our convention will be for the ports to be numbered clockwise from the left of
the circle, unless otherwise indicated, and to omit the white dot if $S = \im \tau$. We refer to such an annotated circle as a \emph{shell}.

As a syntactic shorthand for the shell in \eqref{eqn.shell_pic}, we may combine all the ports and the white dot into a single wire labeled with the context $\Gamma\in\frb$:\quad 
$
  \begin{tikzpicture}[inner WD, pack size=8pt, baseline=(rho.-90)]
    \node[pack] (rho) {};
    \draw (rho.180) to[pos=1] node[left] {$\Gamma$} +(180:2pt);
  \end{tikzpicture}
$.
\end{notation}

\begin{example}
Let $\Gamma=(n,S,\tau)$ be the context with arity $n=3$, support $S = \{w,x,y,z\}
\subseteq \typeset$, and typing $\tau\colon \ord{3} \to S$ given by $\tau(1) = \tau(3)=y$, $\tau(2)=z$. It can be depicted by
the shell
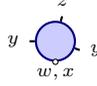
\[
  \begin{tikzpicture}[inner WD]
    \node[pack, minimum size = 3ex] (rho) {};
    \draw (rho.180) to[pos=1] node[left] (w) {$y$} +(180:2pt);
    \draw (rho.75) to[pos=1] node[above] (n) {$z$} +(75:2pt);
    \draw (rho.-20) to[pos=1] node[right] (e) {$y$} +(-20:2pt); 
    \node[link,fill=white,thin] at (rho.270) {};
    \node[below=-.2 of rho.270] (s) {$w,x$};
    \pgfresetboundingbox
		\useasboundingbox (s-|w) rectangle (n-|e);
  \end{tikzpicture}
\]
\end{example}

The hom-posets of $\frb=\rrel{\frc}$ are the subobject posets $
  \frb(\Gamma,\Gamma') = \sub(\Gamma \oplus \Gamma').
$
Explicitly, a morphism $\omega\colon \Gamma_1 \tickar \Gamma_\out$ is a represented by monomorphism
\[
  \Gamma_\omega=(n_\omega \xrightarrow{\tau_\omega} S_\omega \subseteq \typeset)
  \inj \Gamma_1 \oplus \Gamma_\out
\]
in $\frc$, and hence specified by a surjection $\omega$ (see \cref{cor.descriptions}) such that
\[
  \begin{tikzcd}[column sep=large]
    \ord{n}_\omega \ar[r,"\tau_\omega"] 
    & S_\omega \ar[d,phantom, sloped, "\supseteq"] \\
    \ord{n}_1+ \ord{n}_\out \ar[r,"{\tau_1+\tau_\out}"'] \ar[u,two heads, "\omega"] 
    & S_1 \cup S_\out
  \end{tikzcd}
\]
commutes. We depict $\omega$ using a \define{wiring diagram}.  More generally, wiring
diagrams will give graphical representations of morphisms $\omega\colon
\Gamma_1\oplus \dots \oplus \Gamma_k \tickar \Gamma_\out$.  

\begin{notation} \label{notation.wiring_diagrams}
  Suppose we have a morphism $\omega\colon \Gamma_1\oplus \dots \oplus \Gamma_k
  \tickar \Gamma_\out$ in $\frb$. We depict $\omega$ as follows.
  \begin{enumerate}[nolistsep, noitemsep]
    \item Draw the shell for $\Gamma_\out$.
    \item Draw the object $\Gamma_i$, for $i=1,\dots,k$, as non-overlapping
      shells inside the $\Gamma_\out$ shell.
    \item For each $i \in \ord{n}_\omega$, draw a black dot anywhere in the region interior to the $\Gamma_\out$ shell but exterior to all the $\Gamma_i$ shells, and annotate it by the value $\tau_\omega(i)$.
    \item Draw a white dot in the same region, annotated by all elements of
      $S_\omega$ not already present in the diagram.
    \item For each element $(i,j) \in \sum_{i=1,\dots,k, \out} \ord{n}_i$,
      draw a wire connecting the $j$th port on the object $\Gamma_i$ to the black dot $\omega(i,j)$.
  \end{enumerate}
  Just as for objects, we may neglect to draw a white dot when $\im\tau=S$. 
  
  For a more compact notation, we may also neglect to explicitly draw the object
  $\Gamma_\out$, leaving it implicit as comprising the wires left dangling on
  the boundary of the diagram.
\end{notation}

\begin{example} \label{ex.wiring_diagram}
  Here is the set-theoretic data of a morphism $\omega\colon \Gamma_1\oplus \Gamma_2
  \oplus \Gamma_3 \tickar \Gamma_\out$, together with its wiring diagram depiction: 
\begin{align*}
 \parbox{4.15in}{\raggedright
	$\Gamma_1 = (3,\{x,y\},\tau_1)$ where $\tau_1(1)=x,\tau_1(2)=\tau_1(2)=y$;\\
	$\Gamma_2 = (3,\{w,x,y\},\tau_3)$ where $\tau_3(1)=\tau_3(2)=\tau_3(3)=x$;\\
	$\Gamma_3 = (4,\{x,y\},\tau_2)$ where $\tau_2(1)=\tau_2(2)=y, \tau_2(3)=\tau_2(4)=x$;\\
	$\Gamma_\out = (6,\{w,x,y,z\},\tau_\out)$ where
	$\tau_\out(1)=y$,\\\qquad$\tau_\out(2)=\tau_\out(3)=\tau_\out(6)=z$,
	$\tau_\out(4)=\tau_\out(5)=x$;\\
	$\Gamma_\omega = (7,\{v,w,x,y,z\},\tau_r)$ where $\tau_\omega(1)=\tau_\omega(2)=y$,\\
	\qquad$\tau_\omega(3)=\tau_\omega(7)=z$, $\tau_\omega(4)=\tau_\omega(5)=\tau_\omega(6)=x$;}
&\qquad
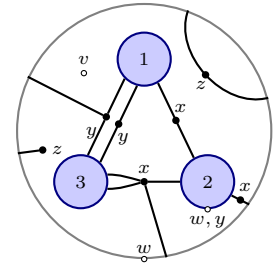
   \begin{tikzpicture}[penetration=0, inner WD,  pack size=20pt, link size=2pt, font=\tiny, scale=2, baseline=(out)]
      \node[pack] at (-1.5,-1) (f) {$3$};
      \node[pack] at (0,1.9) (g) {$1$};
      \node[pack] at (1.5,-1) (h) {$2$};
      \node[outer pack, inner sep=34pt] at (0,.2) (out) {};
      \node[link,label=90:$x$] at ($(f)!.5!(h)$) (link1) {};
      \node[link,label=0:$z$] at (-2.4,-.25) (link2) {};
      \node[link,label distance=-6pt,label=200:$y$] at ($(f.75)!.5!(g.-135)$) (link3) {};
      \node[link,fill=white,thin] at (h.270) {};
      \node[below=-.2 of h.270] {$w,y$};
      \node[link,fill=white,thin] at (out.270) {};
      \node[above=-.2 of out.270] {$w$};
      \node[link,fill=white,thin] at ($(out.160)!.5!(g)$) (dot) {};
      \node[above=-.2 of dot.90] {$v$};
      \begin{scope}[label distance=-6pt]
	\draw (out.280) to (link1);
	\draw (out.190)	to (link2);
	\draw (out.155) to (link3);
	\draw (out.-35)	to node[pos=.5,link,label=10:$x$]{} (h.-30);
	\draw (out.15)	to[out=-165,in=-110] node[pos=.5,link,label
	distance=3pt,label=200:$z$]{} (out.70);
	\draw (f.15) to[out=0,in=165] (link1);
	\draw (f.-15) to[out=0,in=-165] (link1);
	\draw (h.180) to (link1);
	\draw (g.-60) to node[pos=.5,link,label=10:$x$]{} (h.120);
	\draw (f.45) to node[pos=.5,link,label=-10:$y$]{} (g.-105);
	\draw (f.75) to (link3);
	\draw (g.-135) to (link3);
      \end{scope}
    \end{tikzpicture}
\\
\parbox{4.15in}{\raggedright
 $f(1,1)=4$, $f(1,2)=2$, $f(1,3)=1$, $f(2,1)=6$, $f(2,2)=4$,\\
  \qquad $f(2,3)=5$, $f(3,1)=1$, $f(3,2)=2$, $f(3,3)=f(3,4)=6$,\\
  \qquad $f(\out,1)=1$, $f(\out,2)=3$, $f(\out,3)=3$, $f(\out,4)=5$\\
  \qquad $f(\out,5)=6$, $f(\out,6)=7$.}
\end{align*}
\end{example}

\begin{example}\label{ex.no_inner}
Note that we may have $k=0$, in which case there are no inner shells.
For example, the following has $\Gamma_\omega=(2,\{x,y,z,w\},1\mapsto x, 2\mapsto y)$.
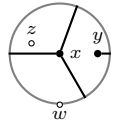
  \[
    \begin{tikzpicture}[inner WD]
      \node[link] (dot270) {};
      \node[right=-.2 of dot270] {$x$};
      \node[outer pack, fit=(dot270), inner sep=12pt] (outer) {};
      \node[link] at ($(outer.0) - (0:5pt)$) (dot0) {};
      \node[above=-.2 of dot0] {$y$};
      \node[link,fill=white,thin] at (outer.270) {};
      \node[below=-.2 of outer.270] (w) {$w$};
      \node[link,fill=white,thin] at ($(outer.160) - (160:8pt)$) (dot) {};
      \node[above=-.2 of dot.90] {$z$};
      \draw (outer.0) -- (dot0);
      \draw (outer.70) -- (dot270);
      \draw (outer.180) -- (dot270);
      \draw (outer.300) -- (dot270);
      \pgfresetboundingbox
    	\useasboundingbox (w.north) rectangle (outer.north);
    \end{tikzpicture}
  \]
\end{example}

\begin{remark}\label{rem.multiple}
When multiple wires meet at a point, our convention will be to draw a dot iff
the number of wires is different from two.
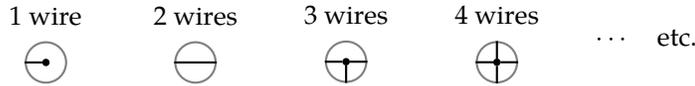
\[
\begin{tikzpicture}[unoriented WD, font=\small]
	\node["1 wire"] (P1) {
  \begin{tikzpicture}[inner WD, surround sep=4pt]
    \node[link] (dot) {};
    \node[outer pack, fit=(dot)] (outer) {};
    \draw (dot) -- (outer.west);
  \end{tikzpicture}	
	};
	\node[right=4 of P1, "2 wires"] (P2) {
  \begin{tikzpicture}[inner WD, surround sep=4pt]
    \node[link, white] (dot) {};
    \node[outer pack, fit=(dot)] (outer) {};
    \draw (outer.west) -- (outer.east);
  \end{tikzpicture}	
	};
	\node[right=4 of P2, "3 wires"] (P3) {
  \begin{tikzpicture}[inner WD, surround sep=4pt]
    \node[link] (dot) {};
    \node[outer pack, fit=(dot)] (outer) {};
    \draw (dot) -- (outer.west);
    \draw (dot) -- (outer.east);
    \draw (dot) -- (outer.south);
  \end{tikzpicture}	
	};
	\node[right=4 of P3, "4 wires"] (P4) {
  \begin{tikzpicture}[unoriented WD, surround sep=4pt, font=\tiny]
    \node[link] (dot) {};
    \node[outer pack, fit=(dot)] (outer) {};
    \draw (dot) -- (outer.west);
    \draw (dot) -- (outer.east);
    \draw (dot) -- (outer.south);
    \draw (dot) -- (outer.north);
  \end{tikzpicture}	
	};
	\node[right=5 of P4, "$\cdots$\quad etc."] (etc){};
  \pgfresetboundingbox
	\useasboundingbox ($(P1.north west)+(0,10pt)$) rectangle ($(etc)+(0,-5pt)$);
\end{tikzpicture}  
\]
When wires intersect and we do not draw a black dot, the intended interpretation is that the wires are \emph{not connected}:\quad $
\begin{tikzpicture}[unoriented WD, font=\small, baseline=(P1.-10)]
	\node (P1) {
  \begin{tikzpicture}[unoriented WD, link size=3pt, surround sep=2pt, font=\tiny]
    \node[link] (dot) {};
    \node[outer pack, fit=(dot)] (outer) {};
    \draw (dot) -- (outer.west);
    \draw (dot) -- (outer.east);
    \draw (dot) -- (outer.south);
    \draw (dot) -- (outer.north);
  \end{tikzpicture}
  };
  \node[right=1 of P1] (P2)	{
  \begin{tikzpicture}[unoriented WD, link size=3pt, surround sep=2pt, font=\tiny]
    \node[link, white] (dot) {};
    \node[outer pack, fit=(dot)] (outer) {};
    \draw (outer.east) -- (outer.west);
    \draw (outer.south) -- (outer.north);
  \end{tikzpicture}
  };
  \node at ($(P1)!.5!(P2)$) {$\neq$};
\end{tikzpicture}
$. Of course this is bound to happen when the graph is non-planar.
\end{remark}

The following examples give a flavor of how composition, monoidal product, and
2-cells are represented using this graphical notation.

\begin{example}[Composition as substitution]\label{ex.comp_as_subst}
  Composition of morphisms is described by \define{nesting} of wiring diagrams.
  Let $\omega'\colon \Gamma'\tickar \Gamma_1$ and $\omega\colon \Gamma_1 \tickar \Gamma_\out$ 
  be morphisms in $\frb$. Then the composite relation $\omega'\cp \omega\colon \Gamma'
  \tickar \Gamma_\out$ is given by
  \begin{enumerate}[nolistsep, noitemsep]
    \item drawing the wiring diagram for $\omega'$ inside the inner circle of the diagram for $\omega$, 
    \item erasing the object $\Gamma_1$, 
    \item amalgamating any connected black dots into a single black dot, and
    \item removing all components not connected to the objects $\Gamma'$ or
      $\Gamma_\out$, and adding a single white dot annotated by the set
      containing all elements of $\typeset$ present in these components, but not
      present elsewhere in the diagram.
  \end{enumerate}
  Note that step 3 corresponds to taking pullbacks in $\frc$ (pushouts in
  $\finset$), while step 4 corresponds to epi-mono factorization. 
  
  As a shorthand for composition, we simply draw one wiring diagram directly
  substituted into another, as per step 1. For example, we have
  
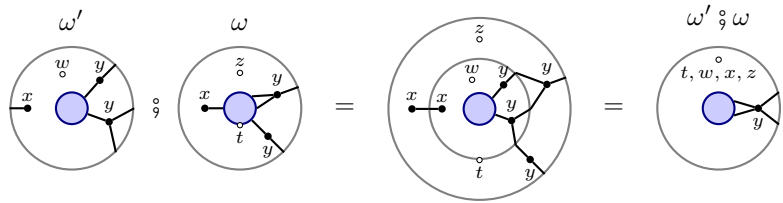
\[
    \begin{tikzpicture}[unoriented WD, font=\small]
      \node["$\omega'$"] (P1a) {
	\begin{tikzpicture}[inner WD]
	  \node[pack] (a) {};
	  \node[outer pack, inner sep=10pt, fit=(a)] (outer) {};
	  \node[link] (link1) at ($(a.west)!.6!(outer.west)$) {};
	  \node[link] (link2) at ($(a.45)!.5!(outer.45)$) {};
	  \node[link] (link3) at ($(a.-20)!.5!(outer.-20)$) {};
	  \node[link,fill=white,thin] at ($(a.100) + (110:7pt)$) (dot) {};
          \node[above=-.3 of link1] {$x$};
          \node[above=-.3 of link2] {$y$};
          \node[above=-.2 of link3] {$y$};
          \node[above=-.3 of dot] {$w$};
	  \draw (outer.west) -- (link1);
	  \draw (a.40) -- (link2);
	  \draw (link2) -- (outer.45);
	  \draw (a.-20) -- (link3);
	  \draw (link3) -- (outer.0);
	  \draw (link3) -- (outer.-45);
	\end{tikzpicture}	
      };
      \node[right=1 of P1a, "$\omega$"] (P1b) {
	\begin{tikzpicture}[inner WD]
	  \node[pack] (c) {};
	  \node[outer pack, inner sep=10pt, fit=(c)] (outer2) {};
	  \node[link] (link4) at ($(c.west)!.4!(outer2.west)$) {};
	  \node[link] (link5) at ($(c.20)!.5!(outer2.20)$) {};
	  \node[link] (link6) at ($(c.-45)!.5!(outer2.-45)$) {};
          \node[link,fill=white,thin] at (c.270) (dotc) {};
	  \node[link,fill=white,thin] at ($(c.90) + (90:7pt)$) (dot) {};
          \node[above=-.2 of link4] {$x$};
          \node[above=-.2 of link5] {$y$};
          \node[below=-.2 of link6] {$y$};
          \node[below=-.3 of dotc] {$t$};
          \node[above=-.3 of dot] {$z$};
	  \draw (c.west) -- (link4);
	  \draw (c.45) -- (link5);
	  \draw (c.0) -- (link5);
	  \draw (link5) -- (outer2.20);
	  \draw (c.-45) -- (link6);
	  \draw (link6) -- (outer2.-45);
	\end{tikzpicture}	
      };
      \node[right=3 of P1b] (P2) {
	\begin{tikzpicture}[inner WD]
	  \node[pack] (a) {};
	  \node[outer pack, inner sep=7pt, fit=(a)] (c) {};
	  \node[outer pack, inner sep=5pt, fit=(c)] (outer2) {};
	  \node[link] (link1) at ($(a.west)!.6!(c.west)$) {};
	  \node[link] (link2) at ($(a.45)!.5!(c.45)$) {};
	  \node[link] (link3) at ($(a.-20)!.5!(c.-20)$) {};
	  \node[link,fill=white,thin] at ($(a.100) + (110:5pt)$) (dot) {};
          \node[above=-.3 of link1] {$x$};
          \node[above=-.3 of link2] {$y$};
          \node[above=-.2 of link3] {$y$};
          \node[above=-.3 of dot] {$w$};
	  \draw (c.west) -- (link1);
	  \draw (a.40) -- (link2);
	  \draw (link2) -- (c.45);
	  \draw (a.-20) -- (link3);
	  \draw (link3) -- (c.0);
	  \draw (link3) -- (c.-45);
	  \node[link] (link4) at ($(c.west)!.4!(outer2.west)$) {};
	  \node[link] (link5) at ($(c.20)!.5!(outer2.20)$) {};
	  \node[link] (link6) at ($(c.-45)!.5!(outer2.-45)$) {};
          \node[link,fill=white,thin] at (c.270) (dotc) {};
	  \node[link,fill=white,thin] at ($(c.90) + (90:7pt)$) (dot) {};
          \node[above=-.2 of link4] {$x$};
          \node[above=-.2 of link5] {$y$};
          \node[below=-.2 of link6] {$y$};
          \node[below=-.3 of dotc] {$t$};
          \node[above=-.3 of dot] {$z$};
	  \draw (c.west) -- (link4);
	  \draw (c.45) -- (link5);
	  \draw (c.0) -- (link5);
	  \draw (link5) -- (outer2.20);
	  \draw (c.-45) -- (link6);
	  \draw (link6) -- (outer2.-45);
	\end{tikzpicture}	
      };
      \node[right=3 of P2, "$\omega'\cp\omega$"] (P3) {
	\begin{tikzpicture}[inner WD]
	  \node[pack] (c) {};
	  \node[outer pack, inner sep=10pt, fit=(c)] (outer2) {};
	  \node[link] (link) at ($(c.0)!.5!(outer2.0)$) {};
	  \node[link,fill=white,thin] at ($(c.90) + (90:12pt)$) (dot) {};
          \node[below=-.2 of link] {$y$};
          \node[below=-.3 of dot] {$t,w,x,z$};
	  \draw (c.25) -- (link);
	  \draw (c.-25) -- (link);
	  \draw (link) -- (outer2.15);
	  \draw (link) -- (outer2.-15);
	\end{tikzpicture}	
      };
      \node (P1) at ($(P1a.east)!.5!(P1b.west)$) {$\cp$};
      \node at ($(P1b.east)!.5!(P2.west)$) {$=$};
      \node at ($(P2.east)!.5!(P3.west)$) {$=$};
    \end{tikzpicture}
  \]

  For the more general $k$-ary or operadic case, we may obtain the composite
  \[
    (\Gamma_1 \oplus \dots \oplus \Gamma_{i-1} \oplus \omega' \oplus
    \Gamma_{i+1} \oplus \dots \oplus \Gamma_k) \cp \omega
  \] 
  of any two morphisms $\omega'\colon \Gamma'_1\oplus \dots \oplus \Gamma'_k \tickar
  \Gamma_i$ and $\omega\colon \Gamma_1\oplus \dots \oplus \Gamma_k \tickar
  \Gamma_\out$ by substituting the wiring diagram for $\omega'$ into the $i$th inner
  circle of the diagram for $\omega$, and following a procedure similar to that in \cref{ex.comp_as_subst}.
\end{example}

\begin{example}[Monoidal product as juxtaposition]
  The monoidal product of two morphisms in $\frb$ is simply their juxtaposition,
  merging the labels on the floating white dots as appropriate. For
  example, leaving off labels, we might have:
  
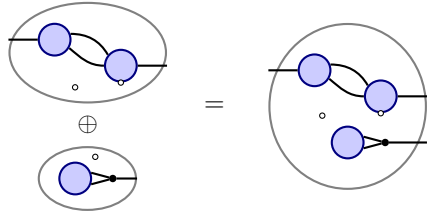
\[
    \begin{tikzpicture}[unoriented WD, font=\small]
      \node (P1a) {
	\begin{tikzpicture}[inner WD]
	  \node[pack] (a) {};
	  \node[pack, below right=.1 and 2 of a] (b) {};
          \node[link,fill=white,thin] at (b.270) {};
	  \node[outer pack, fit=(a) (b)] (outer) {};
          \node[link,fill=white,thin] at ($(outer.250) - (250:6pt)$) (dot) {};
	  \draw (a.180) -- (a.180-|outer.west);
	  \draw (a.20) to[out=0, in=120] (b.140);
	  \draw (a.-30) to[out=-40, in=180] (b.180);
	  \draw (b.east) -- (b.east-|outer.east);
	\end{tikzpicture}
      };
      \node[below=1 of P1a] (P1b) {
	\begin{tikzpicture}[inner WD]
	  \node[pack, below=2 of $(a)!.5!(b)$] (c) {};
	  \node[link, right=.8 of c] (link1) {};
	  \node[outer pack, fit=(link1) (c)] (outer) {};
          \node[link,fill=white,thin] at ($(outer.70) - (70:4pt)$) (dot) {};
	  \draw (c.20) -- (link1);
	  \draw (c.-20) -- (link1);
	  \draw (link1) -- (link1-|outer.east);
	\end{tikzpicture}	
      };
      \node (P1) at ($(P1a.south)!.5!(P1b.north)$) {$\oplus$};
      \node[right=5 of $(P1a.north)!.5!(P1b.south)$] (e) {$=$};
      \node[right=1 of e] (P2) {
	\begin{tikzpicture}[inner WD]
	  \node[pack] (a) {};
	  \node[pack, below right=.1 and 2 of a] (b) {};
          \node[link,fill=white,thin] at (b.270) {};
	  \node[pack, below=2 of $(a)!.5!(b)$] (c) {};
	  \node[link, right=.8 of c] (link1) {};
	  \node[outer pack, fit=(a) (b) (c)] (outer) {};
          \node[link,fill=white,thin] at ($(outer.200) - (200:20pt)$) (dot) {};
	  \draw (a.180) -- (a.180-|outer.west);
	  \draw (a.20) to[out=0, in=120] (b.140);
	  \draw (a.-30) to[out=-40, in=180] (b.180);
	  \draw (b.east) -- (b.east-|outer.east);
	  \draw (c.20) -- (link1);
	  \draw (c.-20) -- (link1);
	  \draw (link1) -- (link1-|outer.east);
	\end{tikzpicture}	
      };
  \pgfresetboundingbox
	\useasboundingbox (P1b.190-|P1a.west) rectangle (P2.east|-P1a.150);
    \end{tikzpicture}
  \]
\end{example}

\begin{example}[2-cells as breaking wires and removing white dots]\label{lem.breaking}
  Let $\omega,\omega'\colon \Gamma_1\oplus \dots \oplus \Gamma_k \tickar \Gamma_\out$ be morphisms in
  $\frb=\rrel{\frc}$. By definition, there exists a 2-cell $\omega\leq\omega'$ if there is a
  monomorphism $m\colon \Gamma_\omega \inj \Gamma_{\omega'}$ in $\frc$ such that
  $m \cp \omega' = \omega$ holds in $\frb$.
  By \cref{cor.descriptions}, this data consists of a surjection of finite sets
  $m\colon \ord{n}_\omega' \to \ord{n}_\omega$ and an inclusion $S_{\omega'} \subseteq S_{\omega}$. In
  diagrams, the former means 2-cells may break wires, and the latter means they
  may remove annotations from the inner white dot (or remove it completely). For example, we have 2-cells: \qquad
  $
    \begin{aligned}
      \begin{tikzpicture}[unoriented WD, font=\small]
	\node (P3) {
	  \begin{tikzpicture}[inner WD]
	    \node[link] (dot) {};
	    \node[outer pack, surround sep=3pt, fit=(dot)] (outer) {};
	    \draw (dot) -- (outer.0);
	    \draw (dot) -- (outer.120);
	    \draw (dot) -- (outer.240);
	  \end{tikzpicture}	
	};
	\node[right=2 of P3] (P4) {
	  \begin{tikzpicture}[inner WD]
	    \node[link, white] (fake dot) {};
	    \node[outer pack, surround sep=3pt, fit=(fake dot)] (outer) {};
	    \node[link] (dot) at ($(outer.0)+(0:-5pt)$) {};
	    \draw (dot) -- (outer.0);
	    \draw (outer.120) to[out=300, in=60] (outer.240);
	  \end{tikzpicture}	
	};	
	\node at ($(P3.east)!.5!(P4.west)$) {$\leq$};
      \end{tikzpicture}
    \end{aligned}
\quad    \mbox{and}\quad
    \begin{aligned}
      \begin{tikzpicture}[unoriented WD, font=\small]
	\node (P1) {
	  \begin{tikzpicture}[inner WD]
	    \node[link, fill=white] (dot) {};
	    \node[outer pack, surround sep=3pt, fit=(dot)] (outer) {};
	  \end{tikzpicture}
	};
	\node[right=2 of P1] (P2) {
	  \begin{tikzpicture}[inner WD]
	    \node[link, white] (dot) {};
	    \node[outer pack, surround sep=3pt, fit=(dot)] (outer) {};
	  \end{tikzpicture}
	};
	\node at ($(P1.east)!.5!(P2.west)$) {$\leq$};
      \end{tikzpicture}
    \end{aligned}
$.
\end{example}

\section{Graphical terms}

Given a regular calculus $\pr\colon\frb\to\pposet$, we give a graphical representations of its predicates, i.e.\ the elements in $\pr(\Gamma)$ for various contexts $\Gamma\in\frb$. Here's how it works.

\begin{definition}
  A \define{$\pr$-graphical term} $(\theta_1,\dots,\theta_k;\omega)$ in an ajax
  po-functor $\pr\colon \frb \to \pposet$ is a morphism $\omega\colon \Gamma_1\oplus
  \dots \oplus \Gamma_k \tickar \Gamma_\out$ in $\frb$ together with, for each
  $i = 1,\dots,k$, an element $\theta_i \in \pr(\Gamma_i)$.

  We say that the graphical term $t = (\theta_1,\dots,\theta_k; \omega)$
  \define{represents} the poset element 
  \[
    \church{ t } \coloneqq (\pr(\omega)\cp \rho)(\theta_1,\dots,\theta_k)
    \in \pr(\Gamma_\out)
  \]
where $\rho$ is the $k$-ary laxator. If $t$ and $t'$ are graphical terms, we write $t \vdash t'$ when $\church{ t
} \vdash \church{ t'}$, and $t =t'$ when $\church{t}=\church{t'}$.
\end{definition}

\begin{notation}
We draw a graphical term $(\theta_1,\dots,\theta_k; \omega)$ by annotating the $i$th inner shell with its corresponding poset element $\theta_i$. In the case
that $k=1$ and $\omega$ is the identity morphism, we may simply draw the object
$\Gamma_1$ annotated by $\theta_1$:
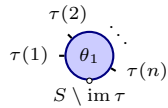
\[
  \begin{tikzpicture}[inner WD, baseline=(wdot)]
    \node[pack,minimum size = 3ex] (rho) {$\theta_1$};
    \draw (rho.180) to[pos=1] node[left] {$\tau(1)$} +(180:2pt);
    \draw (rho.135) to[pos=1] node[above] {$\tau(2)$} +(135:2pt);
    \node at ($(rho.45)+(45:6pt)$) {$\ddots$}; 
    \draw (rho.-30) to[pos=1] node[right] {$\tau(n)$} +(-30:2pt); 
    \node[link,fill=white,thin] (wdot) at (rho.270) {};
    \node[below=-.2 of rho.270] {$S\setminus \im \tau$};
  \end{tikzpicture}
\]
\end{notation}

\begin{example}
  Recall that we have a diagonal map $\delta\colon \Gamma \to \Gamma\oplus \Gamma$ in
  $\frc\ss\frb$. Given $\theta \in \pr(\Gamma)$, the element $\lsh{(\delta)}(\varphi)
  \in \pr(\Gamma\oplus \Gamma)$ is represented by the graphical term
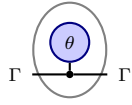
\[
\begin{tikzpicture}[inner WD, baseline=(dot)]
	\node[pack] (phi) {$\theta$};
	\node[link, below=3pt of phi] (dot) {};
  \node[outer pack, fit=(phi) (dot)] (outer) {};
	\draw (phi) -- (dot);
	\draw (dot) to[pos=1] node[left] {$\Gamma$} (dot-|outer.west);
	\draw (dot) to[pos=1] node[right] {$\Gamma$} (dot-|outer.east);
\end{tikzpicture}
\]
\end{example}

\begin{example}
  When $\typeset=\varnothing$ is empty, $\frc[\varnothing]$ is the
  terminal category. By \cref{prop.adjoint_monoids}, an ajax po-functor
  $\pr\colon \frb \to \pposet$ then simply chooses a $\wedge$-semilattice $\pr(0)$. The
  po-category $\rela{\pr}$ is that $\wedge$-semilattice considered as a one
  object po-category: it has a unique object whose poset of endomorphisms is $P(0)$. The diagrammatic
  language has no wires, since there is only the monoidal unit in
  $\frc[\varnothing]$. The semantics of an arbitrary graphical term $(\theta_1,\ldots,\theta_k;\id)$ is simply the meet $\theta_1\wedge\cdots\wedge\theta_k$.
\end{example}

\begin{remark}
  Graphical terms are an alternate syntax for regular logic. While we will not
  dwell on the translation, a graphical term $(\theta_1,\dots,\theta_k; \omega)$
  represents the regular formula
  \[
    \bigexists_{\substack{i \in \ord{n}_j \\ j \in \{1,\dots,k,\omega\}}} x_{ij} .\bigwedge_{j\in\{
    1,\dots k\}} \theta_k(x_{ij}) \quad\wedge \bigwedge_{\substack{i \in \ord{n}_j \\
  j\in\{1,\dots,k,\out\}}} \big(x_{ij} = x_{\omega(i)j}\big).
  \]
  This formula creates a variable of each element of $\ord{n}_j$, where
  $j\in\{1,\dots,k,\out,\omega\}$, equates any two variables with the same image under
  $\omega$, takes the conjunction with all the formulas $\theta_j$, and the
  existentially quantifies over all variables except those in $\Gamma_\out$.
  In particular, if we were to take $\omega\colon\Gamma_1\oplus\Gamma_2\oplus\Gamma_3\tickar\Gamma_\out$ as in \cref{ex.wiring_diagram}, the
  resulting graphical term would simplify to the formula
  \[
    \psi(y,z,z',x,x',z'') = \exists \tilde{x},\tilde{y}.\theta_1(\tilde{x},\tilde{y},y) \wedge
    \theta_2(\tilde{x},x,x') \wedge \theta_3(y,\tilde{y},x',x') \wedge (z=z') \wedge
    (z''=z'').
  \]
\end{remark}

\begin{remark}
  Note that $\pposet$ is a subcategory of $\CCat{Cat}$. This allows us to take the
  monoidal Grothendieck construction $\int \pr$ of $\pr\colon\frb\to\pposet$, \cite{moeller2018monoidal}. A
  $\pr$-graphical term is an object in the comma category $\int\pr
  \mathord{\downarrow} \frb$. This perspective lends structure to the various
  operations on diagrams belows; we, however, pursue it no further here.
\end{remark}

\section{Reasoning with graphical terms}
The following basic rules for reasoning with diagrams express the (2-)functoriality
and monoidality of $\pr$.

\begin{proposition} \label{prop.diagrams_basic}
  Let $(\theta_1,\dots,\theta_k;\omega)$ be a graphical term, where $\theta_i \in
  \pr(\Gamma_i)$.
  \begin{enumerate}[label=(\roman*)]
    \item (Monotonicity) Suppose $\theta_i \vdash \theta_i'$ for
      some $i$. Then 
      \[
	\church{(\theta_1,\dots,\theta_i,\dots,\theta_k; \omega)} \vdash \church{(\theta_1,\dots,\theta_i',\dots,\theta_k; \omega)}.
      \]
    \item (Breaking) Suppose $\omega \leq \omega'$ in $\frb$. Then
      \[
	\church{(\theta_1,\dots,\theta_k; \omega)} \vdash
	\church{(\theta_1,\dots,\theta_k; \omega')}.
      \]
    \item (Nesting) Suppose $\theta_i = \church{(\theta'_1,\dots,\theta'_\ell; \omega')}$ for some $i$. Then 
\begin{multline*}
	\church{(\theta_1,\dots,\theta_k; \omega)}
	=
	\church{(\theta_1,\dots,\theta_{i-1},\theta'_1, \dots,
	\theta'_\ell,\theta_{i+1},\dots,\theta_k;\\(\Gamma_1\oplus\dots\oplus
	\Gamma_{i-1}\oplus \omega' \oplus \Gamma_{i+1} \oplus \dots \oplus
	\Gamma_k)\cp \omega)}.
\end{multline*}
  \end{enumerate}
\end{proposition}
\begin{proof}
  \begin{enumerate}[label=(\roman*)]
    \item This is the monotonicity of the map $\pr(\omega)\cp \rho$.
    \item This is the 2-functoriality of $\pr$.
    \item This follows from the monoidality and 1-functoriality of $\pr$. In
      particular, it is the commutativity of the following diagram. Using the braiding we can assume without loss of generality that $i=k$.
      \[
	\begin{tikzcd}
	  \prod_{j=1}^{k-1}\pr(\Gamma_j) \times \prod_{j=1}^\ell \pr(\Gamma'_j) 
	  \ar[d,"\id \times \rho"'] \ar[dr,"\rho"] 
	  \\
	  \prod_{j=1}^{k-1}\pr(\Gamma_j) \times \pr\left(\bigoplus_{j=1}^\ell\Gamma'_j\right) \ar[r, "\rho"] \ar[d,"\prod_\pr(\Gamma_j)\times \pr(\omega)"']
	  &
	  \pr\left(\bigoplus_{j=1}^{k-1} \Gamma_j \oplus \bigoplus_{j=1}^\ell \Gamma'_j \right) \ar[d, "\pr(\bigoplus_{j=1}^{k-1}\Gamma_j+\omega')"'] \ar[dr, "\pr((\bigoplus_{j=1}^{k-1}\Gamma_j+\omega')\cp \omega)"] 
	  \\
	  \prod_{j=1}^k\pr(\Gamma_j)\ar[r, "\rho"'] 
	  & 
	  \pr\left(\bigoplus_{j=1}^k\Gamma_j\right) \ar[r, "\pr(\omega)"'] 
	  &[5ex]
	  \pr(\Gamma_\out)
	\end{tikzcd}
      \]
      The upper triangle commutes by coherence laws for $\rho$, the square
      commutes by naturality of $\rho$, and the right hand triangle commutes by
      functoriality of $\pr$. \qedhere
  \end{enumerate}
\end{proof}

\begin{example}
  \cref{prop.diagrams_basic} is perhaps more quickly grasped through a graphical example
  of these facts in action. Suppose we have the entailment
  
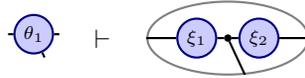
\[
    \begin{tikzpicture}[unoriented WD, font=\small, pack size=7pt, baseline=(P1.south)]
      \def\angle{-65};
      \node (P1) {
	\begin{tikzpicture}[inner WD]
	  \node[pack] (theta) {$\theta_1$};
	  \draw (theta.180) -- +(180:2pt);
	  \draw (theta.0) -- +(0:2pt);
	  \draw (theta.\angle) -- +(\angle:2pt);
	\end{tikzpicture}	
      };
      \node[right=3 of P1] (P2) {
	\begin{tikzpicture}[inner WD]
	  \node[pack] (xi1) {$\xi_1$};
	  \node[pack, right=1 of xi1] (xi2) {$\xi_2$};
	  \node[link] at ($(xi1.east)!.5!(xi2.west)$) (dot) {};
	  \node[outer pack, fit=(xi1) (xi2)] (outer) {};
	  \draw (outer) -- (xi1.west);
	  \draw (xi1.east) -- (dot);
	  \draw (dot) -- (xi2);
	  \draw (xi2) -- (outer);
	  \draw (dot) -- (outer.\angle);
	\end{tikzpicture}		
      };
      \node at ($(P1.east)!.5!(P2.west)$) {$\vdash$};
    \end{tikzpicture}
    \]
    Then using monotonicity, nesting, and then breaking we can deduce the
    entailment
    
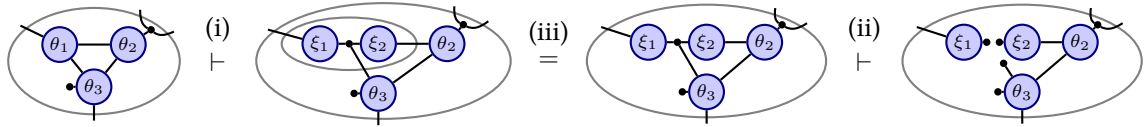
\[
      \begin{tikzpicture}[unoriented WD, font=\small, pack size=5pt, baseline=(P1.195)]
	\node (P1) {
	  \begin{tikzpicture}[inner WD]
	    \node[pack] (theta1) {$\theta_1$};
	    \node[pack, right=1.5 of theta1] (theta2) {$\theta_2$};
	    \node[pack] at ($(theta1)!.5!(theta2)+(0,-2)$) (theta3) {$\theta_3$};
	    \node[outer pack, inner xsep=3pt, inner ysep=1pt, fit=(theta1) (theta2) (theta3.-30)] (outer) {};
	    \node[link] at ($(theta2.30)!.5!(outer.30)$) (dot1) {};
	    \node[link,  left=.1 of theta3.west] (dot2) {};
	    \draw (theta2) -- (dot1);
	    \draw[shorten >= -2pt] (dot1) to[bend right] (outer.20);
	    \draw[shorten >= -2pt] (dot1) to[bend left] (outer.45);
	    \draw (dot2) -- (theta3);
	    \draw[shorten >= -2pt] (theta1) -- (outer);
	    \draw[shorten >= -2pt] (theta3) -- (outer);
	    \draw (theta1) -- (theta3);
	    \draw (theta1) -- (theta2);
	    \draw (theta2) -- (theta3);
	  \end{tikzpicture}
	};
	\node[right=2.5 of P1] (P2) {
	  \begin{tikzpicture}[inner WD]
	    \node[pack] (xi1) {$\xi_1$};
	    \node[pack, right=1 of xi1] (xi2) {$\xi_2$};
	    \node[link] at ($(xi1.east)!.5!(xi2.west)$) (dot) {};
	    \draw (xi1.east) -- (dot);
	    \draw (dot) -- (xi2);
	  \node[outer pack, inner xsep=0, inner ysep=0, fit=(xi1) (xi2)] (outerxi) {};
	    \node[pack, right=1.5 of xi2] (theta2) {$\theta_2$};
	    \node[pack, below=.6 of xi2] (theta3) {$\theta_3$};
	    \node[outer pack, inner xsep=3pt, inner ysep=2pt, fit=(xi1.west) (theta2) (theta3.-10)] (outer) {};
	    \node[link] at ($(theta2.30)!.5!(outer.30)$) (dot1) {};
	    \node[link,  left=.1 of theta3.west] (dot2) {};
	    \draw (theta2) -- (xi2.east);
	    \draw (dot) -- (theta3);
	    \draw[shorten >= -2pt] (dot1) to[bend right] (outer.20);
	    \draw[shorten >= -2pt] (dot1) to[bend left] (outer.35);
	    \draw (dot2) -- (theta3);
	    \draw[shorten >= -2pt] (xi1) -- (outer);
	    \draw[shorten >= -2pt] (theta3.south) -- (theta3|-outer.south);
	    \draw (theta2) -- (theta3);
	    \draw (theta2) -- (dot1);
	  \end{tikzpicture}
	};
	\node[right=2.5 of P2] (P3) {
	  \begin{tikzpicture}[inner WD]
	    \node[pack] (xi1) {$\xi_1$};
	    \node[pack, right=1 of xi1] (xi2) {$\xi_2$};
	    \node[link] at ($(xi1.east)!.5!(xi2.west)$) (dot) {};
	    \draw (xi1.east) -- (dot);
	    \draw (dot) -- (xi2);
	    \node[pack, right=1 of xi2] (theta2) {$\theta_2$};
	    \node[pack, below=.6 of xi2] (theta3) {$\theta_3$};
	    \node[outer pack, inner xsep=3pt, inner ysep=1pt, fit=(theta1) (theta2) (theta3.-20)] (outer) {};
	    \node[link] at ($(theta2.30)!.5!(outer.30)$) (dot1) {};
	    \node[link,  left=.1 of theta3.west] (dot2) {};
	    \draw (theta2) -- (xi2.east);
	    \draw (dot) -- (theta3);
	    \draw[shorten >= -2pt] (dot1) to[bend right] (outer.20);
	    \draw[shorten >= -2pt] (dot1) to[bend left] (outer.35);
	    \draw (dot2) -- (theta3);
	    \draw[shorten >= -2pt] (xi1) -- (outer);
	    \draw[shorten >= -2pt] (theta3) -- (outer);
	    \draw (theta2) -- (theta3);
	    \draw (theta2) -- (dot1);
	  \end{tikzpicture}
	};
	\node[right=2.5 of P3] (P4) {
	  \begin{tikzpicture}[inner WD]
	    \node[pack] (xi1) {$\xi_1$};
	    \node[pack, right=1 of xi1] (xi2) {$\xi_2$};
	    \node[link] at ($(xi1.east)!.5!(xi2.west)+(-.3,0)$) (dota) {};
	    \node[link] at ($(xi1.east)!.5!(xi2.west)+(.3,0)$) (dotb) {};
	    \node[link] at ($(xi1.east)!.5!(xi2.west)+(.5,-1)$) (dotc) {};
	    \draw (xi1.east) -- (dota);
	    \draw (dotb) -- (xi2);
	    \node[pack, right=1 of xi2] (theta2) {$\theta_2$};
	    \node[pack, below=.6 of xi2] (theta3) {$\theta_3$};
	    \node[outer pack, inner xsep=3pt, inner ysep=1pt, fit=(theta1) (theta2) (theta3.-20)] (outer) {};
	    \node[link] at ($(theta2.30)!.5!(outer.30)$) (dot1) {};
	    \node[link,  left=.1 of theta3.west] (dot2) {};
	    \draw (theta2) -- (xi2.east);
	    \draw (dotc) -- (theta3);
	    \draw[shorten >= -2pt] (dot1) to[bend right] (outer.20);
	    \draw[shorten >= -2pt] (dot1) to[bend left] (outer.35);
	    \draw (dot2) -- (theta3);
	    \draw[shorten >= -2pt] (xi1) -- (outer);
	    \draw[shorten >= -2pt] (theta3) -- (outer);
	    \draw (theta2) -- (theta3);
	    \draw (theta2) -- (dot1);
	  \end{tikzpicture}
	};
	\node (imp1) at ($(P1.east)!.5!(P2.west)$) {$\vdash$};
	\node[above=-.5 of imp1] {(i)};
	\node (imp2) at ($(P2.east)!.5!(P3.west)$) {$=$};
	\node[above=-.5 of imp2] {(iii)};
	\node (imp3) at ($(P3.east)!.5!(P4.west)$) {$\vdash$};
	\node[above=-.5 of imp3] {(ii)};
      \end{tikzpicture}
    \]
    We'll see many further examples of such reasoning in the subsequent sections
    of this paper, as we prove that we can construct a regular category from a
    regular calculus.
  \end{example}

\begin{example}\label{lem.combining}
  The nesting rule in \cref{prop.diagrams_basic} has two particularly important cases. The first occurs when we
  consider wiring diagrams themselves as poset elements. More precisely, if
  $f\colon \Gamma_1 \to \Gamma_\out$ is a morphism in $\frc$, and
  $\hat{f}\coloneqq \pair{\id_{\Gamma_1},f}$ is its graph, then taking $i=k=1$,
  $\ell=0$, $\theta = \church{(;\hat{f})}$, $\omega =
  \Gamma_\out$ (the identity) and
  $\omega'=\hat{f}$ in (iii) gives $ \church{(\theta;\Gamma_\out)} =
  \church{(;\hat{f})}$. Note that this
  equates a graphical term with inner object $\Gamma_\out$ and annotation
  $\theta$ with a term that has no inner object at all; see e.g.\ \cref{ex.no_inner}.

  The second important case is that of `exterior AND'. If we take $i=k=1$,
  $\ell=2$, and $\omega = \omega'= \Gamma_1\tens \Gamma_2$, then 
  $
    \church{(\theta'_1,\theta'_2;\Gamma_1\tens \Gamma_2)}
    =
    \church{(\rho(\theta'_1,\theta'_2);\Gamma_1 \tens \Gamma_2)}$.
In pictures, this means we can take any two circles, say $\theta_1\in
\pr(\Gamma_1)$ and $\theta_2\in \pr(\Gamma_2)$, and merge them, labelling
the merged circle with $\rho_{\Gamma_1,\Gamma_2}(\theta_1,\theta_2)$:
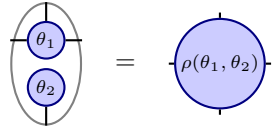
\[
\begin{tikzpicture}[unoriented WD, font=\small]
	\node (P1) {
	\begin{tikzpicture}[inner WD, pack size=6pt]
		\node[pack] (theta1) {$\theta_1$};
		\node[pack, below=.4 of theta1] (theta2) {$\theta_2$};
		\node[outer pack, inner ysep=0pt, fit=(theta1) (theta2)] (outer) {};
		\draw (theta1.0) -- (theta1.0-|outer.east);
		\draw (theta1.90) -- (outer.north);
		\draw (theta1.180) -- (theta1.180-|outer.west);
		\draw (theta2.270) -- (outer.south);
	\end{tikzpicture}	
	};
	\node[right=3 of P1] (P2) {
	\begin{tikzpicture}[inner WD, pack size=6pt]
		\node[pack] (rho) {$\rho(\theta_1,\theta_2)$};
		\draw (rho.0) -- +(0:2pt);
		\draw (rho.90) -- +(90:2pt);
		\draw (rho.180) -- +(180:2pt);
		\draw (rho.270) -- +(270:2pt);
	\end{tikzpicture}
	};
	\node at ($(P1.east)!.5!(P2.west)$) {$=$};
  \pgfresetboundingbox
	\useasboundingbox (P1.130) rectangle (P2.-40);
\end{tikzpicture}
\]
\end{example}

The meet-semilattice structure permits an intuitive graphical
interpretation. In the following proposition, the graphical terms on right are
illustrative examples of the equalities stated on the left.

\begin{proposition} \label{prop.diagrams_meet} \label{lem.true_removes_circles} \label{lem.meets_merge}
  For all contexts $\Gamma$ in $\frb$ and $\theta,\theta'\in \pr(\Gamma)$, we have
  \begin{enumerate}[label=(\roman*)]
    \item (True is removable) $\church{(\true_\Gamma;\Gamma)} = \church{ (;\epsilon_\Gamma)
    }$ 
    \hfill
    $
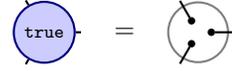
    \begin{tikzpicture}[unoriented WD, font=\small,baseline=(true)]
	\node (P1) {
	\begin{tikzpicture}[inner WD]
		\node[pack] (true) {$\true$};
		\draw (true.0) -- +(0:2pt);
		\draw (true.120) -- +(120:2pt);
		\draw (true.240) -- +(240:2pt);
  \end{tikzpicture}	
	};
	\node[right=3 of P1] (P2) {
	\begin{tikzpicture}[inner WD, shorten >=-2pt]
		\coordinate (helper);
		\node[link] (dot0) at ($(helper)+(0:5pt)$) {};
		\node[link] (dot120) at ($(helper)+(120:5pt)$) {};
		\node[link] (dot240) at ($(helper)+(240:5pt)$) {};
		\node[outer pack, surround sep=8pt, fit=(helper)] (outer) {};
		\draw (dot0) -- (outer.0);
		\draw (dot120) -- (outer.120);
		\draw (dot240) -- (outer.240);
  \end{tikzpicture}			
	};
	\node at ($(P1.east)!.5!(P2.west)$) {$=$};
      \end{tikzpicture}
    $

    \item (Meets are merges)
      $
    \church{ (\theta_1\wedge\theta_2;\Gamma)} = \church{
    (\theta_1,\theta_2;\delta_\Gamma) }.
    $
      \hfill
      $
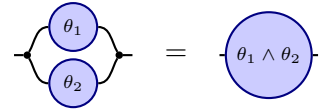
\begin{tikzpicture}[unoriented WD,baseline=(theta)]
	\node (P1) {
	\begin{tikzpicture}[inner WD]
		\node[pack] (theta1) {$\theta_1$};
		\node[pack, below=.3 of theta1] (theta2) {$\theta_2$};
		\coordinate (helper) at ($(theta1)!.5!(theta2)$);
		\node[link, left=2 of helper] (dot L) {};
		\node[link, right=2 of helper] (dot R) {};
		\draw (theta1.west) to[out=180, in=60] (dot L);
		\draw (theta2.west) to[out=180, in=-60] (dot L);
		\draw (theta1.east) to[out=0, in=120] (dot R);
		\draw (theta2.east) to[out=0, in=-120] (dot R);
		\draw (dot L) -- +(-5pt, 0);
		\draw (dot R) -- +(5pt, 0);
  \end{tikzpicture}
	};
	\node[right=3 of P1] (P2) {
		\begin{tikzpicture}[inner WD]
		\node[pack] (theta) {$\theta_1\wedge\theta_2$};
		\draw (theta.180) -- +(180:2pt);
		\draw (theta.0) -- +(0:2pt);
  \end{tikzpicture}	
	};
	\node at ($(P1.east)!.5!(P2.west)$) {$=$};
\end{tikzpicture}
$
  \end{enumerate}
\end{proposition}
\begin{proof}
  These equations are simply the definitions of $\true$
  and meet; see \cref{eq.def_true_meet}.
\end{proof}

\begin{example}[Discarding]\label{lem.dotting_off}
  Note that \cref{prop.diagrams_meet}(i) and the monotonicity of diagrams
  (\cref{prop.diagrams_basic}(i)) further imply that for all $\theta \in \pr(\Gamma)$
  we have $\church{ (\theta;\Gamma)} \vdash \church{
      (;\epsilon_\Gamma) }$:
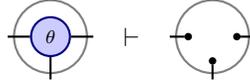
      \[
      \begin{tikzpicture}[unoriented WD, font=\small, baseline=(phi)]
	\node (P1) {
	  \begin{tikzpicture}[inner WD, shorten >=-2pt]
	    \node[pack] (phi) {$\theta$};
	    \node[outer pack, fit=(phi)] (outer) {};
	    \draw (phi) -- (outer.west);
	    \draw (phi) -- (outer.east);
	    \draw (phi) -- (outer.south);
	  \end{tikzpicture}
	};
	\node[right=3 of P1] (P2) {
	  \begin{tikzpicture}[inner WD, shorten >=-2pt]
	    \node[pack, fill=white, white] (phi) {$\theta$};
	    \node[outer pack, fit=(phi)] (outer) {};
	    \node[link] at ($(outer.180) - (180:5pt)$) (dot180) {};
	    \node[link] at ($(outer.0) - (0:5pt)$) (dot0) {};
	    \node[link] at ($(outer.270) - (270:5pt)$) (dot270) {};
	    \draw (dot0) -- (outer.0);
	    \draw (dot180) -- (outer.180);
	    \draw (dot270) -- (outer.270);
	  \end{tikzpicture}
	};
	\node at ($(P1.east)!.5!(P2.west)$) {$\vdash$};
      \end{tikzpicture}
    \]
  \end{example}

\chapter{Internal relations in a regular calculus} \label{chap.relations}

Having set up our proof language, we now return to describing the relationship
between regular categories and regular calculi. In this section, we'll see that
to every regular calculus we can construct a certain po-category, called its
internal relations po-category. Although we shall not prove it directly, this
internal relations po-category is in fact a regular po-category. We'll also get to
see our graphical logic in action.

\section{The internal relations po-category}\label{sec.int_rels}

\begin{definition}\label{def.internal_relations}
Given objects $\Gamma_1,\Gamma_2$ and $\varphi_1\in \pr(\Gamma_1)$ and $\varphi_2\in \pr(\Gamma_2)$, we
define the poset $\rela{\pr}(\varphi_1,\varphi_2)$ of \define{$\pr$-internal relations from
$\varphi_1$ to $\varphi_2$} to be the subposet
\[
  \rela{\pr}(\varphi_1,\varphi_2)\coloneqq
          \big\{\theta\in \pr(\Gamma_1\tens \Gamma_2)\,\big|\,
	  \lsh{(\pi_1)}\theta \vdash_{\Gamma_1} \varphi_1
			\text{ and }
	  \lsh{(\pi_2)}\theta \vdash_{\Gamma_2} \varphi_2
		\big\} \subseteq \pr(\Gamma_1\tens \Gamma_2).
\]
\end{definition}

An internal relation $\theta$ may be represented by the graphical term
$\dectheta$
together with the two entailments
\[
\begin{tikzpicture}[unoriented WD]
	\node (P11) {
  \begin{tikzpicture}[inner WD]
    \node[pack] (theta) {$\theta$};
    \node[link, right=5pt of theta] (dot) {};
    \draw (dot) -- (theta);
    \draw (theta) -- ($(theta.west)-(1,0)$);			
  \end{tikzpicture}	
	};
	\node[right=3 of P11] (P12) {
  \begin{tikzpicture}[inner WD, pack size=6pt]
    \node[pack] (phi) {$\varphi_1$};
		\draw (phi.west) to +(-2pt, 0);
  \end{tikzpicture}	
	};
	\node at ($(P11.east)!.5!(P12.west)$) {$\vdash_{\Gamma_1}$};
	\node[right=6 of P12] (P21) {
  \begin{tikzpicture}[inner WD]
    \node[pack] (theta) {$\theta$};
    \node[link, left=5pt of theta] (dot) {};
    \draw (dot) -- (theta);
    \draw (theta) -- ($(theta.east)+(1,0)$);			
  \end{tikzpicture}	
	};
	\node[right=3 of P21] (P22) {
  \begin{tikzpicture}[inner WD, pack size=6pt]
    \node[pack] (phi) {$\varphi_2$};
		\draw (phi.east) to +(2pt, 0);
  \end{tikzpicture}	
	};	
	\node at ($(P21.east)!.5!(P22.west)$) {$\vdash_{\Gamma_2}$};
\end{tikzpicture}
\]

We check that when this definition is applied to the regular calculus $\prd(\cat{R})$ associated to a regular category $\cat{R}$, it  recovers the usual notion of relation between objects in $\cat{R}$.

\begin{proposition}\label{prop.rela_rels_rrel}
Let $\cat{R}$ be a regular category, let $\Gamma_1,\Gamma_2\in\frc[\ob\cat{R}]$ be contexts, and suppose given $r_1\in\sub_{\cat{R}}\Prod[\Gamma_1]$ and $r_2\in\sub_{\cat{R}}\Prod[\Gamma_2]$. There is a natural isomorphism
\[\rela{\prd(\cat{R})}\big((\Gamma_1, r_1),(\Gamma_2, r_2)\big)\cong\rrel{\cat{R}}(r_1,r_2).\]
\end{proposition}
\begin{proof}
Let $g_1\coloneqq\Prod[\Gamma_1]$ and $g_2\coloneqq\Prod[\Gamma_2]$ so we have $r_1\ss g_1$ and $r_2\ss g_2$; see \cref{eqn.Prod}. By \cref{def.internal_relations,prop.rels}, a $\prd(\cat{R})$-internal relation between them is an element $t\ss g_1\times g_2$ such that there exist dotted arrows making the following diagram commute:
\[
\begin{tikzcd}[row sep=small]
	r_1\ar[d, >->]&
	t\ar[d, >->]\ar[r, dotted]\ar[l, dotted]&
	r_2\ar[d, >->]\\
	g_1&
	g_1\times g_2\ar[l]\ar[r]&
	g_2
\end{tikzcd}
\]
The composite $t\to r_1\times r_2\to g_1\times g_2$ is monic, so we have that $t\ss r_1\times r_2$. The result follows.
\end{proof}

We shall now present some technical lemmas with the goal of proving the
following theorem, that internal relations form a po-category. The proof is
completed on page~\pageref{page.proof_of_thm.internal_relations}.

\begin{theorem}\label{thm.internal_relations}
  Let $\pr\colon \frb \to \pposet$ be a regular calculus. Then there exists a
  po-category $\rela{\pr}$ whose objects are pairs
  $(\Gamma,\varphi)$, where $\Gamma$ is an object of $\frb$ and $\varphi \in
  \pr(\Gamma)$, and with hom-posets $(\Gamma_1,\varphi_1) \to
  (\Gamma_2,\varphi_2)$ given by $\rela{\pr}(\varphi_1,\varphi_2)$.
\end{theorem}

We begin by specifying the composition rule. For objects
$\Gamma_1,\Gamma_2,\Gamma_3$ in $\frb$, let
\[
  \comp_{\Gamma_1,\Gamma_2,\Gamma_3} \coloneqq
  \begin{aligned}
  \begin{tikzpicture}[inner WD]
    \node[pack] (theta) {};
    \node[pack, right=2 of theta] (theta') {};
    \node[outer pack, fit=(theta) (theta')] (outer) {};
    \draw (outer.east) to (theta'.east) node[right=1] {$\scriptstyle \Gamma_3$};
    \draw (theta) to node[above=-.3] {$\scriptstyle \Gamma_2$} (theta');
    \draw (theta.west) node[left=1] {$\scriptstyle \Gamma_1$} to (outer.west);			
  \end{tikzpicture}	
\end{aligned}
\]
It is a morphism $(\Gamma_1 \oplus \Gamma_2 \oplus \Gamma_2 \oplus \Gamma_3) \tickar (\Gamma_1\oplus\Gamma_3)$ in $\frc$. 
We then define
\begin{equation}\label{eqn.composition}
(-)\cp(-) \colon \pr(\Gamma_1\tens \Gamma_2)\times \pr(\Gamma_2\tens
\Gamma_3)\xrightarrow{\rho} \pr(\Gamma_1 \tens \Gamma_2 \tens \Gamma_2 \tens
\Gamma_3) \xrightarrow{\pr(\comp)}
\pr(\Gamma_1\tens \Gamma_3).
\end{equation}

\begin{remark}
  Note that this construction is reminiscent of the composition map defined in
  the construction of a hypergraph category from a cospan algebra in \cite{fong2019hypergraph}.
\end{remark}

\begin{lemma}\label{lemma.comp_rela_rela}
The composite of internal relations is an internal relation. That is, let
$\varphi_1 \in \pr(\Gamma_1)$, $\varphi_2 \in \pr(\Gamma_2)$, and $\varphi_3 \in \pr(\Gamma_3)$. Then given
$\theta_{12}\in\rela{\pr}(\varphi_1,\varphi_2)$ and $\theta_{23}\in\rela{\pr}(\varphi_2, \varphi_3)$, the
element $(\theta_{12}\cp\theta_{23})\in \pr(\Gamma_1 \tens \Gamma_3)$ is in $\rela{\pr}(\varphi_1,\varphi_3)$. 
\end{lemma}
\begin{proof}
  We must prove $\lsh{(\pi_1)}(\theta_{12} \cp \theta_{23}) \vdash \varphi_1$ and
  $\lsh{(\pi_2)}(\theta_{12} \cp \theta_{23}) \vdash \varphi_3$. We prove the first; the second
  follows similarly. This is not hard, we simply use \cref{lem.dotting_off} and then
  that $\theta_{12}$ obeys \cref{def.internal_relations}:
\[
\begin{tikzpicture}[unoriented WD]
	\node (P1) {
  \begin{tikzpicture}[inner WD, pack size=6pt]
    \node[pack] (theta) {$\theta_{12}$};
    \node[pack, right=1 of theta] (theta') {$\theta_{23}$};
    \node[link, right=5pt of theta'] (dot) {};
    \draw (dot) -- (theta');
    \draw (theta) -- (theta');
		\draw (theta.west) -- +(-5pt, 0);
  \end{tikzpicture}	
	};
	\node[right=3 of P1] (P2) {
  \begin{tikzpicture}[inner WD, pack size=6pt]
    \node[pack] (theta) {$\theta_{12}$};
    \node[link, right=5pt of theta] (dot) {};
    \draw (dot) -- (theta);
		\draw (theta.west) -- +(-5pt, 0);
  \end{tikzpicture}		
	};
	\node[right=3 of P2] (P3) {
  \begin{tikzpicture}[inner WD, pack size=6pt]
    \node[pack] (phi) {$\varphi_1$};
		\draw (phi.west) to +(-5pt, 0);
  \end{tikzpicture}		
	};
	\node at ($(P1.east)!.5!(P2.west)$) {$\vdash_{\Gamma_1}$};
	\node at ($(P2.east)!.5!(P3.west)$) {$\vdash_{\Gamma_1}$};
\end{tikzpicture}
\qedhere
\]
\end{proof}

Given an object $\Gamma\in\frb[\typeset]$ and $\varphi\in \pr(\Gamma)$, define $\id_\varphi\coloneqq \lsh{(\delta_\Gamma)}(\varphi)$ in $\pr(\Gamma\oplus\Gamma)$. Here it is graphically.
\begin{equation}\label{eqn.id_phi}
\id_\varphi\coloneqq\begin{tikzpicture}[inner WD, pack size=6pt, baseline=(dot)]
	\node[pack] (phi) {$\varphi$};
	\node[link, below=3pt of phi] (dot) {};
  \node[outer pack, fit=(phi) (dot)] (outer) {};
	\draw (phi) -- (dot);
	\draw (dot) to[pos=1] node[left] {$\Gamma$} (dot-|outer.west);
	\draw (dot) to[pos=1] node[right] {$\Gamma$} (dot-|outer.east);
\end{tikzpicture}
\end{equation}

\begin{lemma}\label{lemma.id_is_rela}
For any $\Gamma\in\frb[\typeset]$ and $\varphi\in \pr(\Gamma)$, the element $\id_\varphi\in
\pr(\Gamma\tens \Gamma)$ is an element of $\rela{\pr}(\varphi,\varphi)$.
\end{lemma}
\begin{proof}
By \cref{prop.diagrams_basic}(iii), composing the nested graphical term on the left is precisely the graphical term on the right (and similarly for the codomain):
\[
\begin{tikzpicture}[unoriented WD, baseline=(P1.base)]
  \node (P1) {
  \begin{tikzpicture}[inner WD, pack size=6pt]
  	\node[pack] (phi) {$\varphi$};
  	\node[link, below=3pt of phi] (dot) {};
    \node[outer pack, fit=(phi) (dot)] (outer) {};
    \node[link, right=3pt] at (dot-|outer.east) (dot2) {};
  	\draw (phi) -- (dot);
  	\draw (dot) -- (dot2);
		\node[outer pack, surround sep=-1pt, fit=(outer) (dot2)] (outer2) {};
  	\draw (dot) -- (dot-|outer2.west);
  \end{tikzpicture}
  };
  \node[right=3 of P1] (P2) {
  \begin{tikzpicture}[inner WD, pack size=6pt]
    \node[pack] (phi) {$\varphi$};
		\draw (phi.west) to +(-2pt, 0);
  \end{tikzpicture}		  
  };
	\node at ($(P1.east)!.5!(P2.west)$) {$\vdash_{\Gamma}$};
\end{tikzpicture}
\qedhere
\] 
\end{proof}

In what follows, we often elide details about---and graphical notation that
indicates---nesting and contexts.

\begin{lemma}\label{lemma.cp_unital}
The map $\cp$ from \cref{eqn.composition} is unital with respect to $\id$, i.e.\
$\theta\cp\id=\theta=\id\cp{\theta}$.
\end{lemma}
\begin{proof}
  We prove that $(\theta\cp\id)=\theta$; the other unitality axiom is similar.
  The inequality $(\theta\cp\id) \vdash \theta$ follows from
  \cref{lem.dotting_off,prop.diagrams_basic}:
  \[
  \begin{tikzpicture}[unoriented WD]
		\node (P1) {
	  \begin{tikzpicture}[inner WD, pack size=6pt]
			\node[pack] (theta) {$\theta$};
			\node[link, right=1.5 of theta] (dot) {};
			\node[pack, above=5pt of dot] (phi) {$\varphi$};
			\draw (theta) -- (dot);
			\draw (dot) -- (phi);
			\draw (dot) -- +(10pt,0);
			\draw (theta.west) -- +(-5pt,0);
    \end{tikzpicture}		
		};
		\node[right=4 of P1] (P2) {
	  \begin{tikzpicture}[inner WD, pack size=6pt]
			\node[pack] (theta) {$\theta$};
			\node[link, right=1.5 of theta] (dot) {};
			\node[link, above=4pt of dot] (phi) {};
			\draw (theta) -- (dot);
			\draw (dot) -- (phi);
			\draw (dot) -- +(10pt,0);
			\draw (theta.west) -- +(-5pt,0);
    \end{tikzpicture}		
		};
		\node[right=4 of P2] (P3) {\simpletheta};
  	\node at ($(P1.east)!.5!(P2.west)$) {$\vdash$};
  	\node at ($(P2.east)!.5!(P3.west)$) {$=$};
  \pgfresetboundingbox
	\useasboundingbox (P1.160) rectangle (P3.-20);
  \end{tikzpicture}
  \]
  The reverse inequality $\theta \vdash(\theta\cp\id)$ uses
  \cref{lem.meets_merge}, \cref{lem.breaking}, and \cref{def.internal_relations}:
  \[
  \begin{tikzpicture}[unoriented WD]
		\node (P1) {\simpletheta};
		\node[right=3 of P1] (P2) {
	  \begin{tikzpicture}[inner WD, pack size=6pt]
			\node[pack] (theta1) {$\theta$};
			\node[pack, below=.3 of theta1] (theta2) {$\theta$};
			\coordinate (helper) at ($(theta1)!.5!(theta2)$);
			\node[link, left=2 of helper] (dot L) {};
			\node[link, right=2 of helper] (dot R) {};
			\draw (theta1.west) to[out=180, in=60] (dot L);
			\draw (theta2.west) to[out=180, in=-60] (dot L);
			\draw (theta1.east) to[out=0, in=120] (dot R);
			\draw (theta2.east) to[out=0, in=-120] (dot R);
			\draw (dot L) -- +(-5pt,0);
			\draw (dot R) -- +(5pt,0);
    \end{tikzpicture}
		};
		\node[right=3 of P2] (P3) {
	  \begin{tikzpicture}[inner WD, pack size=6pt]
			\node[pack] (theta) {$\theta$};
			\node[link, right=1.5 of theta] (dot) {};
			\node[pack, above=3pt of dot] (theta2) {$\theta$};
			\node[link, above=2pt of theta2] (dot2) {};
			\draw (theta) -- (dot);
			\draw (theta2) -- (dot2);
			\draw (dot) -- (phi);
			\draw (dot) -- +(10pt,0);
			\draw (theta.west) -- +(-5pt, 0);
    \end{tikzpicture}				
		};
		\node[right=3 of P3] (P4) {
	  \begin{tikzpicture}[inner WD, pack size=6pt]
			\node[pack] (theta) {$\theta$};
			\node[link, right=1.5 of theta] (dot) {};
			\node[pack, above=5pt of dot] (phi) {$\varphi$};
			\draw (theta) -- (dot);
			\draw (dot) -- (phi);
			\draw (dot) -- +(10pt,0);
			\draw (theta.west) -- +(-5pt, 0);
    \end{tikzpicture}		
		};
  	\node at ($(P1.east)!.5!(P2.west)$) {$=$};
  	\node at ($(P2.east)!.5!(P3.west)$) {$\vdash$};
  	\node at ($(P3.east)!.5!(P4.west)$) {$\vdash$};
  \pgfresetboundingbox
	\useasboundingbox ($(P1.160)+(0,5pt)$) rectangle (P4.-20);
  \end{tikzpicture}  
  \qedhere
  \]
\end{proof}

\begin{lemma}\label{lemma.cp_assoc}
The map $\cp$ from \cref{eqn.composition} is associative, i.e.\ $(\theta_1\cp\theta_2)\cp\theta_3=\theta_1\cp(\theta_2\cp\theta_3)$.
\end{lemma}
\begin{proof}
  This is immediate from \cref{prop.diagrams_basic}(iii). Both sides can be represented by
  (nested versions of) the graphical term
$
  \begin{tikzpicture}[unoriented WD,baseline=(theta1.-40)]
		\node (P1) {
	  \begin{tikzpicture}[inner WD, pack size=2, surround sep=0pt]
			\node[pack] (theta1) {$\theta_1$};
			\node[pack, right=.5 of theta1] (theta2) {$\theta_2$};
			\node[pack, right=.5 of theta2] (theta3) {$\theta_3$};
			\draw (theta1.west) -- +(-5pt, 0);
			\draw (theta1) -- (theta2);
			\draw (theta2) -- (theta3);
			\draw (theta3.east) -- +(5pt, 0);
    \end{tikzpicture}	
		};
	 \end{tikzpicture}
$.
\end{proof}

\begin{proof}[Proof of \cref{thm.internal_relations}]\label{page.proof_of_thm.internal_relations}
  \cref{lemma.cp_unital,lemma.cp_assoc} show that we have a 1-category. 
  Each homset $\rela{\pr}(\varphi_1,\varphi_2) \subseteq \pr(\Gamma_1,\Gamma_2)$ inherits a
  partial order from the poset $\pr(\Gamma_1,\Gamma_2)$. Moreover, composition is given by
  the monotone map
  \[
    \rela{\pr}(\varphi_1,\varphi_2) \times \rela{\pr}(\varphi_2,\varphi_3)
    \stackrel{\rho}\longrightarrow
    \rela{\pr}(\varphi_1,\varphi_2,\varphi_2,\varphi_3)
    \xrightarrow{\pr(\comp)} \rela{\pr}(\varphi_1,\varphi_3).
  \]
  We thus have a po-category.
\end{proof}

\begin{remark}
Note that although each homset is a $\wedge$-semilattice, composition does
\emph{not} preserve meets, and so $\rela{\pr}$ is not $\wedge$-semilattice
enriched; see \cref{rem.meet_pres}.
\end{remark}

To conclude this section, we mention a useful characterization of internal relations.

\begin{proposition} \label{prop.characterize_relation}
  Let $\theta \in \pr(\Gamma_1 \oplus \Gamma_2)$, $\varphi_i \in \pr(\Gamma_i)$. Then $\theta$ is a relation $\varphi_1 \to
  \varphi_2$ if and only if
  \begin{equation}\label{eq.relation_with_domains}
    \begin{aligned}
    \begin{tikzpicture}[unoriented WD, baseline=(P2)]
      \node (P1) {
	\begin{tikzpicture}[inner WD, pack size=6pt]
	  \node[pack] (theta) {$\theta$};
	  \node[link, left=1 of theta] (dot) {};
	  \node[pack, above=3pt of dot] (phi) {$\varphi_{1}$};
	  \node[link, right=1 of theta] (dot2) {};
	  \node[pack, above=3pt of dot2] (phi2) {$\varphi_2$};
	  \draw (theta) -- (dot);
	  \draw (dot) -- (phi);
	  \draw (theta) -- (dot2);
	  \draw (dot2) -- (phi2);
		\draw (dot) -- +(-10pt, 0);
		\draw (dot2) -- +(10pt, 0);
	\end{tikzpicture}		
      };
      \node[right=4 of P1] (P2) {\simpletheta};
      \node at ($(P1.east)!.5!(P2.west)$) {$=$};
    \end{tikzpicture}
  \end{aligned}
  \end{equation}
\end{proposition}
\begin{proof}
  Any internal relation $\varphi_1 \to \varphi_2$ obeys the identity
  \cref{eq.relation_with_domains} by unitality, \cref{lemma.cp_unital}.
  Conversely, if $\theta$ obeys \cref{eq.relation_with_domains}, then by
  \cref{lem.dotting_off}
  \[
    \begin{tikzpicture}[unoriented WD, baseline=(P1)]
      \node (P1) {
	\begin{tikzpicture}[inner WD, pack size=6pt]
	  \node[pack] (theta) {$\theta$};
	  \node[link, right=5pt of theta] (dot) {};
	  \draw (dot) -- (theta);
		\draw (theta.west) -- +(-5pt, 0);			
	\end{tikzpicture}	
      };
      \node[right=3 of P1] (P2) {
	\begin{tikzpicture}[inner WD, pack size=6pt]
	  \node[pack] (theta) {$\theta$};
	  \node[link, left=1 of theta] (dot) {};
	  \node[pack, above=3pt of dot] (phi) {$\varphi_1$};
	  \node[link, right=1 of theta] (dot2) {};
	  \node[pack, above=3pt of dot2] (phi2) {$\varphi_2$};
	  \node[link, right=.7 of dot2] (dot3) {};
	  \draw (theta) -- (dot);
	  \draw (dot) -- (phi);
	  \draw (theta) -- (dot2);
	  \draw (dot2) -- (phi2);
		\draw (dot) -- +(-10pt, 0);
	  \draw (dot2) -- (dot3);
	\end{tikzpicture}		
      };
      \node[right=3 of P2] (P3) {
	\begin{tikzpicture}[inner WD, pack size=6pt]
	  \node[pack] (phi) {$\varphi_1$};
	  \draw (phi.west) to +(-2pt, 0);
	\end{tikzpicture}	
      };
      \node at ($(P1.east)!.5!(P2.west)$) {$=$};
      \node at ($(P2.east)!.5!(P3.west)$) {$\vdash$};
    \end{tikzpicture}
  \]
  and similarly for $\varphi_2$, proving that $\theta \in
  \rela{\pr}(\varphi_1,\varphi_2)$.
\end{proof}

\begin{definition}\label{def.transpose}
Write $\sigma_{\Gamma_1,\Gamma_2}\colon \Gamma_1 \tens \Gamma_2 \longrightarrow
\Gamma_2 \tens \Gamma_1$ for the braiding in $\frc$, and define the map
$(-)\tp\coloneqq \lsh{{\sigma_{\Gamma_1,\Gamma_2}}}\colon
\pr(\Gamma_1\tens\Gamma_2)\to \pr(\Gamma_2\tens \Gamma_1)$. We say that the
\define{transpose} of a graphical term $(\theta;\Gamma_1\oplus \Gamma_2)$ is the
graphical term $(\theta\tp;\Gamma_2\oplus \Gamma_1)$.
\end{definition}

\begin{remark} \label{lem.transpose_rotate}
  Note that transposes are given by ``rotating the shell'':
\[
\begin{tikzpicture}[unoriented WD, font=\small, pack size=6pt]
	\node (P1) {
	\begin{tikzpicture}[inner WD, baseline=(theta.base)]
  	\node[pack] (theta) {$\theta\tp$};
  	\draw (theta.west) to[pos=1] node[left] {$\Gamma_2$}  +(-2pt, 0);
  	\draw (theta.east) to[pos=1] node[right] {$\Gamma_1$}  +(2pt, 0);
  \end{tikzpicture}
	};
	\node[right=3 of P1] (P2) {
  \begin{tikzpicture}[inner WD, surround sep=5pt]
    \node[pack] (theta) {$\theta$};
    \coordinate (helper1) at ($(theta)+(0,.4cm)$);
    \coordinate (helper2) at ($(theta)-(0,.4cm)$);
    \node[outer pack, fit=(theta)] (outer) {};
    \draw	(theta.west) to[out=180, in=180, looseness=2] (helper1)
    to[out=0, in=180, pos=1] node[right] {$\Gamma_1$} (outer.east|-helper1);
    \draw	(theta.east) to[out=0, in=0, looseness=2] (helper2) to[out=180,
    in=0, pos=1] node[left] {$\Gamma_2$} (outer.west|-helper2);
  \end{tikzpicture}			
	};
	\node at ($(P1.east)!.5!(P2.west)$) {$=$};
  \pgfresetboundingbox
	\useasboundingbox (P1.west|-P2.160) rectangle (P2.east|-P2.-20);
\end{tikzpicture}
\]
In particular, for $\varphi \in \pr(\Gamma)$, we have $\church{
(\varphi\tp;\Gamma)} = \church{ (\varphi;\Gamma) }$. That
is, both $\varphi$ and $\varphi\tp$ can be represented by the diagram 
$
\begin{tikzpicture}[inner WD, pack size=6pt, baseline=(phi.-60)]
  \node[pack] (phi) {$\varphi$};
  \draw (phi.west) to +(-2pt, 0);
\end{tikzpicture}\; .
$
\end{remark}

We briefly note the following connection to hypergraph categories.

\begin{proposition}
  The monoidal category underlying $\rela{\pr}$ is a hypergraph 
  category. 
  
  More precisely, recall that we write $\rho$ for the laxators of
  $\pr$. We may equip $\rela{\pr}$ with the symmetric strict monoidal product given
  on objects by $(\Gamma_1,\varphi_1) \otimes (\Gamma_2,\varphi_2) = (\Gamma_1
  \oplus \Gamma_2, \rho(\varphi_1,\varphi_2))$, and on morphisms by the
  restriction to $\rela{\pr}(\varphi_1,\varphi_2) \times
  \rela{\pr}(\varphi_3,\varphi_4)$ of the map 
  \[
    \rho\colon \pr(\Gamma_1\oplus\Gamma_2) \times \pr(\Gamma_3\oplus\Gamma_4) \to
    \pr(\Gamma_1\oplus \Gamma_3\oplus \Gamma_2 \oplus \Gamma_4).
  \] 
  The braiding $\sigma_{\varphi_1,\varphi_2}$ on objects $(\Gamma_1,\varphi_1)$,
  $(\Gamma_2,\varphi_2)$ is given as below. Given this monoidal structure, we
  may the equip $\rela{\pr}$ with the hypergraph structure given on each
  object $(\Gamma,\varphi)$ by the internal relations below.
  
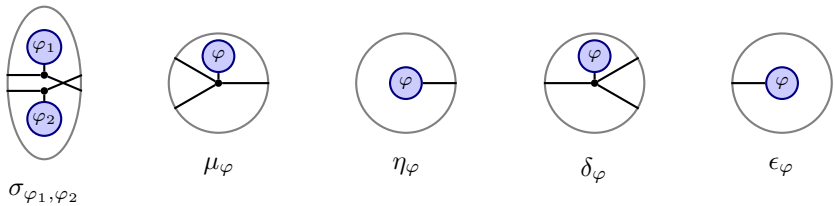
\[
    \begin{aligned}
      \begin{tikzpicture}[unoriented WD, font=\small]
	\node (P0) {
	  \begin{tikzpicture}[inner WD]
	    \node[link] (dot1) {};
	    \node[link,below=.4 of dot1] (dot2) {};
	    \node[pack,above=.3 of dot1,inner sep=1pt] (phi1) {$\varphi_1$};
	    \node[pack,below=.3 of dot2,inner sep=1pt] (phi2) {$\varphi_2$};
	    \node[outer pack, inner xsep=3pt, inner ysep=0pt, fit=(phi1) (phi2)] (outer) {};
	    \draw (dot1) -- (phi1.south);
	    \draw (dot2) -- (phi2.north);
	    \draw (dot1) -- (dot1-|outer.west);
	    \draw (dot1) -- (dot2-|outer.east);
	    \draw (dot2) -- (dot2-|outer.west);
	    \draw (dot2) -- (dot1-|outer.east);
	  \end{tikzpicture}	
	};
	\node[below=.1 of P0] {$\sigma_{\varphi_1,\varphi_2}$};
	\node[right=3 of P0] (P1) {
	  \begin{tikzpicture}[inner WD]
	    \node[link] (dot1) {};
	    \node[pack,above=.3 of dot1,inner sep=1pt] (phi1) {$\varphi$};
	    \node[outer pack, circle, minimum size=8ex, fit=(dot1)] (outer) {};
	    \draw (dot1) -- (phi1.south);
	    \draw (dot1) -- (outer.0);
	    \draw (dot1) -- (outer.150);
	    \draw (dot1) -- (outer.210);
	  \end{tikzpicture}	
	};
	\node[below=.1 of P1] {$\mu_{\varphi}$};
	\node[right=3 of P1] (P2) {
	  \begin{tikzpicture}[inner WD]
	    \node[pack,inner sep=1pt] (phi1) {$\varphi$};
	    \node[outer pack, minimum size=8ex, fit=(phi1)] (outer) {};
	    \draw (phi1.east) -- (outer.0);
	  \end{tikzpicture}	
	};	
	\node[below=.1 of P2] {$\eta_{\varphi}$};
	\node[right=3 of P2] (P3) {
	  \begin{tikzpicture}[inner WD]
	    \node[link] (dot1) {};
	    \node[pack,above=.3 of dot1,inner sep=1pt] (phi1) {$\varphi$};
	    \node[outer pack, minimum size=8ex, fit=(dot1)] (outer) {};
	    \draw (dot1) -- (phi1.south);
	    \draw (dot1) -- (outer.180);
	    \draw (dot1) -- (outer.30);
	    \draw (dot1) -- (outer.-30);
	  \end{tikzpicture}	
	};	
	\node[below=.1 of P3] {$\delta_{\varphi}$};
	\node[right=3 of P3] (P4) {
	  \begin{tikzpicture}[inner WD]
	    \node[pack,inner sep=1pt] (phi1) {$\varphi$};
	    \node[outer pack, minimum size=8ex, fit=(phi1)] (outer) {};
	    \draw (phi1.west) -- (outer.180);
	  \end{tikzpicture}	
	};	
	\node[below=.1 of P4] {$\epsilon_{\varphi}$};
      \end{tikzpicture}
    \end{aligned}
  \]
\end{proposition}
\begin{proof}
  Recall that we have already shown, in \cref{thm.internal_relations}, that
  $\rela{\pr}$ is a po-category. To prove this theorem then, it
  remains to check that the proposed monoidal products and structure maps are
  always well-defined internal relations, and then that the coherence laws
  for symmetric monoidal categories and hypergraph categories hold. These facts
  are all straightforward to verify using the logic of graphical terms.
\end{proof}

\section{The Carboni-Walters theorem}\label{sec.CW}

In \cite{Carboni:1987a}, Carboni and Walters defined the notions of \emph{cartesian bicategory} and \emph{functionally complete bicategory of relations}. The first of these falls out of our work so far. In what follows, we freely use notation from \cref{sec.int_rels}, such as $\cp$, $\tp$, $\delta$, $\mu$, $\epsilon$, and $\eta$.

\begin{definition}[Carboni-Walters]\label{def.cart_bicat}
	A \emph{cartesian bicategory} is a po-category $\ccat{C}$ with a unique adjoint monoid structure on each object $c$, such that each map $\alpha\colon c\to c'$ induces a lax comonoid homomorphism,
	\[
		\alpha\cp\epsilon_{c'}\leq\epsilon_{c}
		\qqand
		\alpha\cp\delta_{c'}\leq\delta_{c}\cp(\alpha\otimes\alpha).
	\]
\end{definition}

Now for any po-category $\ccat{C}$, there is a po-functor $U\colon\aadjmon(\ccat{C})\to\ccat{C}$ sending an ajax functor $1\to \ccat{C}$ to the image of $1$.

\begin{theorem}
A po-category $\ccat{C}$ is a cartesian bicategory iff $U\colon\aadjmon(\ccat{C})\to\ccat{C}$ is an isomorphism of po-categories.
\end{theorem}
\begin{proof}
This follows from \cref{eqn.monoid_comonoid_ajax,eqn.lax_mon_hom_2}.
\end{proof}

Our goal is to convert any regular calculus $\pr\colon\frb\to\pposet$ into a regular category $\syn(\pr)$. One approach is to show this directly; we do so in \cref{chap.functions,chap.ess_refl}. Another approach would be to use the Carboni-Walters theorem. While seemingly more direct, the latter approach has two drawbacks. First, it would make our paper less self-contained. Second, \cite{Carboni:1987a} seem not to describe functors between ``functionally complete bicategories of relations'' precisely enough for our needs. Thus we recall their theorem here and proceed to the direct approach, where we really see the graphical calculus in action. We will not see cartesian bicatgories again in this paper.

\begin{theorem}[Carboni-Walters]\label{def.CW_reg_bicat}
  Let $\ccat{C}$ be a cartesian bicategory. It is equivalent to $\rrel{\cat{R}}$ for some regular category $\cat{R}$ if and only if
	\begin{itemize}
		\item (Frobenius) $\mu_c\cp\delta_c=(c\otimes\delta_c)\cp(\mu_c\otimes c)$ for each $c\in\ccat{C}$, and
    \item (Images) For every $f\colon b\to I$ there exists an object $\im(f)$
      and a left adjoint $i\colon \im(f)\to b$ such that:%
	 \[
	 i\cp i\tp=\id_{\im(b)}
		\qqand	 i\tp\cp \epsilon_{\im(b)}=f.
	 \]
    \end{itemize}
\end{theorem}
\begin{proof}
  This is \cite[Theorem 3.5]{Carboni:1987a}.
\end{proof}

\chapter{Internal functions and the syntactic category construction}\label{chap.functions}

Internal functions are defined to be the left adjoints in the po-category $\rela{\pr}$ of
internal relations (see \cref{thm.internal_relations}).

\begin{definition} \label{def.RT}
  Given a regular calculus $(\typeset, \pr)$, where $\pr\colon \frb \to \pposet$, we define the category $\func{\pr}$
  of $\pr$-internal functions to be the category of left adjoints in $\rela{\pr}$:
  \begin{equation}\label{eqn.RT}
		\func{\pr}\coloneqq\ladj(\rela{\pr}).
	\end{equation}
  In more detail, suppose given elements $\varphi_1\in \pr(\Gamma_1)$ and $\varphi_2\in \pr(\Gamma_2)$. We say that an internal relation $\theta\in\rela{\pr}(\varphi_1, \varphi_2)\ss \pr(\Gamma_1\oplus\Gamma_2)$ is an \define{internal function} if there exists an internal
  relation $\xi$ such that
\[
      \begin{aligned}
	\begin{tikzpicture}[unoriented WD, font=\small]
	  \node (P1) {
	    \begin{tikzpicture}[inner WD, pack size=6pt]
	      \node[pack] (phi) {$\varphi_1$};
	      \node[link, below=2pt of phi] (dot) {};
	      \draw (phi) -- (dot);
	      \draw (dot) -- +(-8pt, 0);
	      \draw (dot) -- +(8pt, 0);
	    \end{tikzpicture}	
	  };
	  \node[right=2 of P1] (P2) {
	    \begin{tikzpicture}[inner WD, pack size=6pt]
	      \node[pack] (theta) {$\theta$};
	      \node[pack, right=1 of theta] (theta') {$\xi$};
	      \draw (theta) -- +(-10pt, 0);
	      \draw (theta') -- +(10pt, 0);
	      \draw (theta) to (theta');
	    \end{tikzpicture}	
	  };
	  \node at ($(P1.east)!.5!(P2.west)$) {$\vdash$};
	\end{tikzpicture}
      \end{aligned}
  \qqand        \begin{aligned}
      \begin{tikzpicture}[unoriented WD, font=\small]
	\node (P3) {
	  \begin{tikzpicture}[inner WD, pack size=6pt]
	    \node[pack] (theta') {$\xi$};
	    \node[pack, right=1 of theta'] (theta) {$\theta$};
	    \draw (theta) -- +(10pt, 0);
	    \draw (theta') -- +(-10pt, 0);
	    \draw (theta)  to  (theta');
	  \end{tikzpicture}	
	};
	\node[right=2 of P3] (P4) {
	  \begin{tikzpicture}[inner WD, surround sep=4pt, pack size=6pt]
	    \node[pack] (phi) {$\varphi_2$};
	    \node[link, below=2pt of phi] (dot) {};
	    \draw (phi) -- (dot);
	    \draw (dot) -- +(-8pt, 0);
	    \draw (dot) -- +(8pt, 0);
	  \end{tikzpicture}
	};
	\node at ($(P3.east)!.5!(P4.west)$) {$\vdash$};
      \end{tikzpicture}
      \end{aligned}
      .\]
      The category $\func{\pr}$ has the same objects $(\Gamma,\varphi)$ as $\rela{\pr}$, and morphisms given by internal functions.
\end{definition}

\begin{notation}
Graphically, we'll sometimes denote an internal function $\theta\in \pr(\Gamma_1,\Gamma_2)$
by the shape 
$
\begin{tikzpicture}[unoriented WD, font=\tiny]
	\node[funcr] (th) {$\theta$};
	\draw (th.west) to[pos=1] node[left=0] {$\Gamma_1$} +(-2pt, 0);
	\draw (th.east) to[pos=1] node[right=0] {$\Gamma_2$} +(2pt, 0);
\end{tikzpicture}
$
.
\end{notation}

Our aim in this section is to prove that the internal functions form a regular category.

\begin{theorem} \label{thm.internal_functions}
  For any regular calculus $\pr\colon\frb \to\pposet$, the category $\func{\pr}$
  of internal functions in $\rela{\pr}$ is regular.
\end{theorem}

The proof, found on page \pageref{page.proof_thm.internal_functions} is divided into three parts. In \cref{sec.internal_funs} we'll explore properties of internal
functions, in \cref{sec.finlims} we'll show $\func{\pr}$ has finite limits, and in \cref{sec.images} we'll show it has pullback stable
image factorizations and conclude with the theorem.

\section{Properties and examples of internal functions}\label{sec.internal_funs}
Before we embark on the theorem, let's get to know the category of internal
functions a bit. We'll first characterize functions in two ways: they're the
relations that have their own transposes as right adjoints, and they're the
relations that are total and deterministic. We'll then note that the order
inherited by functions as a subposet of the poset of relations is just the
discrete order, and give two important examples of functions: bijections and
projections.

To obtain our characterizations of functions, we'll need definitions of deterministic and total.
\begin{definition}\label{def.tot_det}
  Let $\theta \in \rela{\pr}(\varphi_1,\varphi_2)$. We say that $\theta$ is
  \begin{itemize}[nolistsep,noitemsep]
    \item  \define{total} if 
      $
      \begin{aligned}
	\begin{tikzpicture}[unoriented WD, pack size=6pt]
	  \node (P1) {
	    \begin{tikzpicture}[inner WD]
	      \node[pack] (xi) {$\varphi_1$};
	      \draw (xi.west) to[pos=1] +(-3pt, 0);
	    \end{tikzpicture}
	  };
	  \node[right=2 of P1] (P2) {
	    \begin{tikzpicture}[inner WD]
	      \node[pack] (theta) {$\theta$};
	      \node[link, right=.5 of theta] (dot) {};
	      \draw (theta.west) -- +(-3pt,0);
	      \draw (theta) -- (dot);
	    \end{tikzpicture}	
	  };
	  \node at ($(P1.east)!.5!(P2.west)$) {$\vdash$};
	\end{tikzpicture}
      \end{aligned}
      $
      , and 
    \item \define{deterministic} if
      $
      \begin{aligned}
	\begin{tikzpicture}[unoriented WD, pack size=6pt]
	  \node (P1) {
	    \begin{tikzpicture}[inner WD]
	      \node[pack] (theta) {$\theta$};
	      \node[pack,below=.3 of theta] (theta2) {$\theta$};
	      \node[link, left=1.5 of $(theta)!.5!(theta2)$] (dotw) {};
	      \draw (dotw) -- +(-10pt,0);
	      \draw (theta.west) -- (dotw);
	      \draw (theta2.west) -- (dotw);
	      \draw (theta.east) -- +(5pt,0);
	      \draw (theta2.east) -- +(5pt,0);
	    \end{tikzpicture}	
	  };
	  \node[right=3 of P1] (P2) {
	    \begin{tikzpicture}[inner WD]
	      \node[pack] (theta) {$\theta$};
	      \node[link, right=.5 of theta] (dot) {};
	      \draw (theta.west) -- +(-5pt,0);
	      \draw (theta.east) -- (dot);
	      \draw (dot) -- +(5pt,5pt);
	      \draw (dot) -- +(5pt,-5pt);
	    \end{tikzpicture}	
	  };
	  \node at ($(P1.east)!.5!(P2.west)$) {$\vdash$};
	\end{tikzpicture}
      \end{aligned}
      $
      .
  \end{itemize}
\end{definition}

\begin{remark}
  Note that by the domain of $\theta$ and discarding (\cref{lem.dotting_off}) we always have
  \[
    \begin{tikzpicture}[unoriented WD, pack size=6pt]
      \node (P1) {
	\begin{tikzpicture}[inner WD]
	  \node[pack] (theta) {$\theta$};
	  \node[link, right=.5 of theta] (dot) {};
		\draw (theta.west) -- +(-5pt, 0);
	  \draw (theta) -- (dot);
	\end{tikzpicture}	
      };
      \node[right=3 of P1] (P2) {
	\begin{tikzpicture}[inner WD]
	  \node[pack] (theta) {$\theta$};
	  \node[link, right=.5 of theta] (dot) {};
	  \node[link, left=.8 of theta] (dotw) {};
	  \node[pack, above=.8 of dotw] (phi) {$\varphi_1$};
	  \draw (dotw) -- +(-1,0);
	  \draw (phi) -- (dotw);
	  \draw (theta) -- (dotw);
	  \draw (theta) -- (dot);
	\end{tikzpicture}	
      };
      \node[right=3 of P2] (P3) {
	\begin{tikzpicture}[inner WD]
	  \node[pack] (xi) {$\varphi_1$};
	  \draw (xi.west) to[pos=1] +(-3pt, 0);
	\end{tikzpicture}
      };
      \node at ($(P1.east)!.5!(P2.west)$) {$=$};
      \node at ($(P2.east)!.5!(P3.west)$) {$\vdash$};
  \pgfresetboundingbox
	\useasboundingbox (P1.west|-P2.170) rectangle (P3.east|-P2.-10);
   \end{tikzpicture}
  \]
  and that by meets (\cref{prop.diagrams_meet}(ii)) and breaking
  (\cref{prop.diagrams_basic}(ii)) we always have
  \[
    \begin{tikzpicture}[unoriented WD, pack size=6pt]
      \node (P1) {
	\begin{tikzpicture}[inner WD]
	  \node[pack] (theta) {$\theta$};
	  \node[link, right=.5 of theta] (dot) {};
	  \draw (theta.west) -- +(-5pt, 0);
	  \draw (theta) -- (dot);
	  \draw (dot) -- +(45:8pt);
	  \draw (dot) -- +(-45:8pt);
	\end{tikzpicture}	
      };
      \node[right=3 of P1] (P2) {
	\begin{tikzpicture}[inner WD]
	  \node[pack] (theta) {$\theta$};
	  \node[pack,below=.5 of theta] (theta2) {$\theta$};
	  \node[link, right=1.5 of $(theta)!.5!(theta2)$] (dot) {};
	  \node[link, left=1.5 of $(theta)!.5!(theta2)$] (dotw) {};
	  \draw (dotw) -- +(-5pt, 0);
	  \draw (theta.west) -- (dotw);
	  \draw (theta2.west) -- (dotw);
	  \draw (theta.east) -- (dot);
	  \draw (theta2.east) -- (dot);
	  \draw (dot) -- +(45:8pt);
	  \draw (dot) -- +(-45:8pt);
	\end{tikzpicture}	
      };
      \node[right=3 of P2] (P3) {
	\begin{tikzpicture}[inner WD]
	  \node[pack] (theta) {$\theta$};
	  \node[pack,below=.5 of theta] (theta2) {$\theta$};
	  \node[link, left=1.5 of $(theta)!.5!(theta2)$] (dotw) {};
	  \draw (dotw) -- +(-5pt, 0);
	  \draw (theta.west) -- (dotw);
	  \draw (theta2.west) -- (dotw);
	  \draw (theta.east) -- +(5pt, 0);
	  \draw (theta2.east) -- +(5pt, 0);
	\end{tikzpicture}	
      };
      \node at ($(P1.east)!.5!(P2.west)$) {$=$};
      \node at ($(P2.east)!.5!(P3.west)$) {$\vdash$};
  \pgfresetboundingbox
	\useasboundingbox (P1.west|-P2.150) rectangle (P3.east|-P2.-10);
    \end{tikzpicture}
  \]
  This means that in \cref{def.tot_det} the two entailments are in fact equalities.
\end{remark}

In what follows, we'll often omit the transpose symbol $\tp$ (see \cref{def.transpose})
from our diagrams when it can be deduced from the ambient contextual information.

\begin{theorem}\label{thm.characterize_functions}
  Let $\theta \in \rela{\pr}(\varphi_1,\varphi_2)$. Then the following are
  equivalent.
  \begin{enumerate}[label=(\roman*),nolistsep, noitemsep]
    \item $\theta\in\func{\pr}$ is an internal function in the sense of \cref{def.RT}. 

    \item $\theta$ has right adjoint $\theta\tp$. That is, 
      $
      \begin{aligned}
	\begin{tikzpicture}[unoriented WD, font=\small, pack size=6pt]
	  \node (P1) {
	    \begin{tikzpicture}[inner WD]
	      \node[pack] (phi) {$\varphi_1$};
	      \node[link, below=2pt of phi] (dot) {};
	      \draw (phi) -- (dot);
	      \draw (dot) -- +(-8pt, 0);
	      \draw (dot) -- +(8pt, 0);
	    \end{tikzpicture}	
	  };
	  \node[right=2 of P1] (P2) {
	    \begin{tikzpicture}[inner WD]
	      \node[pack] (theta) {$\theta$};
	      \node[pack, right=1 of theta] (theta') {$\theta$};
	      \draw (theta) -- +(-10pt, 0);
	      \draw (theta') -- +(10pt, 0);
	      \draw (theta) to (theta');
	    \end{tikzpicture}	
	  };
	  \node at ($(P1.east)!.5!(P2.west)$) {$\vdash$};
	\end{tikzpicture}
      \end{aligned}
      $
      and
      $
      \begin{aligned}
      \begin{tikzpicture}[unoriented WD, pack size=6pt, font=\small]
	\node (P3) {
	  \begin{tikzpicture}[inner WD]
	    \node[pack] (theta') {$\theta$};
	    \node[pack, right=1 of theta'] (theta) {$\theta$};
	    \draw (theta) -- +(10pt, 0);
	    \draw (theta') -- +(-10pt, 0);
	    \draw (theta)  to  (theta');
	  \end{tikzpicture}	
	};
	\node[right=2 of P3] (P4) {
	  \begin{tikzpicture}[inner WD, surround sep=4pt]
	    \node[pack] (phi) {$\varphi_2$};
	    \node[link, below=2pt of phi] (dot) {};
	    \draw (phi) -- (dot);
	    \draw (dot) -- +(-8pt, 0);
	    \draw (dot) -- +(8pt, 0);
	  \end{tikzpicture}
	};
	\node at ($(P3.east)!.5!(P4.west)$) {$\vdash$};
      \end{tikzpicture}
      \end{aligned}
      $.
    \item $\theta$ is total and deterministic in the sense of \cref{def.tot_det}.
\end{enumerate}
\end{theorem}
\begin{proof}
  (i) $\iff$ (ii): Clearly (ii) $\imp$ (i). Conversely,
  assume $\theta$ has a right adjoint $\xi$. Note that the unit axiom implies 
  $
      \begin{aligned}
	\begin{tikzpicture}[unoriented WD, pack size=6pt, font=\small]
	  \node (P1) {
	\begin{tikzpicture}[inner WD]
	  \node[pack] (xi) {$\varphi_1$};
	  \draw (xi.west) to[pos=1] node[left, font=\tiny] {$\tau_1$} +(-3pt, 0);
	\end{tikzpicture}
	  };
	  \node[right=2 of P1] (P2) {
	    \begin{tikzpicture}[inner WD]
	      \node[pack] (theta) {$\theta$};
	      \node[pack, right=.75 of theta] (xi) {$\xi$};
	      \node[link, right=.5 of xi] (dot) {};
	      \draw (theta) -- +(-10pt, 0);
	      \draw (theta) to (xi);
	      \draw (xi) -- (dot);
	    \end{tikzpicture}	
	  };
	  \node[right=2 of P2] (P3) {
	    \begin{tikzpicture}[inner WD]
	      \node[pack] (theta) {$\theta$};
	      \node[link, right=.5 of theta] (dot) {};
	      \draw (theta) -- +(-10pt, 0);
	      \draw (theta) to (dot);
	    \end{tikzpicture}	
	  };
	  \node at ($(P1.east)!.5!(P2.west)$) {$\vdash$};
	  \node at ($(P2.east)!.5!(P3.west)$) {$\vdash$};
	\end{tikzpicture}
      \end{aligned}
      $.
  Then using meets and breaking we have 
  \[
    \begin{tikzpicture}[unoriented WD, pack size=6pt]
      \node (P1) {
	\begin{tikzpicture}[inner WD]
	  \node[pack, ellipse] (xi) {$\xi$};
	  \draw (xi.west) to[pos=1] node[left, font=\tiny] {$\Gamma_2$} +(-3pt, 0);
	  \draw (xi.east) to[pos=1] node[right, font=\tiny] {$\Gamma_1$} +(3pt, 0);
	\end{tikzpicture}
      };
      \node[right=2.5 of P1] (P2) {
	\begin{tikzpicture}[inner WD]
	  \node[pack] (xi) {$\xi$};
	  \node[link, right=6pt of xi] (dot) {};
	  \node[pack, above=3pt of dot] (theta) {$\theta$};
	  \node[link, above=3pt of theta] (dot2) {};
	  \draw (xi.west) -- +(-5pt, 0);
	  \draw (xi.east) -- (dot);
	  \draw (theta.south) -- (dot);
	  \draw (theta.north) -- (dot2);
	  \draw (dot) to +(10pt, 0);
	\end{tikzpicture}			
      };
      \node[right=2.5 of P2] (P3) {
	\begin{tikzpicture}[inner WD]
	  \node[pack] (xi) {$\xi$};
	  \node[link, right=12pt of xi] (dot) {};
	  \node[pack, above left=6pt and 2pt of dot] (theta) {$\theta$};
	  \node[pack, above right=6pt and 2pt of dot] (theta2) {$\theta$};
	  \draw (xi.west) -- +(-5pt, 0);
	  \draw (xi.east) -- (dot);
	  \draw (theta.south) -- (dot);
	  \draw (theta2.south) -- (dot);
	  \draw (theta.north) to[bend left=50pt] (theta2.north);
	  \draw (dot) to +(20pt, 0);
	\end{tikzpicture}			
      };
      \node[right=2.5 of P3] (P4) {
	\begin{tikzpicture}[inner WD]
	  \node[pack] (xi) {$\xi$};
	  \node[pack, right=6pt of xi] (theta) {$\theta$};
	  \node[pack, right=6pt of theta] (theta2) {$\theta$};
	  \draw (xi.west) -- +(-5pt, 0);
	  \draw (xi.east) -- (theta);
	  \draw (theta) -- (theta2);
	  \draw (theta2.east) -- +(5pt, 0);
	\end{tikzpicture}			
      };
      \node[right=2.5 of P4] (P5) {
	\begin{tikzpicture}[inner WD]
	  \node[pack, ellipse] (xi) {$\theta\tp$};
	  \draw (xi.west) to[pos=1] node[left, font=\tiny] {$\Gamma_2$} +(-3pt, 0);
	  \draw (xi.east) to[pos=1] node[right, font=\tiny] {$\Gamma_1$} +(3pt, 0);
	\end{tikzpicture}
      };
      \node at ($(P1.east)!.5!(P2.west)$) {$=$};
      \node at ($(P2.east)!.5!(P3.west)$) {$=$};
      \node at ($(P3.east)!.5!(P4.west)$) {$\vdash$};
      \node at ($(P4.east)!.5!(P5.west)$) {$\vdash$};
  \pgfresetboundingbox
	\useasboundingbox (P1.west|-P2.170) rectangle (P5.east|-P2.-10);
    \end{tikzpicture}
  \]
  Similarly we can show $\theta \vdash \xi\tp$, and hence $\xi = \theta\tp$.

(ii) $\iff$ (iii): We shall prove a stronger statement,
  that $\theta$ has a unit if and only if it is total, and that it has a counit if and only
  if it is deterministic.

  First, (ii)-units iff (iii)-totalness. Using the unit of the adjunction we have
  \[
    \begin{tikzpicture}[unoriented WD, pack size=6pt]
      \node (P1) {
	\begin{tikzpicture}[inner WD]
	  \node[pack] (xi) {$\varphi_1$};
	  \draw (xi.west) to[pos=1] node[left, font=\tiny] {$\Gamma_1$} +(-3pt, 0);
	\end{tikzpicture}
      };
      \node[right=3 of P1] (P2) {
	\begin{tikzpicture}[inner WD]
	  \node[pack] (theta) {$\theta$};
	  \node[pack,below=.5 of theta] (theta2) {$\theta$};
	  \node[link, left=1.5 of $(theta)!.5!(theta2)$] (dotw) {};
	  \draw (dotw) -- +(-5pt, 0);
	  \draw (theta.west) -- (dotw);
	  \draw (theta2.west) -- (dotw);
	  \draw (theta.east) to[bend left=50pt] (theta2.east);
	\end{tikzpicture}	
      };
      \node[right=3 of P2] (P3) {
	\begin{tikzpicture}[inner WD]
	  \node[pack] (theta) {$\theta$};
	  \node[link, right=.5 of theta] (dot) {};
	  \draw (theta) -- +(-10pt,0);
	  \draw (theta) -- (dot);
	\end{tikzpicture}	
      };
      \node at ($(P1.east)!.5!(P2.west)$) {$\vdash$};
      \node at ($(P2.east)!.5!(P3.west)$) {$=$};
  \pgfresetboundingbox
	\useasboundingbox (P1.west|-P2.170) rectangle (P3.east|-P2.-10);
    \end{tikzpicture}
  \]
  Conversely, using totalness, meets, and breaking we have
  \[
    \begin{tikzpicture}[unoriented WD, font=\small, pack size=6pt]
      \node (P1) {
	\begin{tikzpicture}[inner WD, surround sep=4pt]
	  \node[pack] (phi) {$\varphi_1$};
	  \node[link, below=3pt of phi] (dot) {};
	  \draw (phi) -- (dot);
	  \draw (dot) -- +(-10pt, 0);
	  \draw (dot) -- +(10pt, 0);
	\end{tikzpicture}	
      };
      \node[right=3 of P1] (P2) {
	\begin{tikzpicture}[inner WD, surround sep=4pt]
	  \node[pack] (theta) {$\theta$};
	  \node[link, below=3pt of theta] (dot) {};
	  \node[link, above=3pt of theta] (dot2) {};
	  \draw (theta) -- (dot2);
	  \draw (theta) -- (dot);
	  \draw (dot) -- +(-10pt, 0);
	  \draw (dot) -- +(10pt, 0);
	\end{tikzpicture}	
      };
      \node[right=3 of P2] (P3) {
	\begin{tikzpicture}[inner WD]
	  \node[link] (dot) {};
	  \node[pack, above left=6pt and 2pt of dot] (theta) {$\theta$};
	  \node[pack, above right=6pt and 2pt of dot] (theta2) {$\theta$};
	  \draw (dot) -- +(-10pt, 0);
	  \draw (theta.south) -- (dot);
	  \draw (theta2.south) -- (dot);
	  \draw (theta.north) to[bend left=50pt] (theta2.north);
	  \draw (dot) to +(10pt, 0);
	\end{tikzpicture}			
      };
      \node[right=3 of P3] (P4) {
	\begin{tikzpicture}[inner WD, surround sep=4pt]
	  \node[pack] (theta) {$\theta$};
	  \node[pack, right=2 of theta] (theta') {$\theta$};
	  \draw (theta.west) -- +(-10pt, 0);
	  \draw (theta'.east) -- +(10pt, 0);
	  \draw (theta) to node[above] {$\Gamma_2$} (theta');
	\end{tikzpicture}	
      };
      \node at ($(P1.east)!.5!(P2.west)$) {$=$};
      \node at ($(P2.east)!.5!(P3.west)$) {$=$};
      \node at ($(P3.east)!.5!(P4.west)$) {$\vdash$};
  \pgfresetboundingbox
	\useasboundingbox (P1.west|-P2.170) rectangle (P4.east|-P2.-10);
    \end{tikzpicture}
  \]

  Next, (ii)-counits iff (iii)-determinism. We can use the counit of the adjunction to give
  \[
    \begin{tikzpicture}[unoriented WD, pack size=6pt]
      \node (P1) {
	\begin{tikzpicture}[inner WD]
	  \node[pack] (theta) {$\theta$};
	  \node[pack,below=.5 of theta] (theta2) {$\theta$};
	  \node[link, left=1.5 of $(theta)!.5!(theta2)$] (dotw) {};
	  \draw (dotw) -- +(-5pt, 0);
	  \draw (theta.west) -- (dotw);
	  \draw (theta2.west) -- (dotw);
	  \draw (theta.east) -- +(5pt, 0);
	  \draw (theta2.east) -- +(5pt, 0);
	\end{tikzpicture}	
      };
      \node[right=3 of P1] (P2) {
	\begin{tikzpicture}[inner WD]
	  \node[pack] (theta) {$\theta$};
	  \node[pack,below=.5 of theta] (theta2) {$\theta$};
	  \node[pack,below=.5 of theta2] (theta3) {$\theta$};
	  \node[pack,below=.5 of theta3] (theta4) {$\theta$};
	  \node[link, left=1.5 of $(theta2)!.5!(theta3)$] (dotw) {};
	  \node[link, right=1.5 of $(theta)!.5!(theta2)$] (doteu) {};
	  \node[link, right=1.5 of $(theta3)!.5!(theta4)$] (doted) {};
	  \draw (dotw) -- +(-5pt, 0);
	  \draw (theta.west) -- (dotw);
	  \draw (theta2.west) -- (dotw);
	  \draw (theta3.west) -- (dotw);
	  \draw (theta4.west) -- (dotw);
	  \draw (theta.east) -- (doteu);
	  \draw (theta2.east) -- (doteu);
	  \draw (theta3.east) -- (doted);
	  \draw (theta4.east) -- (doted);
	  \draw (doteu) -- +(5pt, 0);
	  \draw (doted) -- +(5pt, 0);
	\end{tikzpicture}	
      };
      \node[right=3 of P2] (P3) {
	\begin{tikzpicture}[inner WD]
	  \node[pack] (theta) {$\theta$};
	  \node[pack,below=.5 of theta] (theta2) {$\theta$};
	  \node[pack,below=.5 of theta2] (theta3) {$\theta$};
	  \node[pack,below=.5 of theta3] (theta4) {$\theta$};
	  \node[link, left=2 of $(theta2)!.5!(theta3)$] (dotww) {};
	  \node[link, right=1.5 of $(theta)!.5!(theta2)$] (doteu) {};
	  \node[link, right=1.5 of $(theta3)!.5!(theta4)$] (doted) {};
	  \draw (dotww) -- +(-5pt, 0);
	  \draw (theta.west) -- (dotww);
	  \draw (theta3.west) to[bend left=50pt] (theta2.west);
	  \draw (theta4.west) -- (dotww);
	  \draw (theta.east) -- (doteu);
	  \draw (theta2.east) -- (doteu);
	  \draw (theta3.east) -- (doted);
	  \draw (theta4.east) -- (doted);
	  \draw (doteu) -- +(5pt, 0);
	  \draw (doted) -- +(5pt, 0);
	\end{tikzpicture}	
      };
      \node[right=3 of P3] (P4) {
	\begin{tikzpicture}[inner WD]
	  \node[pack] (theta) {$\theta$};
	  \node[pack,below=.5 of theta] (theta2) {$\theta$};
	  \node[link, right=1.5 of $(theta)!.5!(theta2)$] (dot) {};
	  \node[link, left=1.5 of $(theta)!.5!(theta2)$] (dotw) {};
	  \draw (dotw) -- +(-5pt, 0);
	  \draw (theta.west) -- (dotw);
	  \draw (theta2.west) -- (dotw);
	  \draw (theta.east) -- (dot);
	  \draw (theta2.east) -- (dot);
	  \draw (dot) -- +(45:8pt);
	  \draw (dot) -- +(-45:8pt);
	\end{tikzpicture}	
      };
      \node[right=3 of P4] (P5) {
	\begin{tikzpicture}[inner WD]
	  \node[pack] (theta) {$\theta$};
	  \node[link, right=.5 of theta] (dot) {};
	  \draw (theta.west) -- +(-5pt, 0);
	  \draw (theta) -- (dot);
	  \draw (dot) -- +(45:8pt);
	  \draw (dot) -- +(-45:8pt);
	\end{tikzpicture}	
      };
      \node at ($(P1.east)!.5!(P2.west)$) {$=$};
      \node at ($(P2.east)!.5!(P3.west)$) {$\vdash$};
      \node at ($(P3.east)!.5!(P4.west)$) {$\vdash$};
      \node at ($(P4.east)!.5!(P5.west)$) {$=$};
  \pgfresetboundingbox
	\useasboundingbox (P1.west|-P2.130) rectangle (P5.east|-P2.-40);
    \end{tikzpicture}
  \]
  Conversely, assuming determinism we get the counit, which concludes the proof:
  \[
    \begin{tikzpicture}[unoriented WD, font=\small, pack size=6pt]
      \node (P1) {
	\begin{tikzpicture}[inner WD, surround sep=4pt]
	  \node[pack] (theta') {$\theta$};
	  \node[pack, right=2 of theta'] (theta) {$\theta$};
	  \draw (theta.east) -- +(5pt, 0);
	  \draw (theta'.west) -- +(-5pt, 0);
	  \draw (theta)  to node[above] {$\Gamma_1$} (theta');
	\end{tikzpicture}	
      };
      \node[right=3 of P1] (P2) {
	\begin{tikzpicture}[inner WD, surround sep=4pt]
	  \node[pack] (theta) {$\theta$};
	  \node[link, below=3pt of theta] (dot) {};
	  \node[link, above=3pt of theta] (dot2) {};
	  \draw (theta) -- (dot2);
	  \draw (theta) -- (dot);
	  \draw (dot) -- +(-10pt, 0);
	  \draw (dot) -- +(10pt, 0);
	\end{tikzpicture}	
      };
      \node[right=3 of P2] (P3) {
	\begin{tikzpicture}[inner WD, surround sep=4pt]
	  \node[pack] (phi) {$\varphi_2$};
	  \node[link, below=3pt of phi] (dot) {};
	  \draw (phi) -- (dot);
	  \draw (dot) -- +(-10pt, 0);
	  \draw (dot) -- +(10pt, 0);
	\end{tikzpicture}
      };
      \node at ($(P1.east)!.5!(P2.west)$) {$=$};
      \node at ($(P2.east)!.5!(P3.west)$) {$\vdash$};
  \pgfresetboundingbox
	\useasboundingbox (P1.west|-P2.160) rectangle (P3.east|-P2.-20);
    \end{tikzpicture}
    \qedhere
  \]
\end{proof}

Next, we describe how the order on relations restricts to the functions. 

\begin{proposition} \label{cor.functions_discrete}
  The order on functions is discrete.
\end{proposition}
\begin{proof}
  Suppose 
  $
    \begin{aligned}
    \begin{tikzpicture}[unoriented WD]
      \node (P1) {
	\begin{tikzpicture}[inner WD]
	  \node[funcr] (phi) {$\theta$};
	  \draw (phi.west) to +(-2pt, 0);
	  \draw (phi.east) to +(2pt, 0);
	\end{tikzpicture}	
      };
      \node[right=2 of P1] (P2) {
	\begin{tikzpicture}[inner WD]
	  \node[funcr] (phi) {$\theta'$};
	  \draw (phi.west) to +(-2pt, 0);
	  \draw (phi.east) to +(2pt, 0);
	\end{tikzpicture}	
    };
      \node at ($(P1.east)!.5!(P2.west)$) {$\vdash$};
    \end{tikzpicture}
  \end{aligned}
  $
  . Then using the unit of $\theta$ and counit of $\theta'$ we have
  $
    \begin{aligned}
    \begin{tikzpicture}[unoriented WD]
      \node (P1) {
	\begin{tikzpicture}[inner WD]
	  \node[funcr] (phi) {$\theta'$};
	  \draw (phi.west) to +(-2pt, 0);
	  \draw (phi.east) to +(2pt, 0);
	\end{tikzpicture}	
      };
      \node[right=2 of P1] (P2) {
	\begin{tikzpicture}[inner WD]
	  \node[funcr] (t3) {$\theta'$};
	  \node[funcl, left=1 of t3] (t2) {$\theta$};
	  \node[funcr,left=1 of t2] (t1) {$\theta$};
	  \draw (t1.west) to +(-2pt, 0);
	  \draw (t1.east) to (t2.west);
	  \draw (t2.east) to (t3.west);
	  \draw (t3.east) to +(2pt, 0);
	\end{tikzpicture}	
    };
      \node[right=2 of P2] (P3) {
	\begin{tikzpicture}[inner WD]
	  \node[funcr] (t3) {$\theta'$};
	  \node[funcl, left=1 of t3] (t2) {$\theta'$};
	  \node[funcr, left=1 of t2] (t1) {$\theta$};
	  \draw (t1.west) to +(-2pt, 0);
	  \draw (t1.east) to (t2.west);
	  \draw (t2.east) to (t3.west);
	  \draw (t3.east) to +(2pt, 0);
	\end{tikzpicture}	
    };
      \node[right=2 of P3] (P4) {
	\begin{tikzpicture}[inner WD]
	  \node[funcr] (phi) {$\theta$};
	  \draw (phi.west) to +(-2pt, 0);
	  \draw (phi.east) to +(2pt, 0);
	\end{tikzpicture}	
    };
      \node at ($(P1.east)!.5!(P2.west)$) {$\vdash$};
      \node at ($(P2.east)!.5!(P3.west)$) {$\vdash$};
      \node at ($(P3.east)!.5!(P4.west)$) {$\vdash$};
    \end{tikzpicture}
  \end{aligned}
  $.
\end{proof}

Finally, we note that bijections and projections are examples of functions.
\begin{example}
  A \define{$\pr$-internal bijection} is an invertible $\pr$-internal
  relation. Note that every bijection is a function. We can also characterise
  bijections as the adjunctions whose unit and counit are the identity.
\end{example}

\begin{proposition}\label{prop.projections}
Suppose given $\varphi_1\in \pr(\Gamma_1)$ and $\varphi_2\in \pr(\Gamma_2)$ and a relation
$\theta\in\rela{\pr}(\varphi_1,\varphi_2)\ss \pr(\Gamma_1\tens \Gamma_2)$. Define
\[
	\pi_1\coloneqq\lsh{(\delta_{\Gamma_1}\tens \Gamma_2)}(\theta)
	\qand
	\pi_2\coloneqq\lsh{(\Gamma_2 \tens \delta_{\Gamma_2})}(\theta).\]
	Then $\pi_i\in \pr(\Gamma_1\tens \Gamma_i\tens \Gamma_2)$ are
internal functions for $i=1,2$, i.e.\ $\pi_i\in\func{\pr}(\theta,\varphi_i)$
\end{proposition}
\begin{proof}
  We prove $\pi_1$ is a function; the argument for $\pi_2$ is similar. Note that
  $\pi_1$ is depicted by the graphical term
\[
\begin{tikzpicture}[inner WD, surround sep=4pt, pack size=6pt, baseline=(dot.north)]
	\node[pack] (phi) {$\theta$};
	\node[link, below=8pt of phi.290] (dot) {};
	\coordinate (g1) at ($(dot)+(-10pt,0)$);
	\draw (phi.290) -- (dot);
	\draw (dot) to[pos=1] node[left] {$\Gamma_1$} (g1);
	\draw (dot) to[pos=1] node[right] {$\Gamma_1$} +(10pt, 0);
	\draw (phi.250) to[out=270, in=0, looseness=1, pos=1] node[left] {$\Gamma_2$} ($(g1)+(0,7pt)$);
\end{tikzpicture}
\]
By \cref{lem.meets_merge} and the fact that $\theta \in \rela{\pr}(\varphi_1,\varphi_2)$ we have 
\[
  \begin{tikzpicture}[unoriented WD, font=\small, pack size=6pt]
    \node (P1) {
      \begin{tikzpicture}[inner WD, surround sep=4pt]
	\node[pack] (phi) {$\theta$};
	\node[link, below=9pt of phi.290] (dot) {};
	\coordinate (g1) at ($(dot)+(-10pt,0)$);
	\draw (phi.290) -- (dot);
	\draw (dot) to[pos=1] node[left] {$\Gamma_1$} (g1);
	\draw (dot) to[pos=1] node[right] {$\Gamma_1$} +(10pt, 0);
	\draw (phi.250) to[out=270, in=0, looseness=1, pos=1] node[left] {$\Gamma_2$} ($(g1)+(0,7pt)$);
  \end{tikzpicture}
    };
    \node[right=3 of P1] (P2) {
      \begin{tikzpicture}[inner WD, surround sep=4pt]
	\node[pack] (theta) {$\theta$};
	\node[link, below=9pt of theta.290] (dot) {};
	\node[pack, left=10pt of theta] (thetaA) {$\theta$};
	\node[link, below=9pt of thetaA.290] (dotA) {};
	\coordinate (g1) at ($(dotA)+(-20pt,0)$);
	\node[link, inner sep=0, below=4.5pt of thetaA.250] (dot2A) {};
	\node[pack, right=10pt of theta] (phi) {$\varphi_1$};
	\node[link] at (dot-|phi) (dotphi) {};
	\draw (theta.250) to[out=270, in=0, looseness=.8] (dot2A);
	\draw (theta.290) -- (dot);
	\draw (thetaA.250) -- (dot2A);
	\draw (dot2A) to[pos=1] (g1|-dot2A);
	\draw (thetaA.290) -- (dotA);
	\draw (dotA) to[pos=1] (g1);
	\draw (dotA) -- (dot);
	\draw (phi) -- (dotphi);
	\draw (dot) -- (dotphi);
	\draw (dotphi) -- +(20pt,0);
      \end{tikzpicture}
    };
	\node at ($(P1.east)!.5!(P2.west)$) {$=$};
  \pgfresetboundingbox
	\useasboundingbox (P1.west|-P1.160) rectangle (P2.east|-P2.-10);
  \end{tikzpicture}
\]
and hence by \cref{prop.characterize_relation}, $\pi_1 \in
\rela{\pr}(\theta,\varphi_1)$.

Proving that $\pi_1$ is an adjunction in $\rela{\pr}(\theta,\varphi_1)$ again uses
\cref{lem.meets_merge} and that $\theta \in \rela{\pr}(\varphi_1,\varphi_2)$, as
well as \cref{lem.breaking}:
\[
  \begin{tikzpicture}[unoriented WD, font=\small, pack size=6pt, baseline=(P1)]
    \node (P1) {
      \begin{tikzpicture}[inner WD, surround sep=4pt]
	\node[pack] (phi) {$\theta$};
	\node[link, below=11pt of phi.290] (dot) {};
	\node[link, below=3pt of phi.250] (dot2) {};
	\draw (phi.250) -- (dot2);
	\draw (phi.290) -- (dot);
	\draw (dot2) to[pos=1] node[left] {$\Gamma_2$} +(-8pt, 0);
	\draw (dot2) to[pos=1] node[right] {$\Gamma_2$} +(12pt, 0);
	\draw (dot) to[pos=1] node[left] {$\Gamma_1$} +(-10pt, 0);
	\draw (dot) to[pos=1] node[right] {$\Gamma_1$} +(10pt, 0);
      \end{tikzpicture}
    };
    \node[right=2 of P1] (P2) {
      \begin{tikzpicture}[inner WD, surround sep=4pt]
	\node[pack] (theta) {$\theta$};
	\node[link, below=9pt of theta.250] (dot) {};
	\node[link, below=4.5pt of theta.290] (dot2) {};
	\node[pack, left=8pt of theta] (thetaA) {$\theta$};
	\node[link, below=9pt of thetaA.290] (dotA) {};
	\node[link, below=4.5pt of thetaA.250] (dot2A) {};
	\draw (theta.290) -- (dot2);
	\draw (theta.250) -- (dot);
	\draw (thetaA.250) -- (dot2A);
	\draw (thetaA.290) -- (dotA);
	\draw (dot2A) -- +(-8pt, 0);
	\draw (dotA) -- +(-10pt, 0);
	\draw (dotA) -- (dot);
	\draw (dot2A) -- (dot2);
	\draw (dot2) -- +(10pt, 0);
	\draw (dot) -- +(12pt, 0);
      \end{tikzpicture}
    };
    \node[right=2 of P2] (P3) {
      \begin{tikzpicture}[inner WD, surround sep=4pt]
	\node[pack] (theta) {$\theta$};
	\node[link, below=9pt of theta.290] (dot) {};
	\coordinate (g1) at ($(dot)+(-10pt, 0)$);
	\node[pack, right=8pt of theta] (thetaA) {$\theta$};
	\node[link, below=9pt of thetaA.250] (dotA) {};
	\coordinate (g2) at ($(dotA)+(10pt, 0)$);
	\coordinate (h) at ($(theta.south)+(0,-5pt)$);
	\draw (theta.250) to[out=270, in=0, looseness=.8] (g1|-h);
	\draw (theta.290) -- (dot);
	\draw (dot) --  (g1);
	\draw (dot) to (dotA);
	\draw (thetaA.290) to[out=270, in=180, looseness=.8]   (g2|-h);
	\draw (thetaA.250) -- (dotA);
	\draw (dotA) --  (g2);
      \end{tikzpicture}
    };
	\node at ($(P1.east)!.5!(P2.west)$) {$=$};
	\node at ($(P2.east)!.5!(P3.west)$) {$\vdash$};
  \end{tikzpicture}
\qquad\text{and}\quad
  \begin{tikzpicture}[unoriented WD, font=\small, pack size=6pt, baseline=(P1)]
    \node (P1) {
      \begin{tikzpicture}[inner WD, surround sep=4pt]
	\node[pack] (theta) {$\theta$};
	\node[link, below=8pt of theta.290] (dot) {};
	\node[pack, left=8pt of theta] (thetaA) {$\theta$};
	\node[link, below=8pt of thetaA.250] (dotA) {};
	\coordinate[below=4.5pt of $(theta.250)!.5!(thetaA.290)$] (dot2);
	\draw (theta.290) -- (dot);
	\draw (thetaA.250) -- (dotA);
	\draw (dotA) -- (dot);
	\draw (thetaA.290) to[out=270,in=180,looseness=1.3] (dot2);
	\draw (theta.250) to[out=270,in=0,looseness=1.3] (dot2);
	\draw (dotA) to[pos=1] node[left] {$\Gamma_1$} +(-10pt, 0);
	\draw (dot) to[pos=1] node[right] {$\Gamma_1$} +(10pt, 0);
      \end{tikzpicture}
    };
    \node[right=2 of P1] (P2) {
      \begin{tikzpicture}[inner WD, surround sep=4pt]
	\node[pack] (phi) {$\theta$};
	\node[link, below=9pt of phi.290] (dot) {};
	\node[link, below=4.5pt of phi.250] (dot2) {};
	\draw (phi.250) -- (dot2);
	\draw (phi.290) -- (dot);
	\draw (dot) -- +(10pt, 0);
	\draw (dot) -- +(-10pt, 0);
      \end{tikzpicture}
    };
    \node[right=2 of P2] (P3) {
      \begin{tikzpicture}[inner WD, surround sep=4pt]
	\node[pack] (phi) {$\varphi_1$};
	\node[link, below=5pt of phi.270] (dot) {};
	\draw (phi.270) -- (dot);
	\draw (dot) -- +(10pt, 0);
	\draw (dot) -- +(-10pt, 0);
      \end{tikzpicture}
    };
	\node at ($(P1.east)!.5!(P2.west)$) {$=$};
	\node at ($(P2.east)!.5!(P3.west)$) {$\vdash$};
  \end{tikzpicture}
  \qedhere
\]
\end{proof}

\begin{definition}
	With $\varphi_1,\varphi_2,\theta$ as in \cref{prop.projections}, we refer to the map $\lsh{(\delta_{\Gamma_1}\tens \Gamma_2)}\in\func{\pr}(\theta,\varphi_1)$ as the \emph{left projection} and similarly to $\lsh{(\Gamma_1\tens\delta_{\Gamma_2})}\in\func{\pr}(\theta,\varphi_2)$ as the \emph{right projection}.
\end{definition}

\section{Finite limits in $\func{\pr}$}\label{sec.finlims}

We now show how to construct finite limits in the category $\func{\pr}$ of
internal functions in $P$.

\begin{lemma}[Terminal object]\label{lemma.terminal}
	The object $(\unit,\true)\in\func{\pr}$ is terminal.
\end{lemma}
\begin{proof}
For any context $\Gamma$ and element $\varphi\in \pr(\Gamma)$ we shall show
$\varphi\in\func{\pr}((\Gamma,\varphi),(\unit,\true)) \subseteq
\pr(\Gamma\tens\unit)$ is the unique element. Note first that $\varphi$ is indeed
an internal function: it's an internal relation because $\varphi\vdash \varphi$
and $\lsh{\pi_2}(\varphi) \vdash \true$, and is an adjunction with counit given by
the fact that $\true$ is the top element, and unit given by meets and breaking
as follows
\[
\begin{tikzpicture}[unoriented WD, font=\small, pack size=6pt]
	\node (P1) {
  \begin{tikzpicture}[inner WD]
  	\node[pack] (phi) {$\varphi$};
  	\node[link, below=3pt of phi] (dot) {};
  	\draw (phi) -- (dot);
  	\draw (dot) -- +(-10pt, 0);
  	\draw (dot) -- +(10pt, 0);
  \end{tikzpicture}	
	};
	\node[right=4 of P1] (P2) {
  \begin{tikzpicture}[inner WD]
  	\node[pack] (phi) {$\varphi$};
		\node[pack, below=1 of phi] (phi2) {$\varphi$};
  	\node[link] at ($(phi)!.5!(phi2)$) (dot) {};
  	\draw (phi) -- (dot);
  	\draw (phi2) -- (dot);
  	\draw (dot) -- +(-10pt, 0);
  	\draw (dot) -- +(10pt, 0);
  \end{tikzpicture}		
	};
	\node[right=4 of P2] (P3) {
  \begin{tikzpicture}[inner WD]
  	\node[pack] (phi) {$\varphi$};
		\node[pack, below=1 of phi] (phi2) {$\varphi$};
		\coordinate (ow) at ($(phi)!.5!(phi2)+(-10pt, 0)$);
		\coordinate (oe) at ($(phi)!.5!(phi2)+(10pt, 0)$);
  	\draw (phi.south) to[out=270, in=0] (ow);
  	\draw (phi2.north) to[out=90, in=180] (oe);
  \end{tikzpicture}		
	};
	\node[right=4 of P3] (P4) {
  \begin{tikzpicture}[inner WD]
		\node[pack] (theta) {$\varphi$};
		\node[pack, right=1.5 of theta] (theta') {$\varphi$};
		\draw (theta.west) -- +(-5pt, 0);
		\draw (theta'.east) -- +(5pt, 0);
  \end{tikzpicture}	
	};
	\node at ($(P1.east)!.5!(P2.west)$) {$=$};
	\node at ($(P2.east)!.5!(P3.west)$) {$\vdash$};
	\node at ($(P3.east)!.5!(P4.west)$) {$=$};
  \pgfresetboundingbox
	\useasboundingbox (P1.west|-P2.140) rectangle (P4.east|-P2.-30);
\end{tikzpicture}
\] 

It remains to show uniqueness. If $\theta$ is an internal function then $\theta\vdash\varphi$, so it remains to show that $\varphi\vdash\theta$.  But it is easy to verify:
$
\begin{tikzpicture}[unoriented WD, font=\small, pack size=6pt, baseline=(P1.-20)]
	\node (P1) {
  \begin{tikzpicture}[inner WD]
  	\node[pack] (phi) {$\varphi$};
  	\draw (phi.west) -- +(-5pt,0);
	\end{tikzpicture}
	};
	\node[right=2 of P1] (P2) {
  \begin{tikzpicture}[inner WD]
    \node[pack] (theta) {$\theta$};
    \node[pack, right=.5 of theta] (theta') {$\theta$};
    \node[link, right=3pt of theta'] (dot) {};
    \draw (dot) -- (theta');
    \draw (theta.west) -- +(-5pt, 0);			
  \end{tikzpicture}			
	};
	\node[right=2 of P2] (P3) {
  \begin{tikzpicture}[inner WD]
    \node[pack] (phi) {$\theta$};
		\draw (phi.west) to +(-2pt, 0);
  \end{tikzpicture}	
	};
	\node at ($(P1.east)!.5!(P2.west)$) {$\vdash$};
	\node at ($(P2.east)!.5!(P3.west)$) {$\vdash$};
\end{tikzpicture}
$\;.
\end{proof}

\begin{lemma}[Pullbacks]\label{lemma.pullbacks}
	Let $\theta_1\colon(\Gamma_1,\varphi_1)\to(\Gamma,\varphi)$ and $\theta_2\colon(\Gamma_2,\varphi_2)\to(\Gamma,\varphi)$ be morphisms in $\func{\pr}$. Let $\theta_{12}\coloneqq(\theta_1\cp\theta_2\tp)$. Then the following is a pullback square in $\func{\pr}$:
	\[
	\begin{tikzcd}[row sep=22pt, column sep=70pt]
		\left((\Gamma_1\tens \Gamma_2),\theta_{12}\right)
		\ar[d, "{\lsh{(\delta_{\Gamma_1}\oplus\Gamma_2)}(\theta_{12})}"']
		\ar[r, "{\lsh{(\Gamma_1\oplus\delta_{\Gamma_2})}(\theta_{12})}"]&
		(\Gamma_2, \varphi_2)
			\ar[d, "\theta_2"]\\
		(\Gamma_1, \varphi_1)
			\ar[r, "\theta_1"']&
		(\Gamma,\varphi)
	\end{tikzcd}
	\]
\end{lemma}
\begin{proof}
The graphical term for the proposed pullback $\left((\Gamma_1\tens \Gamma_2),\theta_{12}\right)$ is shown left, and its proposed projection maps are shown middle and right:
\[
\begin{tikzpicture}[unoriented WD, font=\small]
  \node (P1) {
  \begin{tikzpicture}[inner WD, pack size=6pt]
  	\node[pack] (theta) {$\theta_{12}$};
  	\draw (theta.west) -- +(-2pt, 0);
  	\draw (theta.east) -- +(2pt, 0);
	\end{tikzpicture}
	};
	\node[right=3 of P1] (P2) {
  \begin{tikzpicture}[inner WD]
		\node[funcr] (theta) {$\theta_1$};
		\node[funcl, right=1 of theta] (theta') {$\theta_2$};
		\draw (theta.west) -- +(-5pt,0);
		\draw (theta'.east) -- +(5pt,0);
		\draw (theta) to node[above] {$\Gamma$} (theta');
  \end{tikzpicture}	
	};
	\node[right=7 of P2] (P3) {
  \begin{tikzpicture}[inner WD]
		\node[funcr] (theta) {$\theta_1$};
		\node[funcl, right=1 of theta] (theta') {$\theta_2$};
		\node[link, left=.5 of theta] (dot) {};
		\draw (theta) -- (dot);
		\draw (dot) to[pos=1] node[left] {$\Gamma_1$} +(135:10pt);
		\draw (dot) to[pos=1] node[left] {$\Gamma_1$} +(225:10pt);
		\draw (theta'.east) to[pos=1] node[right] {$\Gamma_2$} +(5pt,0);
		\draw (theta) -- (theta');
  \end{tikzpicture}	
	};
	\node[right=4 of P3] (P4) {
  \begin{tikzpicture}[inner WD]
		\node[funcr] (theta) {$\theta_1$};
		\node[funcl, right=1 of theta] (theta') {$\theta_2$};
		\node[link, right=.5 of theta'] (dot) {};
		\draw (theta.west) to[pos=1] node[left] {$\Gamma_1$} +(-5pt, 0);
		\draw (theta') -- (dot);
		\draw (dot) to[pos=1] node[right] {$\Gamma_2$} +(45:10pt);
		\draw (dot) to[pos=1] node[right] {$\Gamma_2$} +(-45:10pt);
		\draw (theta) -- (theta');
  \end{tikzpicture}	
	};
	\node at ($(P1.east)!.5!(P2.west)$) {$\coloneqq$};
  \pgfresetboundingbox
	\useasboundingbox (P1.west|-P1.160) rectangle (P4.east|-P4.-10);
\end{tikzpicture}
\]
Both projections are internal functions by \cref{prop.projections}. The necessary diagram commutes, i.e.\ we have equalities
\begin{equation}\label{eqn.comm_diag_projs}
\begin{tikzpicture}[unoriented WD, font=\small,baseline=(P1)]
  \node (P1) {
	\begin{tikzpicture}[inner WD]
		\node[funcr] (theta1) {$\theta_1$};
		\node[funcl, right=1 of theta1] (theta2) {$\theta_2$};
		\node[funcd, below left=.5 and 1 of theta1.west] (theta1') {$\theta_1$};
		\node[link] at (theta1'|-theta1) (dot) {};
		\draw (theta1'.south) -- +(0, -5pt);
		\draw (theta1'.north) -- (dot);
		\draw (dot) -- +(-10pt,0);
		\draw (dot) -- (theta1.west);
		\draw (theta1.east) -- (theta2.west);
		\draw (theta2.east) -- +(5pt, 0);
  \end{tikzpicture}
	};
	\node[right=3 of P1] (P2) {
  \begin{tikzpicture}[inner WD, surround sep=4pt]
		\node[funcr] (theta) {$\theta_1$};
		\node[funcl, right=1 of theta] (theta') {$\theta_2$};
		\node[link, "$\Gamma$"] at ($(theta)!.5!(theta')$) (dot) {};
		\draw (theta.west) -- +(-5pt,0);
		\draw (theta'.east) -- +(5pt, 0);
		\draw (theta) -- (dot);
		\draw (theta') -- (dot);
		\draw (dot) -- +(0, -10pt);
  \end{tikzpicture}	
	};
	\node[right=3 of P2] (P3) {
	\begin{tikzpicture}[inner WD, minimum size=20pt]
		\node[funcr] (theta1) {$\theta_1$};
		\node[funcl, right=1 of theta1] (theta2) {$\theta_2$};
		\node[funcd, below right=.5 and 1 of theta2.east] (theta2') {$\theta_2$};
		\node[link] at (theta2'|-theta2) (dot) {};
		\draw (theta2'.south) -- +(0, -5pt);
		\draw (theta2'.north) -- (dot);
		\draw (dot) -- +(10pt, 0);
		\draw (dot) -- (theta2.east);
		\draw (theta1.east) -- (theta2.west);
		\draw (theta1.west) -- +(-5pt, 0);
  \end{tikzpicture}	
	};
	\node at ($(P1.east)!.5!(P2.west)$) {$=$};
	\node at ($(P2.east)!.5!(P3.west)$) {$=$};
\end{tikzpicture}
\end{equation}
because functions are deterministic (\cref{thm.characterize_functions}).

Now we come to the universal property. Suppose given an object $(\Gamma',\varphi')$ and morphisms $\theta_1'\colon(\Gamma',\varphi')\to(\Gamma_1,\varphi_1)$ and $\theta_2'\colon(\Gamma',\varphi')\to(\Gamma_2,\varphi_2)$ in $\func{\pr}$, such that the $\theta_1'\cp\theta_1=\theta_2'\cp\theta_2$. Let $\pair{\theta_1',\theta_2'}$ denote the following graphical term:
\begin{equation}\label{eqn.pairing_pb}
  \begin{tikzpicture}[inner WD,baseline=(dot)]
		\node[funcl] (theta) {$\theta_1'$};
		\node[funcr, right=2 of theta] (theta') {$\theta_2'$};
		\node[link, "$\Gamma$" below] at ($(theta)!.5!(theta')$) (dot) {};
		\draw (theta.west) -- +(-5pt, 0);
		\draw (theta'.east) -- +(5pt, 0);
		\draw (theta) -- (dot);
		\draw (theta') -- (dot);
		\draw (dot) -- +(0,10pt);
  \end{tikzpicture}	
\end{equation}
We give one half of the proof that $\pair{\theta_1',\theta_2'}\in\rela{\pr}(\varphi',\theta_{12})$, the other half being easier.
\[
\begin{tikzpicture}[unoriented WD, font=\small]
  \node (P1) {
  \begin{tikzpicture}[inner WD]
		\node[funcl] (theta) {$\theta_1'$};
		\node[funcr, right=1 of theta] (theta') {$\theta_2'$};
		\draw (theta.west) -- +(-5pt, 0);
		\draw (theta'.east) -- +(5pt, 0);
		\draw (theta) -- (theta');
  \end{tikzpicture}	
	};
	\node[right=3 of P1] (P2) {
  \begin{tikzpicture}[inner WD]
		\node[funcl] (theta1') {$\theta_1'$};
		\node[funcr, right=1 of theta1'] (theta2') {$\theta_2'$};
		\node[funcl, left=1 of theta1'] (theta1) {$\theta_1$};
		\node[funcr, left=1 of theta1] (theta1t) {$\theta_1$};
		\draw (theta1t.west) -- +(-5pt, 0);
		\draw (theta1t) -- (theta1);
		\draw (theta1) -- (theta1');
		\draw (theta1') -- (theta2');
		\draw (theta2'.east) -- +(5pt, 0);
  \end{tikzpicture}		
	};
	\node[right=3 of P2] (P3) {
  \begin{tikzpicture}[inner WD]
		\node[funcr] (theta1') {$\theta_2'$};
		\node[funcl, right=1 of theta1'] (theta2') {$\theta_2'$};
		\node[funcl, left=1 of theta1'] (theta1) {$\theta_2$};
		\node[funcr, left=1 of theta1] (theta1t) {$\theta_1$};
		\draw (theta1t.west) -- +(-5pt, 0);
		\draw (theta1t) -- (theta1);
		\draw (theta1) -- (theta1');
		\draw (theta1') -- (theta2');
		\draw (theta2'.east) -- +(5pt, 0);
  \end{tikzpicture}
  };
  \node[right=3 of P3] (P4) {
  \begin{tikzpicture}[inner WD]
		\node[funcr] (theta) {$\theta_1$};
		\node[funcl, right=1 of theta] (theta') {$\theta_2$};
		\draw (theta.west) -- +(-5pt, 0);
		\draw (theta) -- (theta');
		\draw (theta'.east) -- +(5pt, 0);
  \end{tikzpicture}	  
  };
	\node at ($(P1.east)!.5!(P2.west)$) {$\vdash$};
	\node at ($(P2.east)!.5!(P3.west)$) {$=$};
	\node at ($(P3.east)!.5!(P4.west)$) {$\vdash$};  
  \pgfresetboundingbox
	\useasboundingbox (P1.west|-P1.160) rectangle (P4.east|-P4.-10);
\end{tikzpicture}
\]
Moreover, applying \cref{thm.characterize_functions}, a similarly straightforward diagrammatic argument shows $\pair{\theta_1',\theta_2'}\in\func{\pr}(\varphi',\theta_{12})$. We next need to show that $\pair{\theta_1',\theta_2'}\cp\lsh{(\delta_{\Gamma_1}\oplus\Gamma_2)}(\theta_{12})=\theta_1'$ and similarly for $\theta_2'$. This follows easily from \cref{cor.functions_discrete} and the diagram
\[
\begin{tikzpicture}[unoriented WD, font=\small]
	\node (P1) {
	\begin{tikzpicture}[inner WD]
		\node[funcl] (theta1) {$\theta_1'$};
		\node[funcr, below=1 of theta1] (theta2) {$\theta_1$};
		\coordinate (helper) at ($(theta1)!.5!(theta2)$);
		\node[link, left=2 of helper] (dot L) {};
		\node[funcr, right=2 of theta1] (theta'1) {$\theta_2'$};
		\node[link] at ($(theta1)!.5!(theta'1)$) (dot R) {};
		\node[funcl, right=2 of theta2] (theta'2) {$\theta_2$};
		\draw (theta1.west) to[out=180, in=60] (dot L);
		\draw (theta2.west) to[out=180, in=-60] (dot L);
		\draw (theta1.east) -- (dot R);
		\draw (dot R) -- (theta'1.west);
		\draw (dot R) -- ++(0, 5pt) to[out=90, in=180] ++(35pt,5pt);
		\draw (theta2.east) -- (theta'2.west);
		\draw (theta'1.east) to[out=0, in=0] (theta'2.east);
		\draw (dot L) -- +(-8pt, 0);
	\end{tikzpicture}
	};
	\node[right=3 of P1] (P2) {
	\begin{tikzpicture}[inner WD]
		\node[funcl] (theta) {$\theta_1'$};
		\draw (theta.west) -- +(-5pt, 0);
		\draw (theta.east) -- +(5pt, 0);
  \end{tikzpicture}
	};
	\node at ($(P1.east)!.5!(P2.west)$) {$\vdash$};
  \pgfresetboundingbox
	\useasboundingbox (P1.west|-P1.160) rectangle (P2.east|-P1.-20);
\end{tikzpicture}
\]
It only remains to show that this is unique. So suppose given $\theta'\in\func{\pr}(\varphi',\theta_{12})$ with 
$
\begin{aligned}
\begin{tikzpicture}[unoriented WD, font=\small]
  \node (P1) {
  \begin{tikzpicture}[inner WD]
  	\node[funcr] (theta) {$\theta_1'$};
  	\draw (theta.west) to[pos=1] node[left] {$\Gamma'$} +(-2pt, 0);
  	\draw (theta.east) to[pos=1] node[right] {$\Gamma_1$} +(2pt, 0);
	\end{tikzpicture}
	};
	\node[right=1 of P1] (P2) {
  \begin{tikzpicture}[inner WD, pack size=6pt]
  	\node[pack, minimum size=15pt] (theta) {$\theta'$};
		\node[link] at ($(theta.-30)+(-30:4pt)$) (dot){};
  	\draw (theta.west) to[pos=1] node[left] {$\Gamma'$} +(-2pt, 0);
  	\draw (theta.30) to[pos=1] node[right] {$\Gamma_1$} +(30:2pt);
  	\draw (theta.-30) -- (dot);
	\end{tikzpicture}
	};
\node at ($(P1.east)!.5!(P2.west)$) {$=$};
\end{tikzpicture}
\end{aligned}
$
and 
$
\begin{aligned}
\begin{tikzpicture}[unoriented WD, font=\small]
	\node (P3) {
  \begin{tikzpicture}[inner WD]
  	\node[funcr] (theta) {$\theta_2'$};
  	\draw (theta.west) to[pos=1] node[left] {$\Gamma'$} +(-2pt, 0);
  	\draw (theta.east) to[pos=1] node[right] {$\Gamma_2$} +(2pt, 0);
	\end{tikzpicture}	
	};
	\node[right=1 of P3] (P4) {
  \begin{tikzpicture}[inner WD, pack size=6pt]
  	\node[pack, minimum size=15pt] (theta) {$\theta'$};
		\node[link] at ($(theta.30)+(30:4pt)$) (dot){};
  	\draw (theta.west) to[pos=1] node[left] {$\Gamma'$} +(-2pt, 0);
  	\draw (theta.-30) to[pos=1] node[right] {$\Gamma_2$} +(-30:2pt);
  	\draw (theta.30) -- (dot);
	\end{tikzpicture}
	};
	\node at ($(P3.east)!.5!(P4.west)$) {$=$};
\end{tikzpicture}
\end{aligned}
$.
Then by basic diagram manipulations, one shows that $\theta'$ must equal the graphical term in \cref{eqn.pairing_pb}, as desired.
\end{proof}

\begin{proposition}\label{prop.monos}
Suppose that
$
\begin{tikzpicture}[unoriented WD, surround sep=2pt, font=\tiny, baseline=(theta.base)]
	\node[funcr] (theta) {$\theta$};
	\draw (theta.west) -- +(-2pt, 0);
	\draw (theta.east) -- +(2pt, 0);
\end{tikzpicture}
\in\func{\pr}(\varphi_1,\varphi_2)$ is an internal function. It is a monomorphism iff it satisfies
	$
	\begin{tikzpicture}[unoriented WD, baseline=(P1)]
  	\node (P1) {
  	\begin{tikzpicture}[inner WD, pack size=6pt]
    	\node[pack] (phi) {$\varphi_1$};
    	\node[link, below=3pt of phi] (dot) {};
    	\draw (phi) -- (dot);
    	\draw (dot) -- +(6pt,0);
    	\draw (dot) -- +(-6pt,0);
    \end{tikzpicture}	
    };
    \node[right=2 of P1] (P2) {
    \begin{tikzpicture}[inner WD]
  		\node[funcr] (theta) {$\theta$};
  		\node[funcl, right=1 of theta] (theta') {$\theta$};
			\draw (theta.west) -- +(-4pt,0);
			\draw (theta.east) -- (theta'.west);
			\draw (theta'.east) -- +(4pt,0);
    \end{tikzpicture}
    };
		\node at ($(P1.east)!.5!(P2.west)$) (vdash) {$=$};
  \end{tikzpicture}
	$.
\end{proposition}
\begin{proof}
  Recall that a morphism is a monomorphism iff the projection maps of its
  pullbacks along itself are the identity maps. Using the characterization of
  the projection maps of the pullback of $\theta$ along itself
  (\cref{lemma.pullbacks}) and the graphical logic, the proposition is immediate. 
\end{proof}

\begin{corollary}[Monomorphisms] \label{cor.monos}
If $\varphi\vdash_\Gamma\varphi'$, then $\id_{\varphi}\in \pr(\Gamma\tens \Gamma)$ as in
\cref{eqn.id_phi} is an element of $\func{\pr}((\Gamma,\varphi),(\Gamma,\varphi'))$ and it
is a monomorphism.
\end{corollary}
\begin{proof}
Since meets merge circles, we have the equality
\begin{equation}\label{eqn.phi_phi_phi}
\begin{tikzpicture}[unoriented WD, pack size=6pt, baseline=(P1)]
  \node (P1) {
	\begin{tikzpicture}[inner WD]
	  \coordinate (theta);
	  \node[link, left=1.5 of theta] (dot) {};
	  \node[pack, above=3pt of dot] (phi) {$\varphi$};
	  \node[link, right=1.5 of theta] (dot2) {};
	  \node[pack, above=3pt of dot2] (phi2) {$\varphi$};
	  \draw (theta) -- (dot);
	  \draw (dot) -- (phi);
	  \draw (theta) -- (dot2);
	  \draw (dot2) -- (phi2);
	  \draw (dot) -- +(-.5cm, 0);
	  \draw (dot2) -- +(.5cm, 0);
  \end{tikzpicture}
	};
	\node[right=3 of P1] (P2) {
  \begin{tikzpicture}[inner WD]
  	\node[pack] (phi) {$\varphi$};
  	\node[link, below=1pt of phi] (dot) {};
  	\draw (phi) -- (dot);
  	\draw (dot) -- +(-7pt, 0);
  	\draw (dot) -- +(7pt, 0);
  \end{tikzpicture}		
	};
	\node at ($(P1.east)!.5!(P2.west)$) {$=$};
\end{tikzpicture}	
\end{equation}
and it follows easily that $\id_{\varphi}\in\func{\pr}(\varphi,\varphi')$. But this also proves that $\id_{\varphi}$ is a monomorphism, by \cref{prop.monos}.
\end{proof}

\begin{remark}[Equalizers]
Given parallel arrows $\theta, \theta' \colon (\Gamma_1,\varphi_1) \to (\Gamma_2, \varphi_2)$, their equalizing object $(\Gamma_1,e)$ is the following graphical term:
\[
\begin{tikzpicture}[unoriented WD]
	\node (P1) {
	\begin{tikzpicture}[inner WD]
		\node[pack] (xi) {$e$};
	  \draw (xi.west) to[pos=1] node[left, font=\tiny] {$\Gamma_1$} +(-3pt, 0);
	\end{tikzpicture}
  };
  \node[right=3 of P1] (P2) {
	\begin{tikzpicture}[inner WD, pack size=12pt]
	  \node[funcr] (theta) {$\theta$};
	  \node[funcr, below=.5 of theta] (theta2) {$\theta'$};
	  \node[link, left=1.5 of $(theta)!.5!(theta2)$] (dotw) {};
	  \draw (dotw) -- +(-5pt,0);
	  \draw (theta.west) -- (dotw);
	  \draw (theta2.west) -- (dotw);
	  \draw (theta.east) to[out=0, in=0] (theta2.east);
	\end{tikzpicture}	
	};
	\node at ($(P1.east)!.5!(P2.west)$) {$=$};
\end{tikzpicture}
\]
\end{remark}

\section{Image factorizations}\label{sec.images}
We next discuss image factorizations, and show that they are stable under
pullback.

\begin{definition}
Suppose that
$
\theta\in\func{\pr}(\varphi_1,\varphi_2)$ is an internal function. Define its
\define{image}, denoted $\im(\theta)\in \pr(\Gamma_2)$ to be the graphical term 
$
\begin{tikzpicture}[inner WD]
	\node[funcr] (theta) {$\theta$};
	\node[link, left=5pt of theta] (dot) {};
  \draw (dot) -- (theta);
  \draw (theta.east) -- +(5pt,0);			
\end{tikzpicture}
$
or, in symbols, $\ust{\epsilon_{\Gamma_1}}\cp \theta$.
\end{definition}

We will now show that this has the usual properties of images, for example that
$\theta$ is a regular epimorphism in $\func{\pr}$ iff it satisfies 
$
\begin{aligned}
\begin{tikzpicture}[inner WD, pack size=6pt]
	\node[pack] (theta) {$\varphi_2$};
  	\draw (theta.east) -- +(5pt,0);			
\end{tikzpicture}
\end{aligned}
\vdash
\begin{aligned}
\begin{tikzpicture}[inner WD]
	\node[funcr] (theta) {$\theta$};
	\node[link, left=5pt of theta] (dot) {};
  \draw (dot) -- (theta);
  \draw (theta.east) -- +(5pt,0);			
\end{tikzpicture}
\end{aligned}
$.

\begin{proposition} \label{prop.epis}
Consider an element $\theta\in\func{\pr}(\varphi_1,\varphi_2)$. The following are equivalent:
\begin{enumerate}
	\item $\theta$, considered as a morphism in $\func{\pr}$, is a regular epimorphism,
	\item $\varphi_2\vdash_{\Gamma_2} \im(\theta)$,
	\item $\varphi_2= \im(\theta)$, and
	\item
	$
	\begin{tikzpicture}[unoriented WD, baseline=(P1)]
  	\node (P1) {
  	\begin{tikzpicture}[inner WD, pack size=6pt]
    	\node[pack] (phi) {$\varphi_2$};
    	\node[link, below=3pt of phi] (dot) {};
    	\draw (phi) -- (dot);
    	\draw (dot) -- +(10pt,0);
    	\draw (dot) -- +(-10pt,0);
    \end{tikzpicture}	
    };
    \node[right=2 of P1] (P2) {
    \begin{tikzpicture}[inner WD]
  		\node[funcl] (theta) {$\theta$};
  		\node[funcr, right=1 of theta] (theta') {$\theta$};
			\draw (theta.west) -- +(-4pt,0);
			\draw (theta.east) -- (theta'.west);
			\draw (theta'.east) -- +(4pt,0);
    \end{tikzpicture}
    };
		\node at ($(P1.east)!.5!(P2.west)$) (vdash) {$=$};
  \end{tikzpicture}
	$\;.
\end{enumerate}
\end{proposition}
\begin{proof}
$(1\imp 2)$:\quad It is straightforward to show that $\theta\in \pr(\Gamma_1,\Gamma_2)$ is an element of $\func{\pr}(\varphi_1,\im(\theta))$. Now supposing that $\theta$ is a regular epi, i.e.\ that the kernel pair diagram
\[
\begin{tikzcd}
	\varphi_1\times_{\varphi_2}\varphi_1\ar[r, shift left=3pt]\ar[r, shift right=3pt]&
	\varphi_1\ar[r]&
	\varphi_2
\end{tikzcd}
\]
is a coequalizer, it suffices to show that $\im(\theta)$ also coequalizes the
parallel pair:
\begin{equation}\label{eqn.coeqs}
\begin{tikzpicture}[unoriented WD, font=\small,baseline=(P1)]
  \node (P1) {
	\begin{tikzpicture}[inner WD, minimum size=20pt]
		\node[funcr] (theta1) {$\theta$};
		\node[funcl, right=1 of theta1] (theta2) {$\theta$};
		\node[funcd, below left=0 and .5 of theta1] (theta1') {$\theta$};
		\node[link] at (theta1'|-theta1) (dot) {};
		\draw (theta1'.south) -- +(0, -5pt);
		\draw (theta1'.north) -- (dot);
		\draw (dot) -- +(-10pt, 0);
		\draw (dot) -- (theta1.west);
		\draw (theta1.east) -- (theta2.west);
		\draw (theta2.east) -- +(10pt, 0);
  \end{tikzpicture}
	};
	\node[right=3 of P1] (P2) {
	\begin{tikzpicture}[inner WD, minimum size=20pt]
		\node[funcr] (theta1) {$\theta$};
		\node[funcl, right=1 of theta1] (theta2) {$\theta$};
		\node[funcd, below right=0 and .5 of theta2] (theta2') {$\theta$};
		\node[link] at (theta2'|-theta2) (dot) {};
		\draw (theta2'.south) -- +(0, -5pt);
		\draw (theta2'.north) -- (dot);
		\draw (dot) -- +(10pt, 0);
		\draw (dot) -- (theta2.east);
		\draw (theta1.east) -- (theta2.west);
		\draw (theta1.west) -- +(-10pt, 0);
  \end{tikzpicture}	
	};
	\node at ($(P1.east)!.5!(P2.west)$) {$=$};
\end{tikzpicture}
\end{equation}
This follows directly from determinism.

$(2\imp 3)$: For any relation $\theta\in\rela{\pr}(\varphi_1,\varphi_2)$ we always have the converse $\im(\theta)\vdash\varphi_2$.

$(3\imp 4)$: By determinism of $\theta$, we have
$
	\begin{tikzpicture}[unoriented WD, baseline=(P1)]
	  \node (P1) {
	    \begin{tikzpicture}[inner WD, pack size=6pt]
	      \node[pack] (theta) {$\varphi_2$};
	      \node[link, right=.5 of theta] (dot) {};
	      \draw (theta.east) -- (dot);
	      \draw (dot) -- +(5pt,5pt);
	      \draw (dot) -- +(5pt,-5pt);
	    \end{tikzpicture}	
	  };
	  \node[right=3 of P1] (P2) {
	    \begin{tikzpicture}[inner WD]
	      \node[funcr] (theta) {$\theta$};
	      \node[link, right=.5 of theta] (dot) {};
	      \node[link, left=5pt of theta] (dot L) {};
	      \draw (theta.west) -- +(dot L);
	      \draw (theta.east) -- (dot);
	      \draw (dot) -- +(5pt,5pt);
	      \draw (dot) -- +(5pt,-5pt);
	    \end{tikzpicture}
	  };
	  \node[right=3 of P2] (P3) {
	    \begin{tikzpicture}[inner WD]
	      \node[funcr] (theta) {$\theta$};
	      \node[funcr,below=.3 of theta] (theta2) {$\theta$};
	      \draw (theta.west) to[out=180, in=180] (theta2.west);
	      \draw (theta.east) -- +(5pt,0);
	      \draw (theta2.east) -- +(5pt,0);
	    \end{tikzpicture}	
	  };
	  \node at ($(P1.east)!.5!(P2.west)$) {=};
	  \node at ($(P2.east)!.5!(P3.west)$) {=};
	\end{tikzpicture}
$

$(4\imp 1)$: Assuming 4, we need to show that
$
\begin{tikzcd}[column sep=small]
	\varphi_1\times_{\varphi_2}\varphi_1\ar[r, shift left=2pt]\ar[r, shift right=2pt]&
	\varphi_1\ar[r]&
	\varphi_2
\end{tikzcd}
$
is a coequalizer. It is easy to show that $\varphi_2$ coequalizes the parallel pair; this is basically \cref{eqn.coeqs} again. So let $\theta'\colon\varphi_1\to\varphi_2'$ coequalize the parallel pair, and define $\xi\in\rela{\pr}(\varphi_2,\varphi_2')$ by $\xi\coloneqq\theta\tp\cp\theta'$. We need to show that $\xi$ is a function and that $\theta\cp\xi=\theta'$.

We obtain $\id_{\varphi_2}\vdash \xi\cp\xi\tp$ using (4) and the fact that $\theta'$ is a function: 
\[
	\begin{tikzpicture}[unoriented WD, baseline=(P1)]
  	\node (P1) {
  	\begin{tikzpicture}[inner WD, pack size=6pt]
    	\node[pack] (phi) {$\varphi_2$};
    	\node[link, below=3pt of phi] (dot) {};
    	\draw (phi) -- (dot);
    	\draw (dot) -- +(6pt,0);
    	\draw (dot) -- +(-6pt,0);
    \end{tikzpicture}	
    };
    \node[right=2 of P1] (P2) {
    \begin{tikzpicture}[inner WD]
  		\node[funcl] (theta) {$\theta$};
  		\node[funcr, right=1 of theta] (theta') {$\theta$};
			\draw (theta.west) -- +(-4pt,0);
			\draw (theta.east) -- (theta'.west);
			\draw (theta'.east) -- +(4pt,0);
    \end{tikzpicture}
    };
    \node[right=2 of P2] (P3) {
    \begin{tikzpicture}[inner WD]
  		\node[funcl] (theta) {$\theta$};
  		\node[funcr, right=1 of theta] (theta') {$\theta'$};
  		\node[funcl, right=1 of theta'] (theta'') {$\theta'$};
  		\node[funcr, right=1 of theta''] (theta''') {$\theta$};
			\draw (theta.west) -- +(-4pt,0);
			\draw (theta.east) -- (theta'.west);
			\draw (theta'.east) -- (theta''.west);
			\draw (theta''.east) -- (theta'''.west);
			\draw (theta'''.east) -- +(4pt,0);
    \end{tikzpicture}
    };
		\node at ($(P1.east)!.5!(P2.west)$) (vdash) {$=$};
		\node at ($(P2.east)!.5!(P3.west)$) (vdash) {$\vdash$};
  \end{tikzpicture}
.
\]
We obtain $\xi\tp\cp\xi\vdash\id_{\varphi_2'}$ as follows:
\[
\begin{tikzpicture}[unoriented WD]
	\node (P1) {
    \begin{tikzpicture}[inner WD]
  		\node[funcl] (theta) {$\theta'$};
  		\node[funcr, right=1 of theta] (theta') {$\theta$};
  		\node[funcl, right=1 of theta'] (theta'') {$\theta$};
  		\node[funcr, right=1 of theta''] (theta''') {$\theta'$};
			\draw (theta.west) -- +(-4pt,0);
			\draw (theta.east) -- (theta'.west);
			\draw (theta'.east) -- (theta''.west);
			\draw (theta''.east) -- (theta'''.west);
			\draw (theta'''.east) -- +(4pt,0);
    \end{tikzpicture}	
	};
	\node[right=3 of P1] (P2) {
    \begin{tikzpicture}[inner WD]
  		\node[funcr] (theta') {$\theta$};
			\node[link, left=3pt of theta'] (dot) {};
  		\node[funcl, right=1 of theta'] (theta'') {$\theta$};
			\node[link, right=1.25 of theta''] (dot R) {};
  		\node[funcr, right=1.25 of dot R] (theta''') {$\theta'$};
			\node[funcd, below=.25 of dot R] (theta) {$\theta'$};
			\draw (dot) -- (theta'.west);
			\draw (theta'.east) -- (theta''.west);
			\draw (theta''.east) -- (dot R);
			\draw (dot R) -- (theta'''.west);
			\draw (dot R) -- (theta.north);
			\draw (theta'''.east) -- +(4pt,0);
			\draw (theta.south) to[out=270, in=90] +(0, -5pt);
    \end{tikzpicture}			
	};
	\node[right=3 of P2] (P3) {
	\begin{tikzpicture}[inner WD]
	  \node[funcr] (phi) {$\theta'$};
	  \node[link, left=3pt of phi] (dot) {};
	  \node[link, right=3pt of phi] (dot R) {};
	  \draw (phi.west) to (dot);
	  \draw (phi.east) to (dot R);
	  \draw (dot R) -- +(45:7pt);
	  \draw (dot R) -- +(-45:7pt);
	\end{tikzpicture}	
	};
	\node at ($(P1.east)!.5!(P2.west)$) (vdash) {$=$};
	\node at ($(P2.east)!.5!(P3.west)$) (vdash) {$\vdash$};
  \pgfresetboundingbox
	\useasboundingbox (P1.170) rectangle (P3.-10);
\end{tikzpicture}
\]
where the first equality comes from the fact that $\theta'$ coequalizes the parallel pair, and the second is discarding and determinism of $\theta'$. Finally, $\theta'\vdash\theta\cp\theta\tp\cp\theta=\theta\cp\xi$ follows easily from $\theta$ being a function. The converse $\theta\cp\xi\vdash\theta'$ follows from the fact that $\theta'$ coequalizes the parallel pair:
\[
\begin{tikzpicture}[unoriented WD]
	\node (P1) {
  \begin{tikzpicture}[inner WD]
 		\node[funcr] (theta') {$\theta$};
  	\node[funcl, right=1 of theta'] (theta'') {$\theta$};
 		\node[funcr, right=1 of theta''] (theta''') {$\theta'$};
		\draw (theta'.west) -- +(-4pt,0);
		\draw (theta'.east) -- (theta''.west);
		\draw (theta''.east) -- (theta'''.west);
		\draw (theta'''.east) -- +(4pt,0);
  \end{tikzpicture}	
	};
	\node[right=3 of P1] (P2) {
  \begin{tikzpicture}[inner WD]
 		\node[funcr] (theta') {$\theta$};
  	\node[funcl, right=1 of theta'] (theta'') {$\theta$};
		\node[link, left=1.5 of theta'] (dot L) {};
		\node[link, right=.5 of theta''] (dot R) {};
		\node[funcd, below=.25 of dot L] (theta) {$\theta'$};
		\draw (theta'.west) -- (dot L);
		\draw (theta'.east) -- (theta''.west);
		\draw (theta''.east) -- (dot R);
		\draw (dot L) -- +(-6pt,0);
		\draw (dot L) -- (theta);
		\draw (theta.south) -- +(0,-4pt);
  \end{tikzpicture}	
	};
	\node[right=3 of P2] (P3) {
	\begin{tikzpicture}[inner WD]
	  \node[funcr] (phi) {$\theta'$};
	  \draw (phi.west) to +(-4pt,0);
	  \draw (phi.east) to +(4pt,0);
	\end{tikzpicture}	
	};
	\node at ($(P1.east)!.5!(P2.west)$) (vdash) {$=$};	
	\node at ($(P2.east)!.5!(P3.west)$) (vdash) {$\vdash$};	
\end{tikzpicture}
\qedhere
\]
\end{proof}

\begin{lemma}[Image factorizations]\label{lemma.image_fact}
  Any morphism $\theta\colon(\Gamma',\varphi')\to(\Gamma,\varphi)$ can be factored into a regular epimorphism followed by a monomorphism; the image object is $(\Gamma,\ust{\epsilon_{\Gamma'}}\cp\theta)$.
\end{lemma}
\begin{proof}
  The image factorization of $\theta$ is given by
    \[
    \begin{aligned}
      \begin{tikzpicture}[unoriented WD, font=\small]
	\node (P1) {
	  \begin{tikzpicture}[inner WD]
	    \node[funcr] (xi) {$\theta$};
	    \draw (xi.west) to[pos=1] node[left, font=\tiny] {$\Gamma'$} +(-3pt, 0);
	    \draw (xi.east) to[pos=1] node[right, font=\tiny] {$\Gamma$} +(3pt, 0);
	  \end{tikzpicture}
	};
	\node[right=3 of P1] (P2) {
	  \begin{tikzpicture}[inner WD]
	    \node[funcr] (xi) {$\theta$};
	    \node[outer pack, surround sep = 1pt, fit=(xi)] (cp1) {};
	    \node[link, right=20pt of xi] (dot) {};
	    \node[funcd, above=2pt of dot] (theta) {$\theta$};
	    \node[link, above=2pt of theta] (dot2) {};
	    \node[outer pack, surround sep = 0pt, fit=(theta) (dot) (dot2)] (cp2) {};
	    \node[outer pack, surround sep = 4pt, fit=(xi) (theta) (dot) (dot2)] (outer) {};
	    \draw (xi.west) to (xi-|outer.west);
	    \draw (xi.east) -- (dot);
	    \draw (theta.south) -- (dot);
	    \draw (theta.north) -- (dot2);
	    \draw (dot) to (xi-|outer.east);
	  \end{tikzpicture}			
	};
	\node at ($(P1.east)!.5!(P2.west)$) {$=$};
  \pgfresetboundingbox
	\useasboundingbox (P1.west|-P2.160) rectangle (P2.east|-P2.-20);
      \end{tikzpicture}
    \end{aligned}
  \]
  The graphical representation of the image object $(\Gamma,\ust{\epsilon_{\Gamma'}}\cp\theta)$ is $
\begin{tikzpicture}[inner WD]
	\node[funcr] (theta) {$\theta$};
	\node[link, left=5pt of theta] (dot) {};
  \draw (dot) -- (theta);
  \draw (theta.east) -- +(5pt,0);			
\end{tikzpicture}
$
\;. It is immediate from
  \cref{prop.epis} that $\theta$ is a regular epimorphism $(\Gamma',\varphi') \to
  (\Gamma,\ust{\epsilon_{\Gamma'}}\cp\theta)$, and from \cref{cor.monos} that
  $\lsh{(\delta_\Gamma)}(\ust{\epsilon_{\Gamma'}}\cp\theta)$ is a monomorphism
  $(\Gamma,\ust{\epsilon_{\Gamma'}}\cp\theta) \to (\Gamma,\varphi)$.
\end{proof}

\begin{lemma}[Pullback stability of image factorizations]\label{lemma.pb_stability}
The pullback of a regular epimorphism along any morphism is again a regular epimorphism in $\func{\pr}$.
\end{lemma}
\begin{proof}
Suppose that $\xi\colon\varphi_1\to\varphi$ is a regular epimorphism and that
$\theta\colon\varphi_2\to\varphi$ is any morphism. Then the pullback
$\theta\times_{\varphi}\xi\to\varphi_2$ is a regular epimorphism by
\cref{prop.epis} and the following reasoning:
$\begin{aligned}
\begin{tikzpicture}[unoriented WD]
	\node (P1) {
  \begin{tikzpicture}[inner WD, pack size=6pt]
	  \node[pack] (phi) {$\varphi_2$};
		\draw (phi.west) to +(-2pt, 0);
	\end{tikzpicture}
	};
	\node[right=3 of P1] (P2) {
  \begin{tikzpicture}[inner WD]
		\node[funcr] (theta) {$\theta$};
	  \node[link, right=.5 of theta] (dot) {};
	  \draw (theta.west) -- +(-2pt,0);
	  \draw (theta.east) -- (dot);
  \end{tikzpicture}	
	};
	\node[right=3 of P2] (P3) {
  \begin{tikzpicture}[inner WD]
		\node[funcr] (theta) {$\theta$};
		\node[funcl, right=1 of theta] (theta') {$\xi$};
	  \node[link, right=.5 of theta'] (dot) {};
	  \draw (theta.west) -- +(-2pt,0);
	  \draw (theta.east) -- (theta'.west);
	  \draw (theta'.east) -- (dot);
  \end{tikzpicture}	
	};
	\node at ($(P1.east)!.5!(P2.west)$) (vdash) {$\vdash$};	
	\node at ($(P2.east)!.5!(P3.west)$) (vdash) {$\vdash$};	
\end{tikzpicture}
\end{aligned}$.
\end{proof}

It is now straightforward to observe that $\func{\pr}$ is a regular category.

\begin{proof}[Proof of \cref{thm.internal_functions}]\label{page.proof_thm.internal_functions}
By \cref{lemma.terminal,lemma.pullbacks}, $\func{\pr}$ has all finite limits, and by \cref{lemma.image_fact,lemma.pb_stability}, it has pullback-stable image factorizations.
\end{proof}

\section{Subobject lattices in $\func{\pr}$}

We will find the following characterization of the subobject lattices in $\func{\pr}$ useful.

\begin{proposition}\label{prop.T_sub_RT}
Let $(\typeset, \pr)$ be a regular calculus, let $\Gamma\in\frc$ be a context, and let $s\in \pr(\Gamma)$. There is an isomorphism of posets
\[
  \{t \in \pr(\Gamma) \mid t \leq s\} \cong \sub_{\func{\pr}} (\Gamma, s),
\]
with each element $t \leq s$ mapped to the subobject $\pr(\lsh\delta)(t) =
  \begin{tikzpicture}[inner WD]
    \node[pack, inner sep=0pt] (phi) {$t$};
    \node[link, below=2pt of phi] (dot) {};
    \draw (phi) -- (dot);
    \draw (dot) -- +(-8pt, 0);
    \draw (dot) -- +(8pt, 0);
  \end{tikzpicture}
  \colon (\Gamma,t) \to (\Gamma,s)
$.
\end{proposition}
\begin{proof}
  The proposed map indeed sends each $t$ to a subobject by the characterization of monomorphisms in \cref{cor.monos}. To see that it is surjective, note that given a monomorphism $\theta\colon (\Gamma',s') \to (\Gamma,s)$ in $\func{\pr}$, \cref{lemma.image_fact} (characterizing image factorizations) shows that it is isomorphic to the monomorphism
	\[
		\begin{tikzpicture}[inner WD,baseline=(current  bounding  box.center)]
			\node[link] (dot) {};
			\node[funcd, above=2pt of dot] (theta) {$\theta$};
			\node[link, above=2pt of theta] (dot2) {};
			\draw (dot) -- +(-1,0);
			\draw (theta.south) -- (dot);
			\draw (theta.north) -- (dot2);
			\draw (dot) -- +(1,0);
		\end{tikzpicture}
		\colon
		\big(\Gamma,
		    \resizebox{2em}{!}{
		\begin{tikzpicture}[inner WD]
			\node[funcr] (theta) {$\theta$};
			\node[link, left=5pt of theta] (dot) {};
			\draw (dot) -- (theta);
			\draw (theta.east) -- +(5pt,0);
		\end{tikzpicture}}\big)
		\to (\Gamma,s)
	\]
	where $
		    \resizebox{2em}{!}{
	\begin{tikzpicture}[inner WD]
			\node[funcr] (theta) {$\theta$};
			\node[link, left=5pt of theta] (dot) {};
			\draw (dot) -- (theta);
			\draw (theta.east) -- +(5pt,0);
		\end{tikzpicture}
			}
		=
		\pr(\lsh\epsilon \oplus \Gamma)(\theta)$.

		To see that it is injective, suppose we have a map $\theta$ of monomorphisms
		\[
			\begin{tikzcd}[row sep=.5ex, column sep =6ex]
        (\Gamma,t') \ar[dd, "\theta"'] \ar[dr, "\pr(\lsh\delta)(t')" near start] \\
        & (\Gamma,s) \\
        (\Gamma,t) \ar[ur, "\pr(\lsh\delta)(t)"' near start]
			\end{tikzcd}
    \] 
    Note that this implies that
\[
	\begin{tikzpicture}[unoriented WD, pack size=10pt]
		\node (P0) {$\im\theta \wedge t$};
		\node[right=2 of P0] (P1) {
			\begin{tikzpicture}[inner WD]
				\coordinate (cent);
				\node[link, left=2 of theta] (dot) {};
				\node[funcr, inner sep=2pt, left=.8 of cent] (theta) {$\theta$};
				\node[link, right=.8 of cent] (dot2) {};
				\node[pack, above=2pt of dot2, inner sep=1pt] (phi2) {$t$};
				\draw (theta.east) -- (dot2);
				\draw (theta) -- (phi);
				\draw (dot2) -- (phi2);
				\draw (theta) -- (dot);
				\draw (dot2) -- +(.5cm, 0);
			\end{tikzpicture}
		};
		\node[right=3 of P1] (P2) {
			\begin{tikzpicture}[inner WD]
				\node[pack, inner sep=1pt] (phi) {$t'$};
				\node[link, below=1pt of phi] (dot) {};
				\node[link, left=3pt of dot] (dot0) {};
				\draw (phi) -- (dot);
				\draw (dot) -- (dot0);
				\draw (dot) -- +(7pt, 0);
			\end{tikzpicture}
		};
		\node[right=2 of P2] (P3) {$t'$};
		\node at ($(P0.east)!.5!(P1.west)$) {$=$};
		\node at ($(P1.east)!.5!(P2.west)$) {$=$};
		\node at ($(P2.east)!.5!(P3.west)$) {$=$};
	\end{tikzpicture}
\]
and hence that $t' \le t \in \pr(\Gamma)$. Thus the subobjects $(\Gamma,t)$ and $(\Gamma,t')$ of $(\Gamma,s)$ are isomorphic if and only if $t = t'$. This proves the proposition.
\end{proof}

Our main theorem is to prove an adjunction between regular calculi and regular categories, and we will get to this in the next section. To round out the picture, however, we quickly record that the \emph{po-category} $\rela{\pr}$ of internal relations in a regular calculus is also regular: it is the relations po-category of $\func{\pr}$.

\begin{corollary} \label{thm.syn_is_regular}
 Let $(\typeset, \pr)$ be a regular calculus. Then $\rela{\pr}$ is isomorphic to the po-category of relations in $\func{\pr}$. In particular, $\rela{\pr}$ is a regular po-category.
\end{corollary}
\begin{proof}
  Observe that $\rela{\pr}$ and $\func{\pr}$ have the same set of objects by definition, and that by \cref{prop.T_sub_RT} for any two objects $(\Gamma,s)$, $(\Gamma',s')$ the poset of relations $(\Gamma,s) \tickar (\Gamma',s')$ in $\func{\pr}$ is given by $\{\theta \in \pr(\Gamma\oplus\Gamma') \mid \theta \le s \boxplus s'\}$. It remains to prove that the composition rule in $\rela{\pr}$ agrees with composition of relations in $\func{\pr}$. Reasoning using graphical terms, this is a straightforward consequence of \cref{lemma.pullbacks}, which describes pullbacks in the category $\func{\pr}$. 
\end{proof}

\chapter{$\rgcat$ is essentially a reflective subcategory of $\rgcalc$}\label{chap.ess_refl}

We have now proved \cref{thm.internal_functions}, which constructs a regular category $\func{\pr}$ from any regular calculus $(\typeset, \pr)$. We call $\func{\pr}$ the \emph{syntactic category} corresponding to $\pr$. In this section we show that this construction is functorial, and that there is an adjunction
\[
	\adj{\rgcalc}{\syn}{\prd}{\rgcat.}
\]
Moreover, $\rgcat$ is essentially a reflective subcategory of $\rgcalc$, in the sense that for any regular category $\cat{R}$, the counit map $\syn(\prd(\cat{R}))\to\cat{R}$ is an equivalence of categories. In future work we plan to show that there is 2-dimensional structure throughout, such that the above adjunction extends to a 2-adjunction in which $\rgcat$ is 2-reflective.

\section{The functor $\syn\colon\rgcalc\to\rgcat$}

We want to define a functor $\syn\colon\rgcalc\to\rgcat$ that is adjoint to $\prd$ from \cref{prop.rels}. On objects, this is
now easy: given a regular calculus $(\typeset, \pr)\in\rgcalc$, define
$\syn(\typeset, \pr)\coloneqq\func{\pr}$ as in \cref{eqn.RT}; objects are pairs
$(\Gamma,\varphi)$ where $\Gamma\in\frc$ and $\varphi\in \pr(\Gamma)$,
and morphisms are internal functions $\theta$ as in \cref{thm.characterize_functions}. 

For morphisms, suppose given $(F,F^\sharp)\colon(\typeset, \pr)\to (\typeset',\pr')$:
\[
\begin{tikzcd}[row sep=5pt]
  \typeset\ar[dd, "F"']&
  \frb[\typeset]\ar[dd, "\ol{F}"']\ar[dr, bend left=15pt, "\pr", ""' name={T}]\\
  &&[20pt]\pposet\\
  \typeset'&
  \frb[\typeset']\ar[ur, bend right=15pt, "\pr'"', "" name={T'}]
  \ar[from=T, to={T'-|T}, twocell, "F^\sharp"']
\end{tikzcd}
\]
where again $\ol{F}\coloneqq\frb[F]$. We define $\funr{F}\coloneqq\syn(F,F^\sharp)\colon\func{\pr}\to\func{\pr'}$
on an object $(\Gamma,\varphi)\in\func{\pr}$ by
\begin{equation}\label{eqn.F_objects}
	\funr{F}(\Gamma,\varphi)\coloneqq \left(\ol{F}(\Gamma), F^\sharp_\Gamma(\varphi)\right)\in\func{\pr'}.
\end{equation}
and on a morphism $\theta\colon(\Gamma_1,\varphi_1)\to (\Gamma_2,\varphi_2)$ by
\begin{equation}\label{eqn.F_morphisms}
	\funr{F}(\theta)\coloneqq F^\sharp_{\Gamma_1\oplus\Gamma_2}(\theta).
\end{equation}

\begin{theorem} \label{thm.syn_def}
The assignment $\syn(\typeset, \pr)\coloneqq\func{\pr}$ on objects, and \cref{eqn.F_objects,eqn.F_morphisms} on morphisms, constitutes a functor $\syn\colon\rgcalc\to\rgcat$.
\end{theorem}

\Cref{thm.syn_def} is proved on page~\pageref{page.proof_thm.syn_def}; first we need the following lemma.
\begin{lemma} \label{lem.funr_preserves_diagrams}
  $F^\sharp$ preserves semantics of graphical terms. 
  
  More precisely, given any $\pr$-graphical term $(\theta_1,\dots,\theta_k; \omega)$,
  the morphism $(F,F^\sharp)$ induces a $\pr'$-graphical term
  $(F^\sharp\theta_1,\dots,F^\sharp\theta_k; \overline{F}(\omega))$; we call this its
  \define{image} under $F^\sharp$. The image obeys
  \[
  F^\sharp\church{(\theta_1,\dots,\theta_k; \omega)} = 
  \church{(F^\sharp\theta_1,\dots,F^\sharp\theta_k; \overline{F}(\omega))}.
  \]
  Furthermore, given the entailment $(\theta_1,\dots,\theta_k; \omega) \vdash  (\theta_1',\dots,\theta_{k'}'; \omega')$, it follows that
    \[
    (F^\sharp\theta_1,\dots,F^\sharp\theta_k; \overline{F}(\omega)) \vdash
  (F^\sharp\theta_1',\dots,F^\sharp\theta_{k'}'; \overline{F}(\omega')).
\]
\end{lemma}
\begin{proof}
  The naturality and monoidality of $(F,F^\sharp)$ imply:
  \begin{align*}
    F^\sharp\church{(\theta_1,\dots,\theta_k; \omega)}
    &= F^\sharp(\pr(\omega))(\rho(\theta_1,\dots,\theta_k))) \\ 
    &= \pr'(\overline{F}(\omega))(F^\sharp(\rho(\theta_1,\dots,\theta_k))) \\ 
    &= \pr'(\overline{F}(\omega))(\rho(F^\sharp\theta_1,\dots,F^\sharp\theta_k)) \\ 
    &= \church{(F^\sharp\theta_1,\dots,F^\sharp\theta_k; \overline{F}(\omega))}.
  \end{align*}
  The second claim then follows from the monotonicity of components in $F^\sharp$.
\end{proof}

\begin{proof}[Proof of \cref{thm.syn_def}]\label{page.proof_thm.syn_def}
  First we must check that our data type-checks. We have already shown that $\func{\pr}$
  is a regular category, so it remains to show that $\funr{F}$ is a regular
  functor. This is a consequence of \cref{lem.funr_preserves_diagrams}.

  In particular, recall from \cref{def.RT} that morphisms
  in $\func{\pr}$ can be represented by $\pr$-graphical terms obeying certain
  entailments. It was shown in \cref{sec.finlims,sec.images} that composition, identities, finite limits, and regular epis
  can also be described in this way.  \cref{lem.funr_preserves_diagrams} implies
  that given a $\pr$-graphical term, its image under $F^\sharp$ preserves
  entailments and equalities. Thus $\funr{F}$ sends
  internal functions to internal functions of the required domain and codomain,
  preserves composition, identities, finite limits, and regular epis, and hence
  is a regular functor.

  It is then immediate from the definition
  (\cref{eqn.F_objects,eqn.F_morphisms}) that $\syn$ preserves identity
  morphisms and composition, and so $\syn$ is indeed a functor.
\end{proof}

\section{The essential reflection}

Recall that $\prd(\cat{R})=(\ob\cat{R},\sub_{\cat{R}}\Prod)$ and $\syn(\pr)=\ladj(\rela{\pr})$; see \cref{eqn.rels_on_objects,thm.internal_relations}.

\begin{proposition}\label{prop.essential_reflection}
For any regular category $\cat{R}$, there is a natural equivalence of categories
\[\epsilon\colon\syn(\prd(\cat{R}))\To{\simeq}\cat{R}.\]
\end{proposition}
\begin{proof}
We will define functors $\epsilon\colon\func{\prd(\cat{R})}\tofrom\cat{R}\cocolon\epsilon'$ and show that they constitute an equivalence. We have $\ob(\func{\prd(\cat{R})})=\{(\Gamma,r)\mid \Gamma\in\frc[\ob\cat{R}], \;r\in \sub_{\cat{R}}\Prod[\Gamma]\}$, so  put
\[
  \epsilon(\Gamma,r)\coloneqq r,
  \qqand  \epsilon'(r)\coloneqq (\unary{r},r),
\]
where $\unary{r}$ is the unary context on $r$ and $r\ss r=\Prod[\unary{r}]$ is the top element. Given also $(\Gamma',r')$, we have an isomorphism of hom-sets
\[
\func{\prd(\cat{R})}\left((\Gamma,r),(\Gamma',r')\right) \cong \ladj(\rrel{\cat{R}})(r,r') \cong \cat{R}(r,r'),
\]
by \cref{def.RT,prop.rela_rels_rrel,lemma.fundamental}. Hence, we define $\epsilon$ and $\epsilon'$ on morphisms to be the corresponding mutually-inverse maps. Obviously, $\epsilon$ and $\epsilon'$ are fully faithful functors, and $\epsilon'\cp \epsilon=\id_{\cat{R}}$, so $\epsilon$ is essentially surjective.
\end{proof}

We next prove that $\prd\colon\rgcat\to\rgcalc$ is full, fulfilling a promise made after \cref{prop.rels}, where $\prd$ was first defined. Recall that $\prd(\cat{R})=(\ob(R),\sub_{\cat{R}}\Prod[-])$.

\begin{corollary}\label{cor.rels_full}
The functor $\prd\colon\rgcat\to\rgcalc$ is full.
\end{corollary}
\begin{proof}
Let $\cat{R},\cat{R}'$ be regular categories, and suppose given a map
$(F,F^\sharp)\colon\prd(\cat{R})\to\prd(\cat{R}')$; we need to show there exists
a functor $\funr{F}\colon\cat{R}\to\cat{R}'$ such that
$\prd(\funr{F})=(F,F^\sharp)$. The key idea is that $(F,F^\sharp)$ specifies the
action of the desired functor $\funr{F}$ on subobject semilattices, which is
enough, since every morphism in $\cat{R}$ can be recovered from its graph.

Applying $\syn$ to $(F,F^\sharp)$, we obtain a regular functor
$\syn(F,F^\sharp)\colon \syn(\prd(\cat{R}')) \to \syn(\prd(\cat{R}))$. Pre- and
post-composing this with the equivalences $\epsilon'_R\colon
\cat{R}\to\syn(\prd(\cat{R}))$ and $\epsilon_{\cat{R}'}\colon
\syn(\prd(\cat{R'}))\to\cat{R}'$ from \cref{prop.essential_reflection}, we
obtain a regular functor $\funr{F}\colon\cat{R} \to \cat{R'}$. It is routine to
check that the image of this functor is $\prd(\funr{F})=(F,F^\sharp)$. 
\end{proof}

\begin{theorem}\label{thm.main}
The functors $\prd$ and $\syn$ are adjoint:
\[
\adj{\rgcalc}{\syn}{\prd}{\rgcat.}
\]
Moreover, $\prd$ is fully faithful, and for any regular category $\cat{R}$, the counit map $\syn(\prd(\cat{R}))\to\cat{R}$ is an equivalence.
\end{theorem}
\begin{proof}
We showed that $\prd$ is fully faithful in \cref{prop.rels,cor.rels_full} and that there is a natural transformation $\epsilon\colon\prd\cp\syn\to\id_{\rgcat}$ with the property that $\epsilon_{\cat{R}}$ is an equivalence for any $\cat{R}$. It remains to construct $\eta\colon\id_\rgcalc\to\syn\cp\prd$ and check that $\epsilon$ and $\eta$ satisfy the triangle identities.

Given a regular calculus $(\typeset, \pr)$, we have $\prd(\syn(\typeset, \pr))=(\ob\func{\pr},\sub_{\func{\pr}}\Prod)$, where $\ob\func{\pr}=
	\{(\Gamma,\varphi)\mid\Gamma\in\frc,\varphi\in \pr(\Gamma)\}$. There is an obvious function $e\colon\typeset\to\ob\func{\pr}$ sending $\tau \mapsto (\unary{\tau}, \true)$, where as usual, $\unary{\tau}$ is the unary context and $\true\in \pr(\unary{\tau})$ is its top element. We will define $\eta\coloneqq(e,e^\sharp)$, where $e^\sharp(\Gamma)\colon \pr(\Gamma)\to\sub_{\func{\pr}}\Prod[\ol{e}(\Gamma)]=\sub_{\func{\pr}}(\Gamma,\true)$ is the natural isomorphism given in \cref{prop.T_sub_RT}:
\[
\begin{tikzcd}[row sep=5pt]
  \typeset\ar[dd, "e"']&
  \frb[\typeset]\ar[dd, "\ol{e}"']\ar[dr, bend left=15pt, "\pr", ""' name={T}]\\
  &&[20pt]\pposet\\
  \ob\func{\pr}&
  \frb[\ob\func{\pr}]\ar[ur, bend right=15pt, "\sub_{\func{\pr}}\Prod"', "" name={T'}]
  \ar[from=T, to={T'-|T}, twocell, "e^\sharp"']
\end{tikzcd}
\]
The fact that $\epsilon_{\cat{R}}$ is an equivalence and that $e^\sharp$ is a natural isomorphism make the triangle identities particularly easy (if tedious) to verify. This completes the proof.
\end{proof}

\section*{Acknowledgments}

The authors thank Paolo Perrone for comments that have improved this article.

\printbibliography
\end{document}